\documentclass[reqno,10pt]{amsart}
\usepackage[english]{babel}
\usepackage{amsmath}

\usepackage{amssymb}
\usepackage[scr=boondoxo]{mathalfa}
\usepackage{graphicx}
\usepackage{epstopdf}

\usepackage{subfigure}
\pdfoutput=1

\newcommand{\bpr}{\begin{trivlist} \item[]{\bf Proof. }}
\newcommand{\epr}{\hspace*{\fill} $\qed$\end{trivlist}}

\newcommand{\be}{\begin{eqnarray}}
\newcommand{\ee}{\end{eqnarray}}
\newcommand{\ba}{\begin{align}}
\newcommand{\ea}{\end{align}}
\newcommand{\bi}{\begin{itemize}}
\newcommand{\ei}{\end{itemize}}

\newcommand{\secref}[1]{Section~\ref{sec:#1}}

\newcommand{\seclab}[1]{\label{sec:#1}}
\newcommand{\eqlab}[1]{\label{eq:#1}}
\renewcommand{\eqref}[1]{(\ref{eq:#1})}

\newcommand{\figref}[1]{Fig.~\ref{fig:#1}}
\newcommand{\figlab}[1]{\label{fig:#1}}

\newcommand{\lemmaref}[1]{Lemma~\ref{lemma:#1}}
\newcommand{\lemmalab}[1]{\label{lemma:#1}}
\newcommand{\remref}[1]{Remark~\ref{remark:#1}}
\newcommand{\remlab}[1]{\label{remark:#1}}
\newcommand{\corref}[1]{Corollary~\ref{cor:#1}}
\newcommand{\corlab}[1]{\label{cor:#1}}
\newcommand{\thmref}[1]{Theorem~\ref{theorem:#1}}
\newcommand{\thmlab}[1]{\label{theorem:#1}}
\newcommand{\tablab}[1]{\label{tab:#1}}
\newcommand{\tabref}[1]{Table~\ref{tab:#1}}
\newcommand{\defnlab}[1]{\label{defn:#1}}
\newcommand{\defnref}[1]{Definition~\ref{defn:#1}}
\newcommand{\appref}[1]{Appendix~\ref{app:#1}}

\newcommand{\applab}[1]{\label{app:#1}}

 \usepackage{array}
 \usepackage{floatrow}
\DeclareFloatFont{tiny}{\tiny}
\usepackage{xcolor,colortbl}
\definecolor{green}{rgb}{0.8,0.8,0.8}
\newcommand{\done}{\cellcolor{green}}  

\newtheorem{theorem}{Theorem}[section]
\newtheorem{proposition}[theorem]{Proposition}
\newtheorem{definition}[theorem]{Definition}
\newtheorem{lemma}[theorem]{Lemma}
\newtheorem{cor}[theorem]{Corollary}

\newtheorem{remark}[theorem]{Remark}

\numberwithin{equation}{section}
\usepackage{pdfpages}

\usepackage{tikz}
\usepackage{rotating}
\begin{document}

\title[A new type of relaxation oscillationin a model with R\&S friction]{A new type of relaxation oscillation in a model with rate-and-state friction}

\author {K. Uldall Kristiansen} 
\date\today
\maketitle

\vspace* {-2em}
\begin{center}
\begin{tabular}{c}
Department of Applied Mathematics and Computer Science, \\
Technical University of Denmark, \\
2800 Kgs. Lyngby, \\
DK
\end{tabular}
\end{center}

 \begin{abstract}
 In this paper we prove the existence of a new type of relaxation oscillation occurring in a one-block Burridge-Knopoff model with Ruina rate-and-state friction law. In the relevant parameter regime, the system is slow-fast with two slow variables and one fast. The oscillation is special for several reasons: Firstly, its singular limit is unbounded, the amplitude of the cycle growing like $\log \epsilon^{-1}$ as $\epsilon\rightarrow 0$. As a consequence of this estimate, the unboundedness of the cycle cannot be captured by a simple $\epsilon$-dependent scaling of the variables, see e.g. \cite{Gucwa2009783}. We therefore obtain its limit on the Poincar\'e sphere. Here we find that the singular limit consists of a slow part on an attracting critical manifold, and a fast part on the equator (i.e. at $\infty$) of the Poincar\'e sphere, which includes motion along a center manifold. The reduced flow on this center manifold runs out along the manifold's boundary, in a special way, leading to a complex return to the slow manifold. We prove the existence of the limit cycle by showing that a return map is a contraction. The main technical difficulty in this part is due to the fact that the critical manifold loses hyperbolicity at an exponential rate at infinity. We therefore use the method in \cite{kristiansen2017a}, applying the standard blowup technique in an extended phase space. In this way we identify a singular cycle, consisting of $12$ pieces, all with desirable hyperbolicity properties, that enables the perturbation into an actual limit cycle for $0<\epsilon\ll 1$.  
 The result proves a conjecture in \cite{bossolini2017a}. The reference \cite{bossolini2017a} also includes a priliminary analysis based on the approach in \cite{kristiansen2017a} but several details were missing. We provide all the details in the present manuscript and lay out the geometry of the problem, detailing all of the many blowup steps. 
 
%
%
 \end{abstract}

\section{Introduction}
Relaxation oscillations are special periodic solutions of singularly perturbed ordinary differential equations. They consist of long periods of ``in-activity'' interspersed with short periods of rapid transitions. Mathematically, they are classically defined for slow-fast systems
\begin{align}
 \epsilon \dot x &= f(x,y,\epsilon),\eqlab{slowfastxy}\\
 \dot y &=g(x,y,\epsilon),\nonumber
\end{align}
as elements $\Gamma_\epsilon$ of a family of periodic orbits $\{\Gamma_\epsilon\vert \epsilon\in (0,\epsilon_0]\}$ whose $\epsilon\rightarrow 0$ limit (in the Hausdorff sense), $\Gamma_0$, is a closed loop consisting of a union of a) \textit{slow} orbits of the reduced problem:
\begin{align*}
 0 &=f(x,y,0),\\
 \dot y &=g(x,y,0),
\end{align*}
and b) fast orbits of the layer problem:
\begin{align*}
 x' &=f(x,y,0),\\
 y' &=0.
\end{align*}
Here $()'=\frac{d}{d\tau}$ and $\dot{()}=\frac{d}{dt}$ are related for $\epsilon>0$ by 
\begin{align*}
 \tau = \epsilon^{-1} t.
\end{align*}
$\tau$ is called the fast time whereas $t$ is called the slow time. Obviously, $\Gamma_0$ should allow for a consisting orientation of positive (slow and fast) time. $\Gamma_0$ is in this case called a singular cycle. 

The prototypical system, where relaxation oscillations occur, is the van der Pol system, see e.g. \cite{krupa_relaxation_2001}. Here the critical manifold $C=\{(x,y)\vert f(x,y,0)=0\}$ is \begin{tikzpicture}
\draw [red,  thick,domain=-0.08:0.08] plot (\x, 0.08333333333-3.125000000*\x+651.0416667*\x^3);
\end{tikzpicture}-shaped and relaxation oscillations $\Gamma_\epsilon$ occur, in generic situations, near a $\Gamma_0$ consisting of the leftmost and rightmost pieces of the \begin{tikzpicture}
\draw [red,  thick,domain=-0.08:0.08] plot (\x, 0.08333333333-3.125000000*\x+651.0416667*\x^3);
\end{tikzpicture}-shaped critical manifold $C$ interspersed by two horizontal lines connecting these branches at the ``folds''. See \figref{relax}(a).

But other types of relaxation oscillations also exist. The simplest examples appear in slow-fast systems in nonstandard form
\begin{align}
 \dot z &=h(z,\epsilon),\eqlab{slowfastz}
\end{align}
where $C=\{z\vert h(z,0)=0\}$ is a critical manifold. Here relaxation oscillations may even be the union of one single fast orbit and a single slow orbit on $S$. See \figref{relax}(b). In \cite{Gucwa2009783}, for example, a planar slow-fast system of the form \eqref{slowfastxy} is considered. Here limit cycles $\Gamma_\epsilon$ exist which also have segments that follow the different time scales, $t$ and $\tau$. But $\Gamma_\epsilon$ grows unboundedly as $\epsilon\rightarrow 0^+$ and the limit $\Gamma_0$ is therefore not a cycle. However, in the model considered by \cite{Gucwa2009783} there exists a scaling of the variables that capture the unboundedness and in these scaled variables the system is transformed into a system of nonstandard form \eqref{slowfastz}. For this system, $\Gamma_0$ becomes a closed cycle, albeit with some degeneracy along a critical manifold. Similar or related relaxation oscillations occur in \cite{bro2,kosiuk2015a,kuehn2015a}.

\begin{figure}[h!]
\begin{center}
\subfigure[]{\includegraphics[width=.495\textwidth]{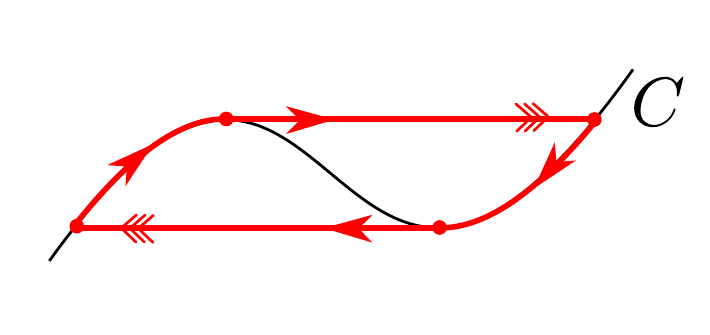}}
\subfigure[]{\includegraphics[width=.495\textwidth]{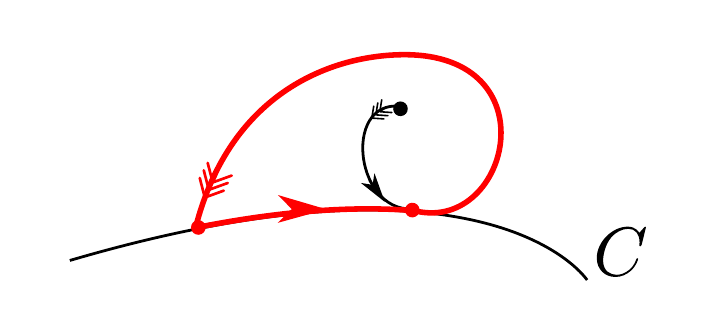}}
\end{center}
\caption{In (a): The prototypical example of a relaxation oscillation in a planar slow-fast system with a folded critical manifold. In (b): Example of a relaxation oscillation in slow-fast system in nonstandard form.}
\figlab{relax}
\end{figure} 

In \cite{kristiansen2019}, see also \cite[Fig. 2(c)]{Tyson2003}, yet another type of relaxation oscillation is described for a planar system which is not slow-fast but still singularly perturbed, like
\begin{align*}
 \dot z = h(z,\epsilon^{-1}).
\end{align*}
In this particular case the singular cycle $\Gamma_0$ is 
\begin{tikzpicture}
\draw [red, thick] (0,-1) circle [radius=0.1];;
\end{tikzpicture}-shaped but it crosses the set, where $$\lim_{\epsilon \rightarrow 0} h(z,\epsilon^{-1}),$$ is undefined, twice, complicating the analysis significantly. As a result, the conditions ensuring that $\Gamma_\epsilon$ exists are also fairly complicated. Related relaxation oscillations occur in similar systems, see \cite{krihog2,kristiansen2018a}.

In the present paper, we will consider the following slow-fast system
\begin{align}
 \dot x &=- e^z \left(x+(1+\alpha)z\right),\eqlab{system}\\
 \dot y &=e^z -1,\nonumber\\
 \epsilon \dot z &=-e^{-z} \left(y+\frac{x+z}{\xi}\right).\nonumber
\end{align}
Here $\alpha>0,\xi>0$ and $0<\epsilon\ll 1$. This is a caricature model of an earthquake fault, see \secref{model} below. Relaxation oscillations in this system therefore models the seismic cycle of earthquakes with years, decades even, of inactivity preceded by sudden dramatic shaking of the ground: the earthquake.

Similar to the case in \cite{Gucwa2009783}, limit cycles of \eqref{system} also grow unboundedly as $\epsilon\rightarrow 0$. But in contrary to \cite{Gucwa2009783}, the right hand side of \eqref{system} does not have polynomial growth, and as a result, the unboundedness of the solutions cannot be captured by a scaling of the variables. As a result, we will in this paper work on the Poincar\'e sphere. Here we then prove the existence of limit cycles $\Gamma_\epsilon$, whose limit $\Gamma_0$ as $\epsilon\rightarrow 0$ consists of a single slow orbit on the $2D$ attracting critical manifold $C=\{(x,y,z)\vert y+(x+z)/\xi=0\}$. The ``fast'' part of $\Gamma_0$ occurs at ``infinity'' (i.e. the equator of the Poincar\'e sphere) and is nontrivial and perhaps even surprising. We uncover this structure by applying the method in \cite{kristiansen2017a} to gain hyperbolicity where this is lost due to exponential decay of eigenvalues. The main theorem, \thmref{mainThm}, proves a conjecture in \cite{bossolini2017a}. 

%
%


\subsection{The model}\seclab{model}
The model we consider, described by the equations \eqref{system}, consists of a single block dragged along a frictional surface by a spring, the end of which moves at a constant velocity. We set this velocity to $1$, without loss of generality. See an illustration in \figref{spring}. Here $v$ is the velocity of the block and $y$ is the relative position, measuring the deformation of the spring. If the moving spring models a sliding fault, then the system becomes a caricature model of an earthquake fault. It is therefore also the extreme case of a single-block version of the Burridge-Knopoff model, which idealizes the earthquake fault as a chain of spring-block systems of the type shown in \figref{spring}. More importantly, the Burridge-Knopoff model has a continuum limit as the distance between the chain blocks vanishes and travelling wave solutions of the resulting PDE system, see \cite[Section 2.1]{bossolini2017a}, are basically solutions of the one-block system. See \cite{putelat08} for a different derivation.

\begin{figure}[h!]
\begin{center}
{\includegraphics[width=.575\textwidth]{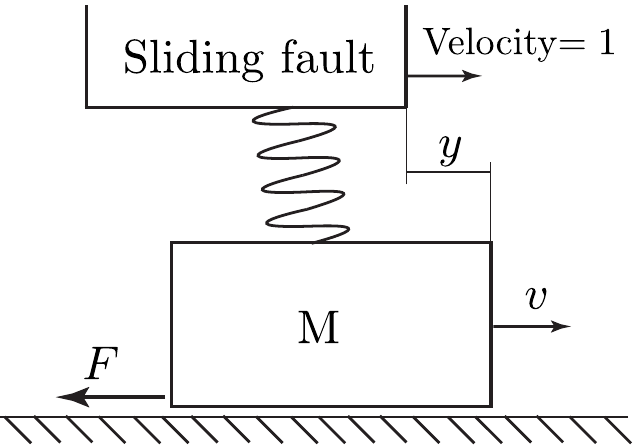}}
\end{center}
\caption{Illustration of model \eqref{system0}.}
\figlab{spring}
\end{figure} 

\subsection{Friction}
The unknown in \figref{spring}, and in earthquake modelling in general, is the friction force $F$. 
Within engineering, friction is frequently modelled using Coulomb's law, the stiction law or the Stribeck law \cite{olsson1998a,feeny1998a}. However, these laws do not account for any of the microscopic processes that are known to occur when surfaces interact in relative motion. Also such models cannot produce phenomena known to occur in earthquakes. 
To capture this, one can use \textit{rate-and-state friction laws}. Such models attempt to account for additional physics, like the condition of the contacting asperities \cite{woodhouse2015a}, by adding additional variables, called ``state variables'', to the problem. The first models of this kind, the Dieterich law \cite{dieterich1978a,dieterich1979a} and the Ruina law \cite{ruina1983a}, were obtained from experiments on rocks. In contrast to e.g. Coulomb's simple model, the friction force  in these models depends logarithmically on the velocity. (It was only later realized that this decay actually agree with theory of Arrhenius processes resulting from breaking bonds at the atomic level \cite{rice2001a}.) 

In this paper, we consider the Ruina friction law. This produces the following equations for the system in \figref{spring}
\begin{align}
 \dot x &= -v \left(x+(1+\alpha)\log v\right),\eqlab{system0}\\
 \dot y &=v-1,\nonumber\\
 \epsilon \dot v &=-y-\frac{x+\log v}{\xi},\nonumber
\end{align}
in its nondimensionalised form. See \cite{bossolini2017a,erickson2008a} for further details on the derivation. The variable $x$ is a single ``state variable''. As in \cite{bossolini2017a} we put $z=\log v$  and arrive at model \eqref{system}, which we shall study in this manuscript as a singular perturbed problem with $0<\epsilon\ll 1$. 

Numerically, existence of relaxation-type oscillations for $\alpha>\xi$ and small values of $\epsilon>0$ is a well-known fact. See also \figref{xyz}, computed in MATLAB using ode23s with tolerances $10^{-12}$. \figref{xyzt} shows $x$, $y$ and $z$ as functions of $t$. But in this paper, we are interested in a rigorous proof of this existence and en-route on how to apply classical methods of singular perturbation theory to \eqref{system0}, or equivalently \eqref{system}, with non-polynomial growth of the right hand side. 

\begin{figure}[h!]
\begin{center}
\subfigure[]{\includegraphics[width=.495\textwidth]{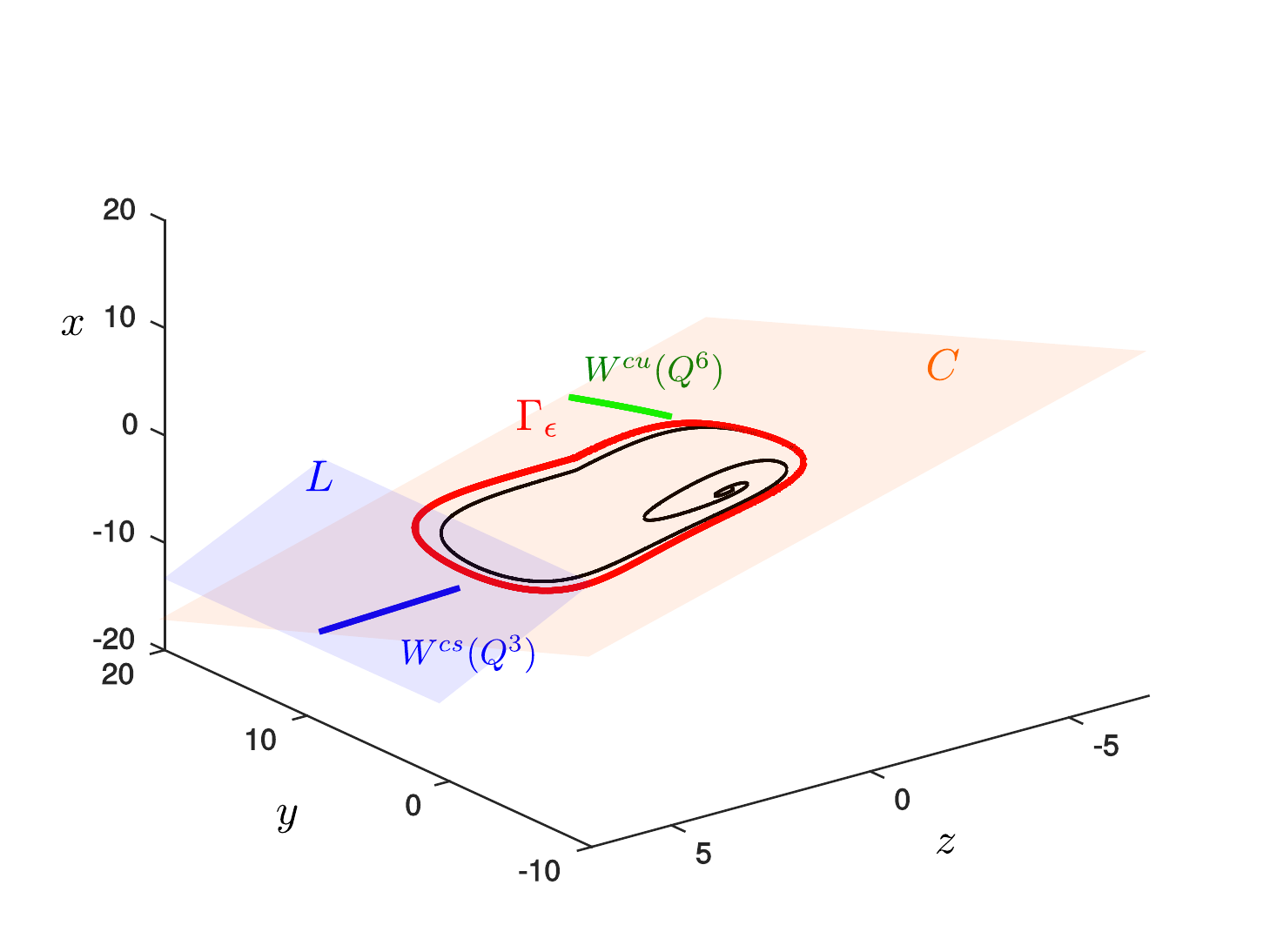}}
\subfigure[]{\includegraphics[width=.495\textwidth]{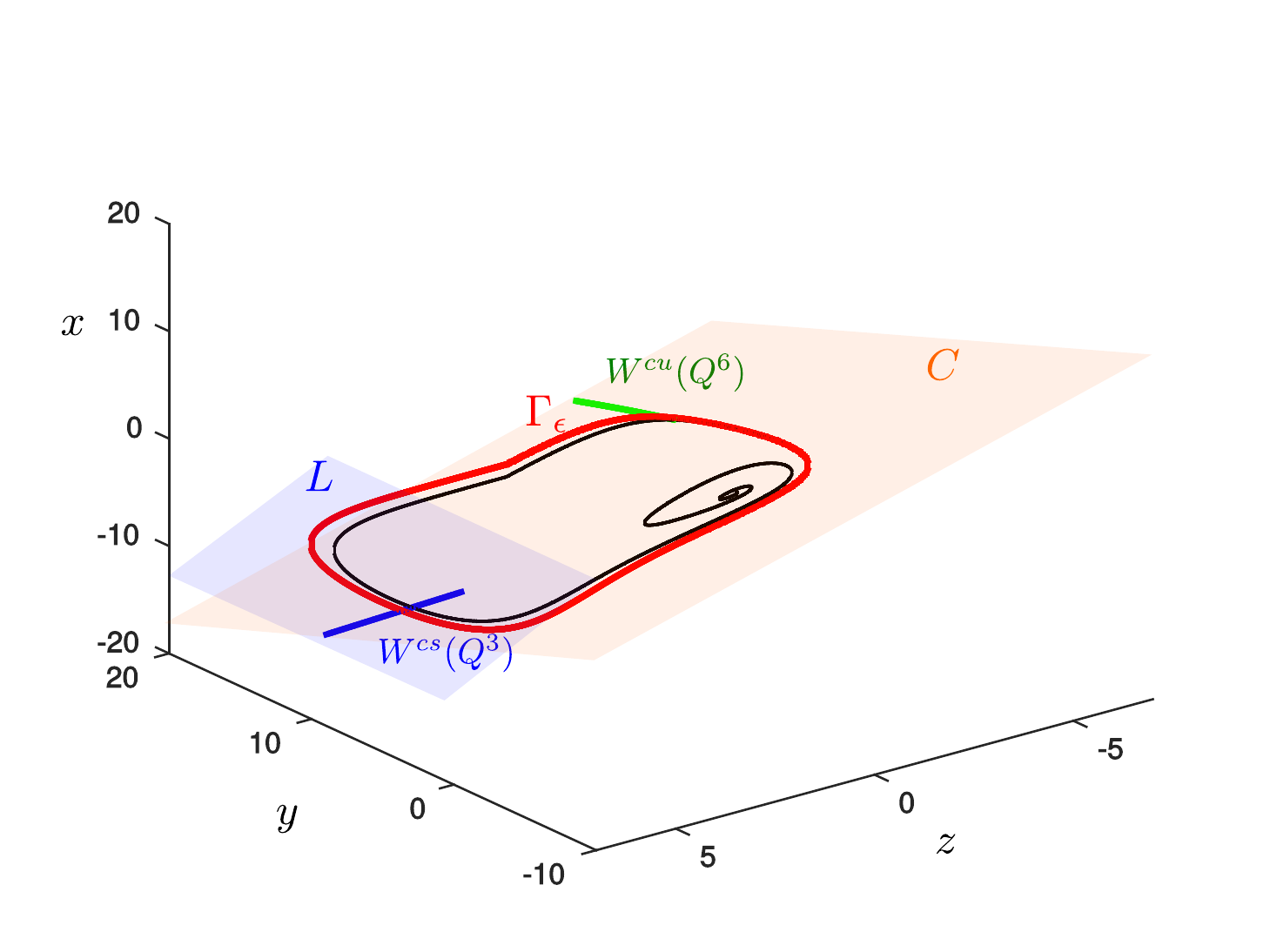}}
\end{center}
\caption{Limit cycles in red for $\alpha=0.8$, $\xi=0.5$. Also in (a): $\epsilon=0.01$ and in (b): $\epsilon=0.001$. The objects $W^{cu}(Q^6)$, $C$ and $L$ are central to the analysis. The motion is clockwise.} 
\figlab{xyz}
\end{figure}
\begin{figure}[h!]
\begin{center}
\subfigure[]{\includegraphics[width=.495\textwidth]{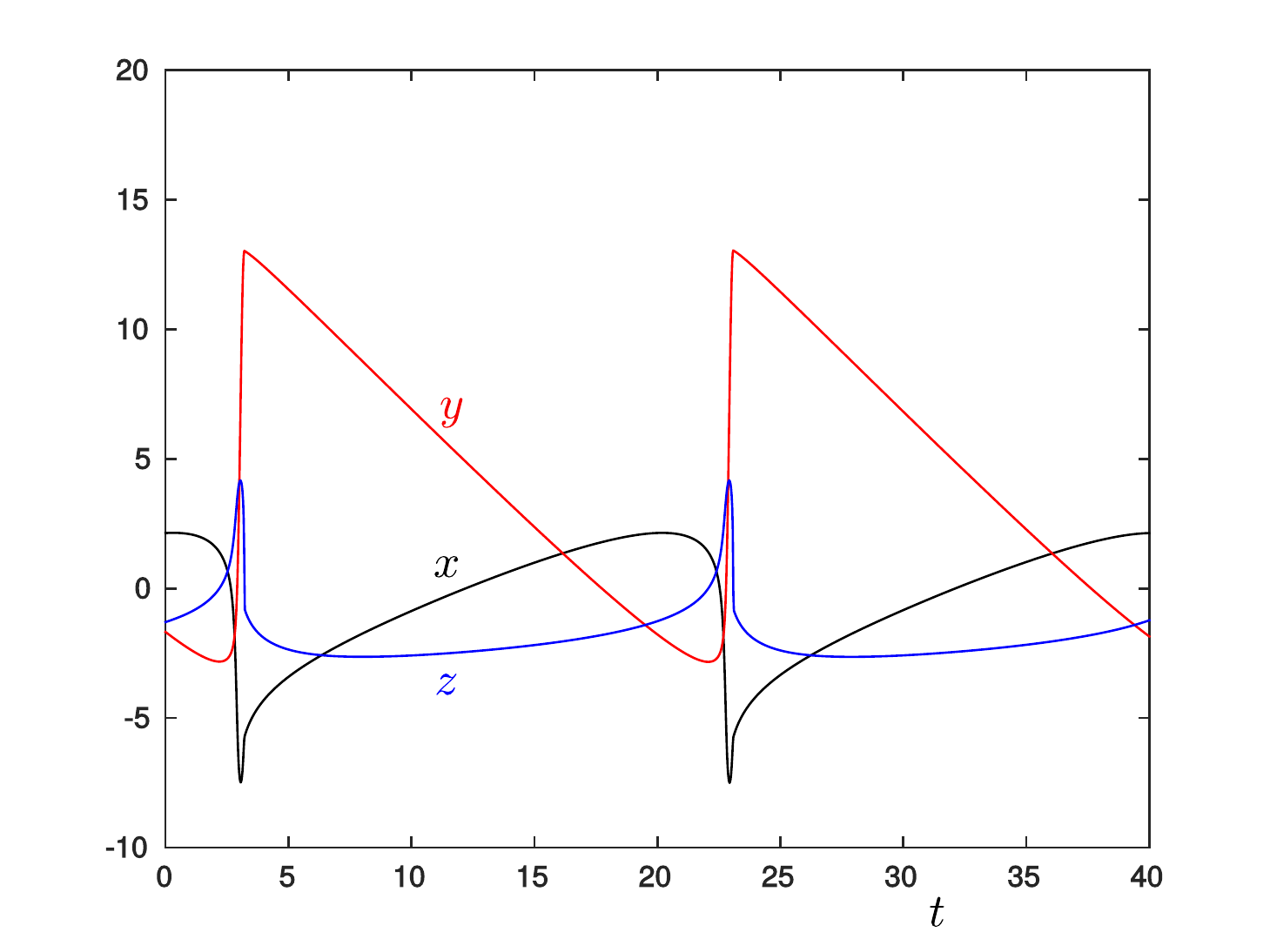}}
\subfigure[]{\includegraphics[width=.495\textwidth]{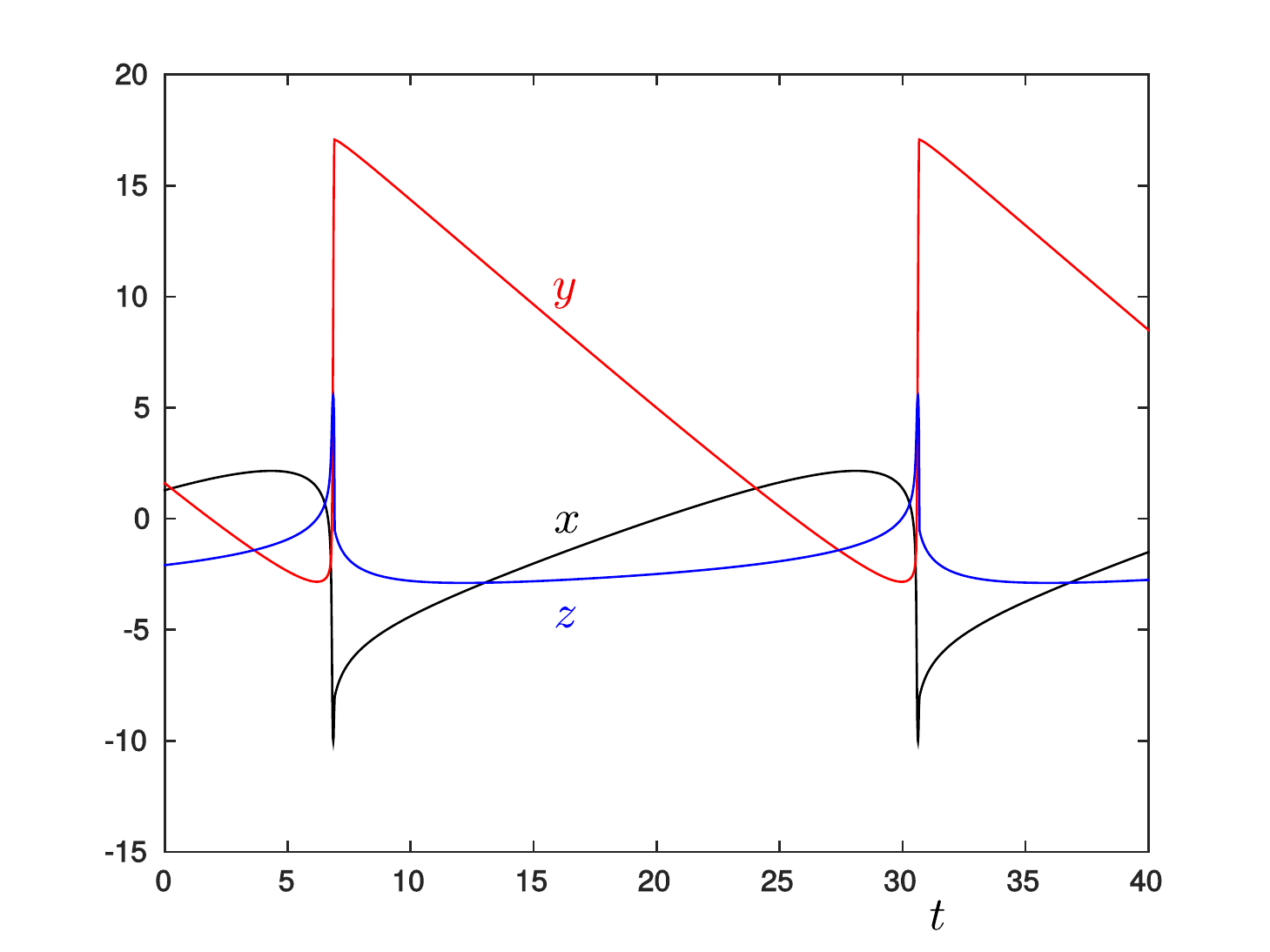}}
\end{center}
\caption{$x$, $y$ and $z$ in \figref{xyz} as functions of $t$ for $\xi=0.5$ and $\alpha=0.8$. In (a): $\epsilon=0.01$. In (b): $\epsilon=0.001$. Notice how the period also depends upon $\epsilon$. } 
\figlab{xyzt}
\end{figure}

The analysis of the Dieterich law
\begin{align*}
 \dot x&=(1+\alpha)(e^{-x/(1+\alpha)}-e^z),\\
 \dot y &=e^z-1,\\
 \epsilon \dot z &=-e^{-z} \left(y+\frac{x+z}{\xi}\right).
\end{align*}
is similar,but slightly more involved, and is therefore postponed to a separate manuscript. Nevertheless, there are known limitations of the Dieterich and Ruina laws. Basically, experiments suggest that friction should be an \begin{tikzpicture}
\draw [red,  thick,domain=-0.08:0.08] plot (\x, 0.08333333333-3.125000000*\x+651.0416667*\x^3);
\end{tikzpicture}-shaped graph of velocity (when the states are in ``quasi-steady states''). Dieterich, for example, only capture the downward diagonal of the \begin{tikzpicture}
\draw [red,  thick,domain=-0.08:0.08] plot (\x, 0.08333333333-3.125000000*\x+651.0416667*\x^3);
\end{tikzpicture}-shaped, see \cite[Fig. 1]{putelat2015a}.

The more recently developed \textit{spinodal rate-and-state friction law}, see \cite{putelat2015a} and references therein, has been developed to capture the missing near-vertical and increasing pieces of the \begin{tikzpicture}
\draw [red,  thick,domain=-0.08:0.08] plot (\x, 0.08333333333-3.125000000*\x+651.0416667*\x^3);
\end{tikzpicture}-profile, producing a potentially widely applicable, yet complicated, friction law. In \cite{putelat2017a}, travelling wave solutions of a simple model for a thin sliding slab with this friction law were analyzed numerically. The results showed a rich bifurcation structure and demonstrated that the spinodal law captures most essential physical phenomena known from friction experiments, also those not produced by the Ruina or the Dieterich law. Ideally, in the future, we hope that our insight into the two simpler models, Ruina and Dieterich, eventually will allow for a detailed analysis of the spinodal law and increase our understanding of the numerical findings in \cite{putelat2017a}. 

\subsection{Singular analysis of \eqref{system}}
In terms of the fast time $\tau= \epsilon^{-1}t$, the (slow) system \eqref{system} becomes the (fast) system
\begin{align}
 x' &=-\epsilon e^z \left(x+(1+\alpha)z\right),\eqlab{systemNew}\\
 y' &=\epsilon (e^z -1),\nonumber\\
 z' &=-e^{-z} \left(y+\frac{x+z}{\xi}\right),\nonumber
\end{align}
Setting $\epsilon=0$ in \eqref{systemNew} then gives the layer problem 
\begin{align*}
 x' &=0,\\
 y'&=0,\\
 z'&=-e^{-z} \left(y+\frac{x+z}{\xi}\right),
\end{align*}
for which the plane
\begin{align}
 C = \left\{(x,y,z)\vert y+\frac{x+z}{\xi}=0\right\}.\eqlab{C2}
\end{align}
is the critical manifold. This manifold is normally hyperbolic and attracting since the linearization about any point $C$ gives 
\begin{align}
-\xi^{-1} e^{-z},\eqlab{eigenvalue}
\end{align}
as a single nonzero eigenvalue. However, $C$ is not compact. 
%

Setting $\epsilon=0$ in \eqref{system}, on the other hand, gives a reduced problem on $C$:
\begin{align*}
 \dot x &=-e^z \left(x+(1+\alpha)z\right),\\
 \dot y &=e^z -1,\nonumber
\end{align*}
with $z=\tilde m(x,y)$ where $$\tilde m(x,y)\equiv -\xi y-x,$$ is obtained from the expression of $C$ in \eqref{C2}. However, there are some advantages in working with the physical meaningful variables $(y,z)$ rather than $(x,y)$. Recall that $x$ is a ``state'' variable describing the friction. (It models a combination of effects and is difficult to measure and observe in practice, see e.g. \cite{woodhouse2015a}.) We therefore write $C$ as a graph 
\begin{align}
x=m(y,z), \eqlab{xmx}
\end{align} over $(y,z)\in \mathbb R^2$ with 
\begin{align}
 m(y,z) = -\xi y-z.\eqlab{mEqn}
\end{align}
Differentiating this then gives following reduced problem
\begin{align}
 \dot y &=e^z-1,\eqlab{reducedC}\\
 \dot z &=\xi+e^z \left(\alpha z-\xi y-\xi\right). \nonumber
\end{align}
on $C$, using the coordinates $(y,z)$. 
In \cite{bossolini2017a} the authors show that \eqref{reducedC} has a degenerate Hopf bifurcation at $\alpha=\xi$, where all periodic orbits emerge at once due to a Hamiltonian structure:
\begin{align}
 \begin{pmatrix}
  \dot y \\
  \dot z
 \end{pmatrix} = J(y,z) \nabla H(y,z),\eqlab{yzHam}
\end{align}
where
\begin{align*}
 J(y,z) &= \begin{pmatrix}
           0&\xi^{-1} e^{\xi y+z}\\
           -\xi^{-1} e^{\xi y+z} & 0
          \end{pmatrix},\\
         H(y,z) &=-\xi e^{-\xi y} (y-z +1-e^{-z})+1-e^{-\xi y}.
\end{align*}
The authors of \cite{bossolini2017a} then put the reduced problem \eqref{reducedC} on the Poincar\'e sphere in the following way: Consider $S^2=\{(\bar y,\bar z,\bar w)\in \mathbb R^3\vert \bar y^2+\bar z^2+\bar w^2=1\}$ and let $\phi:\,S^2\rightarrow \mathbb R^2$ be defined by 
\begin{align}
(\bar y,\bar z,\bar w)\mapsto \left\{\begin{matrix}
y&=& \bar w^{-1} \bar y,\\
z&=& \bar w^{-1} \bar z.
\end{matrix}\right.\eqlab{chartk2}
\end{align}
Then by pull-back, the vector-field \eqref{reducedC} gives a vector-field on $(\bar y,\bar z,\bar w)\in S^2\cap \{\bar w>0\}$. \eqref{chartk2} is then also a chart, obtained by central projection onto the hyperplane $\bar w=1$, parameterizing $\bar w>0$ of $S^2$. To describe $S^2$ near the equator $\bar w=0$ the authors in \cite{bossolini2017a} studied two separate directional charts:
\begin{align*}
 \phi_1:&\,S^2 \rightarrow \mathbb R^2,\\
 \phi_3:&\,S^2 \rightarrow \mathbb R^2,
\end{align*}
defined by
\begin{align}
 (\bar y,\bar z,\bar w)\mapsto \left\{\begin{matrix}
                                              z_1 &=& \bar y^{-1} \bar z,\\
                                              w_1 &=& {\bar y}^{-1}{\bar w},
                                             \end{matrix}\right.\eqlab{phi11}\\
                                             (\bar y,\bar z,\bar w)\mapsto \left\{\begin{matrix}
                                              y_3 &=& \bar z^{-1} \bar y,\\
                                              w_3 &=& {\bar z}^{-1} {\bar w},
                                             \end{matrix}\right.\eqlab{phi31}
\end{align}
respectively. These charts are obtained by central projections onto the planes tangent to $S^2$ at $\bar y=1$ and $\bar z=1$, respectively. See \figref{poincareCompactification1}. 
Using appropriate time transformations (basically slowing time down) near the equator $\bar w=0$, the authors then found three equilibria: $Q^1$ where $\bar y=0$, $\bar z=1$, $Q^3$ where $\bar y^{-1} \bar z = \alpha^{-1}\xi$, $\bar y>0$, and $Q^7$ where $\bar y^{-1}\bar z=-\xi$, $\bar y>0$, and one  singular ``$0/0$'' point: $Q^6$ where $\bar y=1,\bar z=0$. Here $Q^1$ is a stable hyperbolic node while $Q^7$ is an unstable hyperbolic node. The point $Q^3$, on the other hand, is a nonhyperbolic saddle, with a hyperbolic unstable manifold along the equator and a nonhyperbolic stable manifold (a unique center manifold), which we denote by $W^{cs}(Q^3)$. Finally, the point $Q^6$ acts like a saddle, with one ``stable manifold'' along the equator of the sphere, and a unique center-like unstable manifold, which we shall denote $W^{cu}(Q^6)$. See also \figref{poincareCompactification1}. We describe the invariant manifolds of $Q^3$ and $Q^6$ using the original coordinates $(y,z)$ of $C$ in the following lemma.

\begin{figure}[h!]
\begin{center}
{\includegraphics[width=.625\textwidth]{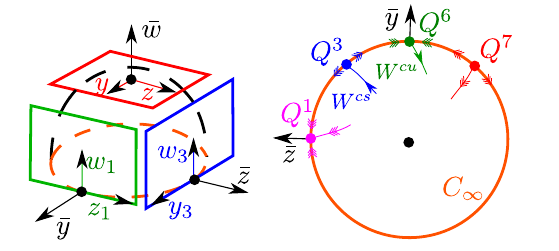}}
\end{center}
\caption{Poincar\'e compactification of the reduced problem.}
\figlab{poincareCompactification1}
\end{figure}
 \begin{lemma}\lemmalab{Wcsu}
 \cite[Proposition 5.1]{bossolini2017a}
Consider any $\alpha>0,\xi>0$. Then there exists two unique one-dimensional invariant manifolds $W^{cu}(Q^6)$ and $W^{cs}(Q^3)$ for the reduced flow on $C$ with the following asymptotics:
 \begin{align}
  z &= -\log (y) \left(1+\frac{\alpha}{\xi y}\right),\eqlab{WcuA}\\
 z &= \frac{\xi}{\alpha}y+\frac{(1+\alpha)\xi}{\alpha^2},\eqlab{WcsA}
 \end{align}
as $y\rightarrow \infty$, respectively. $W^{cu}(Q^6)$ is the set of all trajectories with the asymptotics in \eqref{WcuA} backwards in time (or simply, the unstable set of $Q^6$) whereas $W^{cs}(Q^3)$ is the set of all trajectories with the asymptotics \eqref{WcsA} forward in time (or simply, the stable set of $Q^3$). Moreover, for $\alpha=\xi$, $W^{cs}(Q^3)$ and $W^{cu}(Q^6)$ coincide, such that there exists a unique orbit on $C$ with the asymptotics in \eqref{WcuA}$_{\alpha=\xi}$ in backward time and \eqref{WcsA}$_{\alpha=\xi}$ in forward time, respectively. The intersection is transverse in $(y,z,\alpha)$-space:
\begin{itemize}
 \item[(a)] For $\alpha>\xi$: $W^{cu}(Q^6)$ is contained within the stable set of $Q^1$, in such a way that $z(t)\rightarrow \infty$ and $z(t)^{-1}y(t)\rightarrow 0$ with $y(t)>0$, in forward time, 
while $W^{cs}(Q^3)$ is contained within the unstable set of $(y,z)=(0,0)$. 
\item[(b)] For $\alpha<\xi$: $W^{cu}(Q^6)$ is contained within the stable set of $(y,z)=(0,0)$, while $W^{cs}(Q^3)$ is contained within the unstable set of $Q^7$ with the asymptotics 
\begin{align*}
 z=-\xi y,
\end{align*}
for $y\rightarrow \infty$ in backward time. 
\end{itemize}

%
%
%
%
%
\end{lemma}
\begin{proof}
 See \cite[Proposition 5.1]{bossolini2017a}. Notice, in \cite{bossolini2017a}, however, the authors use Melnikov theory and only deduce (a) and (b) locally near $\alpha=\xi$. To show that these statements hold for any $\alpha>\xi$ and $\alpha<\xi$, respectively, we simply use that $H$ is a Lyapunov function:
 \begin{align*}
 \frac{dH}{dt}(y,z)=- \xi e^{-\xi y} (e^z-1)z(\alpha-\xi),
\end{align*}
such that $\text{sign}(H(y,z)) = -\text{sign}(\alpha-\xi)$ for all $y,z\ne 0$ and $\alpha\ne \xi$. 
Therefore for $\alpha>\xi$, $H$ increases
monotonically along all orbits ($\ne (y,z)(t)\equiv (0,0)$) of \eqref{reducedC}. Therefore limit cycles cannot exist. Recall that $Q^1$ is a stable node on the Poincar\'e sphere, while $Q^7$ is an unstable node. By Poincar\'e-Bendixson, $W^{cu}(Q^6)$ is asymptotic to $Q^1$ when $\alpha>\xi$. The approach is similar for $\alpha<\xi$.  
\end{proof}

By this lemma, we obtain the global phase portraits in \figref{poincareReduced} for the reduced problem. 
%
\begin{figure}[h!]
\begin{center}
\subfigure[]{\includegraphics[width=.295\textwidth]{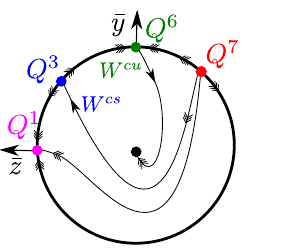}}
\subfigure[]{\includegraphics[width=.295\textwidth]{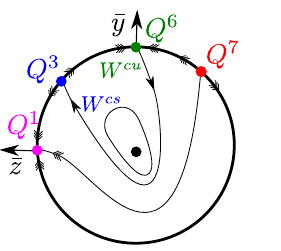}}
\subfigure[]{\includegraphics[width=.295\textwidth]{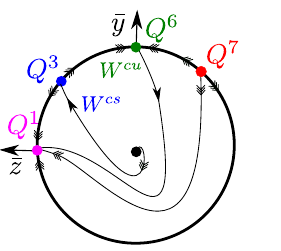}}
\end{center}
\caption{Phase portraits of the reduced problem on the Poinacar\'e sphere. In (a): $\alpha<\xi$, in (b): $\alpha=\xi$ (the Hamiltonian case) and finally in (c): $\alpha>\xi$. }
\figlab{poincareReduced}
\end{figure}
\subsection{Main results}




In this section we now consider the $\epsilon\ll 1$ system. 
In \cite{bossolini2017a}, the authors apply Poincar\'e compactification of the full system \eqref{system} defining $\Phi:\,S^3\rightarrow \mathbb R^3$, $S^3 = \{(\bar x,\bar y,\bar z,\bar w)\vert \bar x^2+\bar y^2+\bar z^2+\bar w^2=1\}$ by
\begin{align}
(\bar x,\bar y,\bar z,\bar w)\mapsto \left\{\begin{matrix}x &=& \bar w^{-1} \bar x,\\
y&=& \bar w^{-1} \bar y,\\
z&=& \bar w^{-1} \bar z.
\end{matrix}\right.\eqlab{chartK2}
\end{align}
%
%
By \eqref{chartk2} and \eqref{xmx}, we obtain $C$ as an ellipsoid (or actually a hemisphere hereof) within $S^3= \{(\bar x,\bar y,\bar z,\bar w)\vert \bar x^2+\bar y^2+\bar z^2+\bar w^2=1\}$, the equator of which, along $\overline w=0$, contains the corresponding points $Q^1$, $Q^3$, $Q^6$ and $Q^7$ along the boundary $C_\infty$ of $C$.
We use the directional charts
\begin{align*}
 \phi_1:&\,S^3 \rightarrow \mathbb R^3,\\
 \phi_3:&\,S^3 \rightarrow \mathbb R^3,
\end{align*}
in the following, 
defined by
\begin{align}
 (\bar x,\bar y,\bar z,\bar w)\mapsto \left\{\begin{matrix}
                                               x_1 &=& \bar y^{-1} \bar x,\\
                                              z_1 &=& \bar y^{-1} \bar z,\\
                                              w_1 &=& {\bar y}^{-1}{\bar w},
                                             \end{matrix}\right.\eqlab{phi1}\\
                                             ( \bar x, \bar y,\bar z,\bar w)\mapsto \left\{\begin{matrix}
                                              x_3 &=&\bar z^{-1} \bar x,\\
                                              y_3 &=& \bar z^{-1} \bar y,\\
                                              w_3 &=& {\bar z}^{-1} {\bar w},
                                             \end{matrix}\right.\eqlab{phi3}
\end{align}
respectively. Here we misuse notation slightly and reuse the symbols in \eqref{phi11} and \eqref{phi31} for the new charts. Notice that the coordinate transformation between $\phi_1$ and $\phi_3$ can be derived from the expressions
\begin{align}
 x_1 &=y_3^{-1} x_3,\eqlab{phi13}\\
 z_1 &=y_3^{-1},\nonumber\\
 w_1 &=y_3^{-1} w_3,\nonumber
\end{align}
for $z_1>0$ and $y_3>0$. 
Furthermore, the coordinates in $\phi_1$ and $\phi_3$ and the original coordinates $(x,y,z)$ are related as follows
\begin{align}
 x &=w_1^{-1} x_1 =w_3^{-1} x_3,\eqlab{phi13xyz}\\
 y&=w_1^{-1}\,\,\,\,\,\, = w_3^{-1} y_3,\nonumber\\
 z&=w_1^{-1} z_1 \,= w_3^{-1},\nonumber
\end{align}
using \eqref{chartK2}. 

Following \cite{bossolini2017a}, we define a ``singular'' cycle as follows:
\begin{definition}\defnlab{Gamma}
 \cite[Definition 1]{bossolini2017a} 
 Let the points $Q^{1,2,4,5,6}$ be given by
\begin{align*}
 Q^1_3 &= (-1,0,0),\\
 Q^2_3 &= (-1- \alpha,0,0),\\
 Q^4_3 &=\left(-1-\alpha,\frac{2\alpha}{\xi},0\right),
\end{align*}
in the coordinates $(x_3,y_3,w_3)$ of chart $\phi_3$,
\begin{align}
Q^5_1 &= \left(-\frac{\xi}{2\alpha}(1+\alpha),\frac{\xi}{2\alpha}(1-\alpha),0\right),\eqlab{Q5expr}\\
 Q^6_1&=(-\xi,0,0),\nonumber
\end{align}
in the coordinates $(x_1,z_1,w_1)$ of chart $\phi_1$.  Then  for any $\alpha>\xi$, we define the (singular) cycle $\Gamma_0$ as follows
    \begin{align}
  \Gamma_0  = \gamma^3 \cup \gamma^4 \cup \gamma^7 \cup \gamma^9\cup W^{cs}(Q^6), \eqlab{Gamma0Eqn}
  \end{align}
where 
 
%
%
 \begin{itemize}
  \item $\gamma^{3}$ connects $Q^1$ and $Q^2$. In the $(x_3,y_3,w_3)$-coordinates it is given as
  \begin{align}
   \gamma^3_3 = \{(x_3,y_3,w_3)\vert x_3\in (-1-\alpha,-1],y_3=w_3=0\}.\eqlab{gamma33}
  \end{align}
\item $\gamma^4$ connects $Q^2$ with $Q^4$. In the $(x_3,y_3,w_3)$-coordinates it is given as
\begin{align*}
 \gamma_3^4 =\{(x_3,y_3,w_3)\vert x_3= -1-\alpha,w_3 = 0,y_3\in [0,2\alpha/\xi)\}.
\end{align*}
\item $\gamma^7$ connects $Q^4$ with $Q^5$. In the $(x_1,z_1,w_1)$-coordinates it is given as
\begin{align}
 \gamma^7_1 = \bigg\{(x_1,z_1,w_1)\vert x_1 &= -\frac{\xi}{2\alpha}(1+\alpha),\\
 z_1&\in \bigg(\frac{\xi}{2\alpha}(1-\alpha),\frac{\xi}{2\alpha}\bigg],w_1=0\bigg\}.\eqlab{gamma71here}
\end{align}
\item $\gamma^9$ connects $Q^5$ with $Q^6$ on $C_\infty$. In the $(x_1,z_1,w_1)$-coordinates it is given as
\begin{align}
 \gamma^9_1 = \bigg\{(x_1,z_1,w_1)\vert x_1& = -\xi-z_1,\\
 z_1&\in \bigg(0,\frac{\xi}{2\alpha}(1-\alpha)\bigg],w_1=0\bigg\},\eqlab{gamma91}
\end{align}
for $0<\alpha<1$. For $\alpha=1$, $\gamma^9_1$ is the empty set, and for $\alpha> 1$ the interval for $z_1$ has to be swapped around such that $z_1\in [{\xi}(1-\alpha)/({2\alpha}),0)$.
\item $W^{cu}(Q^6)$ is the unique center manifold of $Q^6$ for the reduced problem \eqref{reducedC}, described in \lemmaref{Wcsu}, connecting $Q^6$ with $Q^1$ (given that $\alpha>\xi$) in forward slow time. 
 \end{itemize}
 
\end{definition}
\begin{figure}[h!]
\begin{center}
{\includegraphics[width=.695\textwidth]{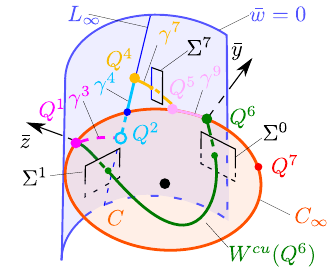}}
\end{center}
\caption{Illustration of the singular cycle $\Gamma_0$ on the Poincar\'e sphere for $\xi<\alpha<1$. For $\alpha=1$, $\gamma^7$ connects directly to $Q^6$ and therefore $\gamma^9$ ``disappears'' whereas for $\alpha>1$, $\gamma^7$ connects to $C_\infty$ with $\bar z<0$ so that $\gamma^9$ gets ``flipped'' on $C_\infty$ relative to $Q^6$. This change in the singular cycle can be observed bifurcation diagrams, see \figref{bif}.  }
\figlab{Gamma0}
\end{figure}
%
%

The segment $\gamma^4$ belongs to a center-like manifold $L_\infty$, that the authors in \cite{bossolini2017a} identified using the blowup method \cite{kristiansen2017a}, also applied in the present paper. 
In the $(x_3,y_3,w_3)$-coordinates it is given as the line
\begin{align}
 L_{\infty,3} = \{(x_3,y_3,w_3)\vert x_3= -1-\alpha,w_3 = 0,y_3\in I\},\eqlab{Lhere}
\end{align}
where $I\subset \mathbb R$ is a large interval. 
$\gamma^3$ is given by the contraction towards this manifold. The segment $\gamma^7$ connects $Q^4$ on $L_\infty$ with a point $Q^5$ on $C_\infty$. The final segment $\gamma^9$ is a segment on $C_\infty$ following the ``desingularized'' reduced slow flow on $C_\infty$ (basically using the time that produces \figref{poincareReduced}). We illustrate $\Gamma_0$ and the segments in \figref{Gamma0}. Here we represent $C$ as a disk and the equator sphere $\bar w=0$ (locally) as a cylindrical object containing $C_\infty$ as a circle.

In \figref{xyz}, we illustrate the set $L$ in the $(x,y,z)$-coordinates obtained by extending \eqref{Lhere} for $w_3>0$ sufficiently small and applying the coordinate change \eqref{phi13xyz}:
\begin{align}
 L =\{(x,y,z)\vert x = (-1-\alpha)z,\,y/z \in I,\,z\gg 0\}.\eqlab{LExpr}
\end{align}
The role of this set (and therefore also the role of $L_\infty$, given that the amplitude increases as $\epsilon\rightarrow 0$) is clearly visible in these diagrams. 
\begin{remark}\remlab{ElenaArgument}
 \cite{bossolini2017a} presents a heuristic argument for how $L$ appears which we for convinience also include here. Divide the right hand side of \eqref{system} by $e^z$ and suppose that $e^{-2z}\gg \epsilon$. Then
 \begin{align}
 \dot x &=- \left(x+(1+\alpha)z\right),\eqlab{xyzHereNew}\\
 \dot y &=1,\nonumber\\
 \dot z &=0.\nonumber
\end{align}
to ``leading order''. The set $x=(-1-\alpha)z$, producing \eqref{LExpr}, is an invariant set of \eqref{xyzHereNew}, along which $y$ increases monotonically. But notice that this naive approach does not explain how orbits leave a neighborhood of $L$. For this we need a more detailed analysis, which we provide in the present paper.  
\end{remark}

%
%
%
In this paper, we prove the following result, conjectured in \cite{bossolini2017a}.
\begin{theorem}\thmlab{mainThm}
 Fix $\xi>0$ and any compact set $K$ in $\mathbb R^3$. Then for all $\alpha>\xi$ the following holds: 
 \begin{itemize}
 \item[(a)] There exists an $\epsilon_0>0$ such that the system has an attracting limit cycle $\Gamma_\epsilon$ for all $0<\epsilon\le \epsilon_0$ which is not contained within $K$. 
 \item[(b)] Moreover, on the Poincar\'e sphere, $\Gamma_\epsilon$ converges in Hausdorff distance to the singular cycle $\Gamma_0$ as $\epsilon\rightarrow 0$. 
 \end{itemize}
\end{theorem}
The main difficulty in proving this result is that $C_0$ loses hyperbolicity at $\bar w=0$. The loss of hyperbolicity is due to the exponential decay of the single non-zero eigenvalue, see \eqref{eigenvalue}.  
To deal with this type of loss of hyperbolicity, we use the method in \cite{kristiansen2017a}, developed by the present author, to gain hyperbolicity in an extended space.

Besides providing all the details of the analysis to obtain a rigorous proof of \thmref{mainThm}, which was only conjectured in \cite{bossolini2017a}, we also provide a better overview of the analysis and the many blowup steps (we count 16 in total!). We lay out the geometry of the blowups and detail the charts and the corresponding coordinate transformations. Also, in the present manuscript we provide a complete analysis of the dynamics near $Q^6$ for $\epsilon>0$, which is missing at any level of formality in \cite{bossolini2017a}. Our blowup approach allows us to identify an improved singular cycle, consisting of $12$ segments, with better hyperbolicity properties. The additional segments $\gamma^{1,2,5,6,8,10,11}$, not visible in the blown down version of $\Gamma_0$ in \figref{Gamma0}, see \defnref{Gamma} also \eqref{Gamma0Eqn}, are described carefully in \secref{detailsK3} and \secref{detailsK1}, see also \figref{gamma} and \figref{gammaK1} from the perspective of $\phi_1$ and $\phi_3$, respectively. 

\subsection{Outline}
In the remainder of this paper we prove \thmref{mainThm}. First, in \secref{proofThm} we present two central lemmas, \lemmaref{Wcsu1} and \lemmaref{kristianLemma}, which in combination proves \thmref{mainThm}. \lemmaref{Wcsu1} is standard whereas \lemmaref{kristianLemma} requires substantial work. We prove \lemmaref{kristianLemma} by splitting it into two parts described in the two charts $\phi_3$ and $\phi_1$. We study these charts in \secref{phi3Section} and \secref{phi1Section}, respectively. Here we apply the method in \cite{kristiansen2017a} and lay out the necessary blowup steps. In each of these sections, we combine the results of blowup analysis, performed in details in \secref{detailsK3} and \secref{detailsK1}, respectively, into two separate lemmas, see \lemmaref{Pi17} and \lemmaref{Pi70}. In \secref{proofofkristianLemma}, we prove \lemmaref{kristianLemma} using \lemmaref{Pi17} and \lemmaref{Pi70}. In \secref{Discussion} we discuss some consequences of \thmref{mainThm} and directions for future work on the topic. 
\section{Proof of \thmref{mainThm}}\seclab{proofThm}
Consider the reduced problem \eqref{reducedC} and $\alpha>\xi$. Then by \lemmaref{Wcsu}, $W^{cu}(Q^6)$ intersects 
 $y = \delta^{-1}$ in a unique point 
 \begin{align}
 q^0=(x^0,\delta^{-1},z^0),\eqlab{q0Here}
 \end{align}
 with $$z^0\approx -\log(\delta^{-1})\left(1+\frac{\alpha \delta}{\xi}\right),$$
cf. \eqref{WcuA}, and $x^0 = m(\delta^{-1},z^0)$, for $\delta>0$ sufficiently small. Let $N^0$ be a small neighborhood of $(x^0,z^0)$ in $\mathbb R^2$. We therefore define a section $\Sigma^0$ as follows
\begin{align}
 \Sigma^0 = \{(x,y,z)\vert y=\delta^{-1},\,(x,z)\in N^0\}.\eqlab{Sigma0}
\end{align}
By \lemmaref{Wcsu} again, $W^{cu}(Q^6)$ also intersects $z = \delta^{-1}$ in a unique point $q^1=(x^1,y^1,\delta^{-1})$ with $y^1>0$ and $x^1=m(y^1,\delta^{-1})$ for $\delta>0$ sufficiently small, recall \eqref{mEqn}. See also \cite[Proposition 5.2]{bossolini2017a}. Then we define a section $\Sigma^1$ as follows
\begin{align}
 \Sigma^1 =  \{(x,y,z)\vert z=\delta^{-1},\,(x,y)\in N^1\},\eqlab{Sigma1}
\end{align}
where $N^1$ is a small neighborhood $(x^1,y^1)$ in $\mathbb R^2$. 
See \figref{Sigma0Sigma1}. Notice that \eqref{systemNew} is transverse to $\Sigma^0$. Also the reduced flow on $C$ is transverse to $\Sigma^1$.
\begin{figure}[h!]
\begin{center}
{\includegraphics[width=.625\textwidth]{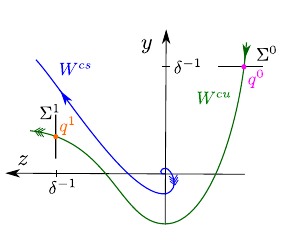}}
\end{center}
\caption{Sections $\Sigma^0$ and $\Sigma^1$.}
\figlab{Sigma0Sigma1}
\end{figure}
Let $\Pi^{0}:\,\Sigma^0\rightarrow \Sigma^1$ be defined for $0<\epsilon\ll 1$ as the transition mapping obtained by the first intersection through the forward flow of \eqref{system}. For $\epsilon=0$, we similarly define $\Pi^0:\Sigma^0\rightarrow \Sigma^1$ as the composition of the following mappings: (a) the projection $\,(x,\delta^{-1},z)\mapsto (x,\delta^{-1},\tilde m(x,\delta^{-1}))$ onto $C$ defined by the stable, critical fibers. (b): The mapping obtained from $(x,\delta^{-1},\tilde m(x,\delta^{-1}))$ by the first intersection with $\Sigma^1$ through the forward flow of the reduced problem on $C$. Hence $\Pi^0$ only depends upon $x$ for $\epsilon=0$:
\begin{align*}
 \Pi^0(x,\delta^{-1},z;0) =\Pi^0(x,\delta^{-1},\tilde m(x,\delta^{-1}))\in C\cap \Sigma^1,
\end{align*}
for all $(x,\delta^{-1},z)\in \Sigma^0$.
Notice, we write $\Pi^0(\cdot;\epsilon)$ to highlight the dependency of $\Pi^0$ on $\epsilon$ (as a parameter). By Fenichel's theory \cite{fen1,fen2,fen3,jones_1995}, we then have the following result.
\begin{lemma}\lemmalab{Wcsu1}
For $N^0$ sufficiently small there exists an $\epsilon_0>0$ such $\Pi^0$ is well-defined and $C^{k\ge 1}$-smooth, even in $\epsilon\in [0,\epsilon_0]$. In particular 
 \begin{align*}
  \Pi^0(x,\delta^{-1},z;\epsilon)= \Pi^0(x,\delta^{-1},\tilde m(x,\delta^{-1});0)+\mathcal O(\epsilon).
 \end{align*}
%
\end{lemma}

The main problem of the proof of \thmref{mainThm} is to prove the following result: Let 
\begin{align}
\Pi^{1}:\,\Pi^0(\Sigma^0)\subset \Sigma^1 \rightarrow \Sigma^0,\eqlab{Pi1Eqn}
\end{align}
be the mapping obtained by the first intersection by the forward flow.  Then we have the following:
\begin{lemma}\lemmalab{kristianLemma}
 There exist a $\delta>0$, a sufficiently small set $N^0$, and an $\epsilon_0>0$ such that the mapping $\Pi^{1}(\cdot;\epsilon)$ is well-defined and $C^1$ for all $0<\epsilon\le \epsilon_0$. In particular, $\Pi^{1}(x,y,\delta^{-1};\epsilon)$ is $C^1$ $o(1)$-close to the constant function $q^0$ as $\epsilon\rightarrow 0$. 
\end{lemma}
Let $\Pi=\Pi^1\circ \Pi^0$. Then by \lemmaref{Wcsu1} and \lemmaref{kristianLemma}, $\Pi$ is a contraction for $\epsilon\ll 1$. The existence of an attracting limit cycle $\Gamma_\epsilon$ in \thmref{mainThm} (a) therefore follows from the contraction mapping theorem - the attracting limit cycle being obtained as the forward flow of the unique fix-point of $\Pi$. The convergence of $\Gamma_\epsilon$ as $\epsilon\rightarrow 0$ in \thmref{mainThm} (b) is a consequence of our approach. We actually ``derive'' $\Gamma_0$ first using successive blowup transformations (working in the charts $\phi_3$ and $\phi_1$) that allow us to prove \lemmaref{kristianLemma} using standard, local, hyperbolic methods of dynamical systems theory obtaining $\Gamma_\epsilon$ as a ``perturbation'' of $\Gamma_0$. 
In more details, we further decompose $\Pi^1$ into two parts $\Pi^{17}$ and $\Pi^{70}$ where $\Pi^{17}:D(\Pi^{17})\subset \Sigma^1\rightarrow \Sigma^7$ and $\Pi^{70}:\Sigma^7\rightarrow \Sigma^0$. Here $\Sigma^7$ is an appropriate $2D$-section, transverse to $\gamma^7$ \eqref{gamma71here}, contained within $y_3=\frac{2\alpha(1+\nu)}{\xi}$ for $\nu>0$ small, see \figref{Gamma0} for an illustration. We describe these mappings in details in the following sections, see \lemmaref{Pi17} and \lemmaref{Pi70}. 

\section{Chart $\phi_3$}\seclab{phi3Section}
In this chart, we obtain the following equations
\begin{align}
 \dot x_3&=-\epsilon (x_3+1+\alpha)+x_3 e^{-2/w_3} \left(y_3+\frac{x_3+1}{\xi}\right),\eqlab{phi3Eqs}\\
 \dot y_3 &=\epsilon w_3(1-e^{-1/w_3})+y_3e^{-2/w_3} \left(y_3+\frac{x_3+1}{\xi}\right),\nonumber\\
 \dot w_3 &=w_3 e^{-2/w_3} \left(y_3+\frac{x_3+1}{\xi}\right),\nonumber
 \end{align}
 using the coordinates $(x_3,y_3,w_3)$, recall \eqref{phi3}. 
Here we cover the part of the critical manifold $C$ \eqref{C2} with $z>0$ as follows
\begin{align*}
 C_3 = \left\{(x_3,y_3,w_3)\vert y_3 + \frac{x_3+1}{\xi}=0,w_3>0\right\}.
\end{align*}
This manifold is still a normally hyperbolic and attracting critical manifold of \eqref{phi3Eqs} in the present chart: The linearization about any point in $C_3$ gives
\begin{align}
 -\xi^{-1} e^{-2/w_3}<0,\eqlab{eigValExp}
\end{align}
for $w_3>0$, 
as a single nonzero eigenvalue. 
But we now also obtain $\{w_3=0\}$, corresponding to the subset of the equator $S^3\cap \{\bar w=0\}$ with $\bar z>0$, as a set of fully nonhyperbolic critical points for $\epsilon=0$. Indeed, the linearization about any point in $\{w_3=0\}$ only has zero eigenvalues.  The intersection $C_3\cap \{w_3=0\}$:
\begin{align*}
 C_{3,\infty} = \left\{(x_3,y_3,w_2)\vert y_3 + \frac{x_3+1}{\xi}=0,w_3=0\right\}.
\end{align*}
is therefore also fully nonhyperbolic for $\epsilon=0$. The exponential decay of \eqref{eigValExp} complicates the blowup analysis and the study of what happens near $\{w_3=0\}$ and $C_{3,\infty}$ for $0<\epsilon\ll 1$. 
%
We follow the blowup approach in \cite{kristiansen2017a}, also used in \cite{bossolini2017a}, and extend the phase space dimension by introducing
\begin{align}
 q_3 = e^{-2/w_3}.\eqlab{q3K3}
\end{align}
By implicit differentiation, we obtain 
\begin{align*}
 \dot q_3 =2w_3^{-1} e^{-2/w_3}\dot w_3 = 2w_3^{-1} q_3^2 \left(y_3+\frac{x_3+1}{\xi}\right).
 \end{align*}
 We therefore consider the extended system
\begin{align}
\dot x &=-\epsilon w (x+1+\alpha)+x w q \left(y+\frac{x+1}{\xi}\right),\eqlab{K3Ext}\\
 \dot y &=\epsilon w^2(1-e^{-1/w})+y w q \left(y+\frac{x+1}{\xi}\right),\nonumber\\
 \dot w &=w^2 q \left(y+\frac{x+1}{\xi}\right),\nonumber\\
\dot q &= 2 q^2 \left(y+\frac{x+1}{\xi}\right),\nonumber\\
\dot \epsilon &=0,\nonumber
\end{align}
having here dropped the subscripts, multiplied the right hand side by $w=w_3$ and finally introduced $\epsilon$ as a dynamic variable. Now by construction, the set
\begin{align}
\{(x,y,w,q,\epsilon)\vert q=e^{-2/w}\}
\end{align}
is an invariant of this system. But this invariance is implicit in the system \eqref{K3Ext} and we shall use it only when needed. Now, we define $C$ by
\begin{align}
C = \left\{(x,y,w,q,\epsilon)\vert y + \frac{x+1}{\xi}=0,w>0,q>0,\epsilon=0\right\},\nonumber
\end{align}
in the extended system, using, for simplicity, the same symbol. It is still a set of normally hyperbolic critical points, now of dimension $3$, since the linearization about any point in $C$ has one single nonzero eigenvalue $-wq/\xi$. Similarly, $\{w=\epsilon=0\}$ and $\{q=\epsilon=0\}$ are fully nonhyperbolic sets of equilibria for \eqref{K3Ext}. The system is therefore very degenerate near
\begin{align}
 C_\infty = \left\{(x,y,w,q,\epsilon)\vert y + \frac{x+1}{\xi}=0,w=q=0,\epsilon=0\right\}.\eqlab{CPhi3ExtInf}
\end{align}
But the system \eqref{K3Ext} is now algebraic to leading order and therefore we can (in principle) apply the classical blowup method of \cite{dumortier_1996,krupa_extending_2001} to study the dynamics near $C_\infty$. We will have to use five separate, successive blowup transformations in the present $\phi_3$-chart. We describe these in the following section.

\subsection{Blowups in chart $\phi_3$}
Let 
\begin{align*}
 P  &= \left\{(x,y,w,q,\epsilon) \in \mathbb R^2\times [0,\infty)^3\right\}.\\
%
 P^1&=\left\{(x,y,w,r,(\bar q,\bar \epsilon))\in \mathbb R^2 \times [0,\infty)^2 \times S^1\right\}.
\end{align*}
Then we first apply a blowup of $\{(x,y,w,q,\epsilon)\in P\vert q=\epsilon=0\}$ to a cylinder through the following blowup transformation $$\Psi^1:P^1\rightarrow P,$$ which fixes $x$, $y$ and $z$ and takes 
\begin{align}
 (r,(\bar q,\bar \epsilon)) \mapsto (q,\epsilon)=r(\bar q,\bar \epsilon),\quad r\ge 0,\,(\bar q,\bar \epsilon)\in S^1.\eqlab{step1}
\end{align}
We illustrate this blowup transformation in \figref{step1}(a). Notice how we artistically combine the $xyw$-space into a single coordinate axis. We use red colours and lines, also in the following, to indicate what variables and coordinate axes that are included in each blowup in \figref{step1}. Points that are blown up are given red dots. 
Clearly, $\Psi^1$ by \eqref{step1} simply corresponds to introducing polar coordinates in the $(q,\epsilon)$-plane. We can therefore study a small neighborhood of $(q,\epsilon)=0$ by studying any $(r,(\bar q,\bar \epsilon))\in [0,\infty)\times S^1$ with $r\ge 0$ small. But the preimage of $\{q=0,\epsilon=0\}$ is a cylinder $(x,y,w,(\bar q,\bar \epsilon))\in \mathbb R^2 \times [0,\infty)\times S^1$. This is in the sense that we understand blowup. Now, the mapping $\Psi^1$ gives rise to a vector-field $\overline X^1$ on $P^1$ by pull-back of the vector-field \eqref{K3Ext} on $P$. Here we shall see that $\overline X^1$ has $r$ as a common factor, so that in particular $\overline X\vert_{r=0}=0$. We will therefore desingularize and study $\widehat X^1 = r^{-1} \overline X^1$ in the following. 

Let
\begin{align*}
 P^2 = \left\{(y,r,\rho,(\bar x,\bar w,\bar {\bar \epsilon}))\in \mathbb R \times [0,\infty)^2 \times S^2\right\}.
\end{align*}
Then in the second step, we blowup $x=-1-\xi y,\,w=0,\,(\bar q,\bar \epsilon)=(1,0)$ for each $y$ within $P^1$ through the blowup transformation $$\Psi^2:P^2\rightarrow P^1,$$ which fixes $y$ and $r$ and takes
\begin{align}
(y,\rho,(\bar x,\bar w,\bar {\bar \epsilon}))\mapsto \left\{\begin{matrix} x &=& -1-\xi y +\rho \bar x,\\
 w&=&\rho \bar w,\\
 \bar q^{-1} \bar \epsilon& =& \rho \bar{\bar \epsilon}.\end{matrix}
\right. \quad \rho\ge 0,(\bar x,\bar w,\bar {\bar \epsilon})\in S^2.\eqlab{step2}
\end{align}
We illustrate this blowup in \figref{step1}(b). Notice that $C_\infty$ in \eqref{CPhi3ExtInf} is the graph $x=-1-\xi y$ over $y$ within $\epsilon=w=q=0$. The second blowup therefore blows up $C_\infty$. 
 
 Clearly, we can study a small neighborhood of $x=-1-\xi y,\,w=0,\,(\bar q,\bar \epsilon)=(1,0)$ by studying $(\rho,(\bar x,\bar w,\bar{\bar \epsilon}))\in [0,\infty)\times S^2$ by taking $\rho\ge 0$ small.  As before, the mapping $\Psi^2$ gives rise to a vector-field $\overline X^2=\Psi^{2*}(\widehat X^1)$ on $P^2$ by pull-back of $\widehat X^1$ on $P^1$. Now, $\overline X^2$ has $\rho$ as a common factor and we therefore study $\widehat X^2=\rho^{-1}\overline X^2$. 
 \begin{remark}Notice that since $(\bar q,\bar \epsilon)\in S^1$, a simple calculation shows that the last equality in \eqref{step2} imply that
\begin{align*}
 (\bar q,\bar \epsilon) = \left(\sqrt{1+\rho^2 \bar{\bar \epsilon}^2},\frac{\rho \bar{\bar \epsilon}}{\sqrt{1+\rho^2 \bar{\bar \epsilon}^2}}\right).
\end{align*}
\end{remark}


Let 
\begin{align*}
 P^3 = \left\{(y,r,\rho,\varrho,(\bar{\bar x},\bar{\bar w}))\in \mathbb R\times [0,\infty)^3\times S^1\right\}.
\end{align*}
Then in the third step, we then proceed to blowup $\bar x=\bar w=0,\bar{\bar \epsilon}=1,\,\rho\ge 0$ for each $y$ within $P^2$ through the blowup transformation $$\Psi^3:P^3\rightarrow P^2$$ which fixes $y$, $r$ and $\rho$ and takes
\begin{align}
(\varrho, (\bar{\bar x},\bar{\bar w}))\mapsto \left\{\begin{matrix}
                                                                                    \bar{\bar \epsilon}^{-1}\bar x &=& \varrho \bar{\bar x},\\
 \bar{\bar \epsilon}^{-1} \bar w &=&\varrho^2 \bar{\bar w},
                                                                                   \end{matrix}\right.\quad \varrho\ge 0,\,(\bar{\bar x},\bar{\bar w})\in S^1.
 \eqlab{step3} 
\end{align}
We illustrate this in \figref{step1}(c).

Clearly, we can study a small neighborhood of $(\bar x,\bar w,,\bar{\bar \epsilon})=(0,0,1)$ by studying $(\varrho,(\bar{\bar{x}},\bar{\bar{w}}))\in [0,\infty)\times S^2$ with $\varrho\ge 0$ small. $\Psi^3$ gives a vector-field $\overline X^3=\Psi^{3*}(\widehat X^2)$ on $P^3$ by pull-back of $\widehat X^2$ on $P^2$. $\overline X^3$ now has $\varrho$ as a common factor and we therefore study $\widehat X^3=\varrho^{-1}\overline X^3$.   
\begin{remark}
Since $(\bar x,\bar w,\bar{\bar \epsilon})\in S^2$ we can also write the right hand side of \eqref{step3} as
\begin{align*}
(\bar x,\bar w,\bar{\bar \epsilon}) = \left(\frac{\varrho \bar{\bar x}}{\sqrt{1+\varrho^2\bar{\bar x}^2+\varrho^4\bar{\bar w}^2}},\frac{\varrho^2 \bar{\bar w}}{\sqrt{1+\varrho^2\bar{\bar x}^2+\varrho^4\bar{\bar w}^2}},\frac{1}{\sqrt{1+\varrho^2\bar{\bar x}^2+\varrho^4\bar{\bar w}^2}}\right).
\end{align*}
This follows from a simple calculation.\end{remark}
%
In the following, we define
\begin{align*}
 \Psi^{12}:P^{2}\rightarrow P
\end{align*}
and 
\begin{align*}
 \Psi^{123}:P^3\rightarrow P
\end{align*}
by the compositions
\begin{align*}
\Psi^{12}=\Psi^1\circ \Psi^2,\quad  \Psi^{123}=\Psi^{12}\circ \Psi^3.
\end{align*}
Therefore by \eqref{step1}, \eqref{step2} and \eqref{step3}
\begin{align}
\Psi^{12}:(y,r,\rho,(\bar x,\bar w,\bar {\bar \epsilon})) &\mapsto \left\{ \begin{matrix}
                                                                                x&=&-1-\xi y +\rho \bar x,\\
                                                                                y&=&y\\
                                                                                w&=&\rho\bar w,\\
                                                                                q&=&\frac{r}{\sqrt{1+\rho^2 \bar{\bar \epsilon}^2}},\\
                                                                                \epsilon&=&\frac{r\rho \bar{\bar \epsilon}}{\sqrt{1+\rho^2 \bar{\bar \epsilon}^2}},
                                                                               \end{matrix}\right.\nonumber\\
 \Psi^{123}:(y,r,\rho,\varrho,(\bar{\bar x},\bar{\bar w}))&\mapsto \left\{\begin{matrix}
                                                                                x&=&-1-\xi y+\frac{\rho \varrho \bar{\bar x}}{\sqrt{1+\varrho^2\bar{\bar x}^2+\varrho^4 \bar{\bar w}^2}},\\
                                                                                y&=&y\\
                                                                                w&=&\frac{\rho \varrho^2 \bar{\bar w}}{\sqrt{1+\varrho^2\bar{\bar x}^2+\varrho^4 \bar{\bar w}^2}},\\
                                                                                q&=&\frac{r}{\sqrt{1+\frac{\rho^2}{1+\varrho^2\bar{\bar x}^2+\varrho^4 \bar{\bar w}^2}}},\\
                                                                                \epsilon&=& \frac{r\rho \frac{1}{\sqrt{1+\varrho^2\bar{\bar x}^2+\varrho^4 \bar{\bar w}^2}}}{\sqrt{1+\frac{\rho^2}{1+\varrho^2\bar{\bar x}^2+\varrho^4 \bar{\bar w}^2}}}.
                                                                               \end{matrix} \right.\eqlab{Psi13}
\end{align}

\begin{figure}[h!]
\begin{center}
\subfigure[]{\includegraphics[width=.695\textwidth]{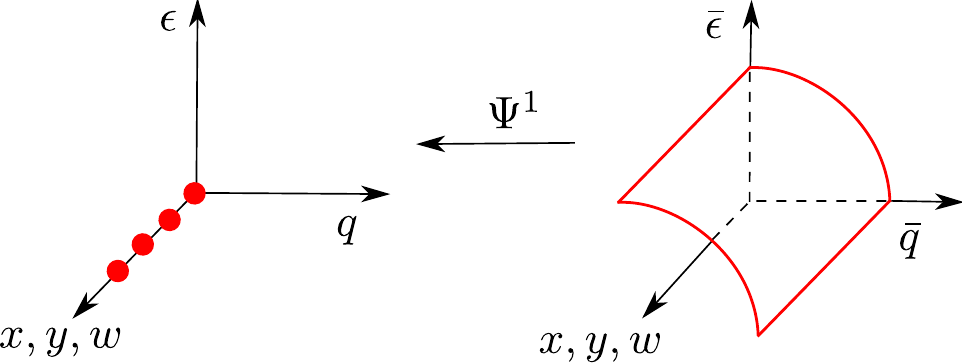}}
\subfigure[]{\includegraphics[width=.995\textwidth]{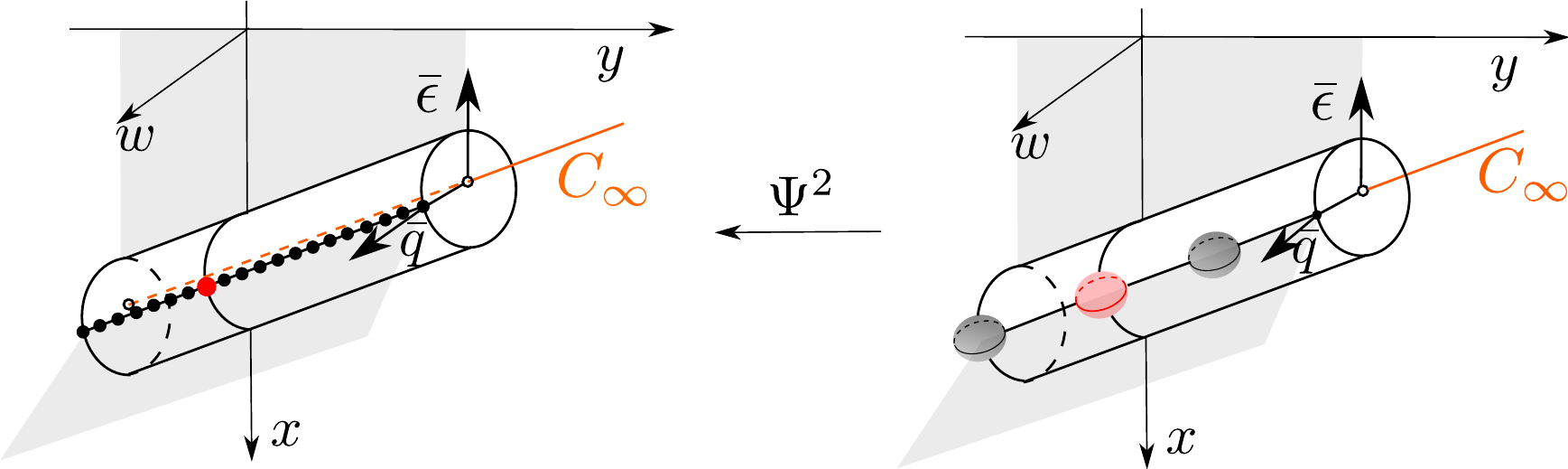}}
\subfigure[]{\includegraphics[width=.995\textwidth]{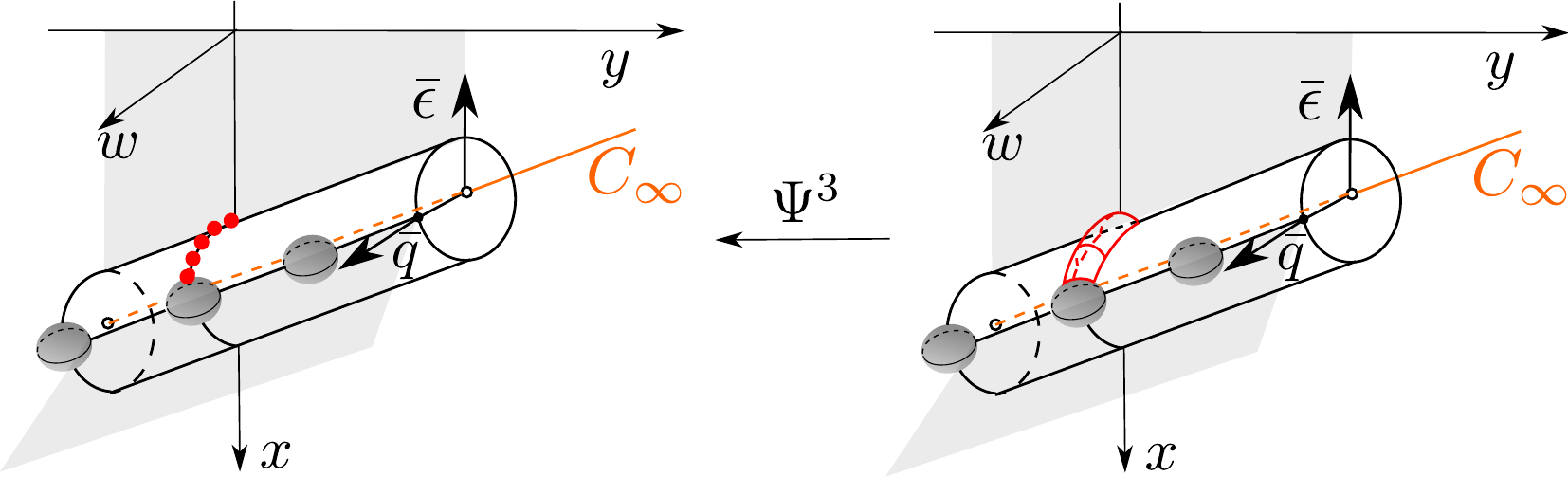}}
\end{center}
\caption{(a) First blowup of $q=\epsilon=0$. (b) Second blowup of $C_\infty$ along $\bar q=1$. We blowup $\bar q=1$, $w=0$, $x=-1-\xi y$ to a line of spheres. (c) Third blowup step. Here we blowup the north pole of the sphere obtained in the second step. This gives rise to a cylinder, its axis being formed by the quarter circle with $\bar q\ge 0,\bar \epsilon\ge 0$.}
\figlab{step1}
\end{figure}

In the fourth step, we work on $P^1$ near $(\bar q,\bar \epsilon)=(0,1)$. Notice that this implies $\rho$ large in \eqref{step2}. We therefore proceed as follows in two steps (enumerated $a$ and $b$). Let
\begin{align*}
 P^{4a} = \left\{(y,r,\sigma,(\tilde x,\tilde w),(\bar q,\bar \epsilon))\in \mathbb R\times [0,\infty)^2 \times S^1\times S^1\right\}.
\end{align*}
Then we first blowup $x=-1-\xi y,\,w=0$ through the blowup transformation $$\Psi^{4a}:P^{4a}\rightarrow P^1,$$ which fixes $r$ and $(\bar q,\bar \epsilon)$ and takes
\begin{align}
 (y,\sigma, (\tilde x,\tilde w))\mapsto \left\{\begin{matrix} x &= &-1-\xi y +\sigma \tilde x,\\
 w &=& \sigma^2 \tilde w.
 \end{matrix}\right.
 \eqlab{step4a}
\end{align}
Crucially, the exponents of $\sigma$ in \eqref{step4a} coincide with the exponents on $\varrho$ in \eqref{step3}. For $r=0$,  \eqref{step4a} is still a blowup of $C_\infty$. We illustrate the blowup in \figref{step2}(a). 
$\Psi^{4a}$ gives a vector-field $\overline X^{4a}=\Psi^{4a*}(\widehat X^1)$ on $P^{4a}$ by pull-back of $\widehat X^1$ on $P^1$. Here $\overline X^{4a}=\sigma \widehat X^{4a}$, with $\widehat X^{4a}$ well-defined. It is $\widehat X^{4a}$ that we shall study. 
%
%
%
%
%
%
%
%
%
%
%
%
%
%
%
%
%

Next, let 
\begin{align*}
 P^{4b}=\left\{(y,r,\sigma,\pi,(\tilde{\tilde w},\tilde{\bar q}))\in \mathbb R\times [0,\infty)^3\times S^1\right\}.
\end{align*}
Then we blowup $\tilde x=-1,\,\tilde w=0,\,\bar \epsilon^{-1}\bar q=0$ within $P^{4a}$ for each $y$ through the blowup transformation $\Psi^{4b}:P^{4b}\rightarrow P^{4a}$ which fixes $y$, $r$ and $\sigma$ and takes
\begin{align}
 (\pi,(\tilde{\tilde w},\tilde{\bar q}))\mapsto \left\{ \begin{matrix} \tilde x^{-2} \tilde w &=& \pi \tilde{\tilde w},\\
 \bar \epsilon^{-1} \bar q &=& \pi \tilde{\bar q},
 \end{matrix}\right. \quad \pi\ge 0, (\tilde{\tilde w},\tilde {\bar q})\in S^1.\eqlab{step4b}
\end{align}
We illustrate the final blowup in \figref{step2}(b) and in \figref{step2}(c) using the viewpoint of \figref{step1}. (See also \lemmaref{MMap} below).

\begin{figure}[h!]
\begin{center}
\subfigure[]{\includegraphics[width=.995\textwidth]{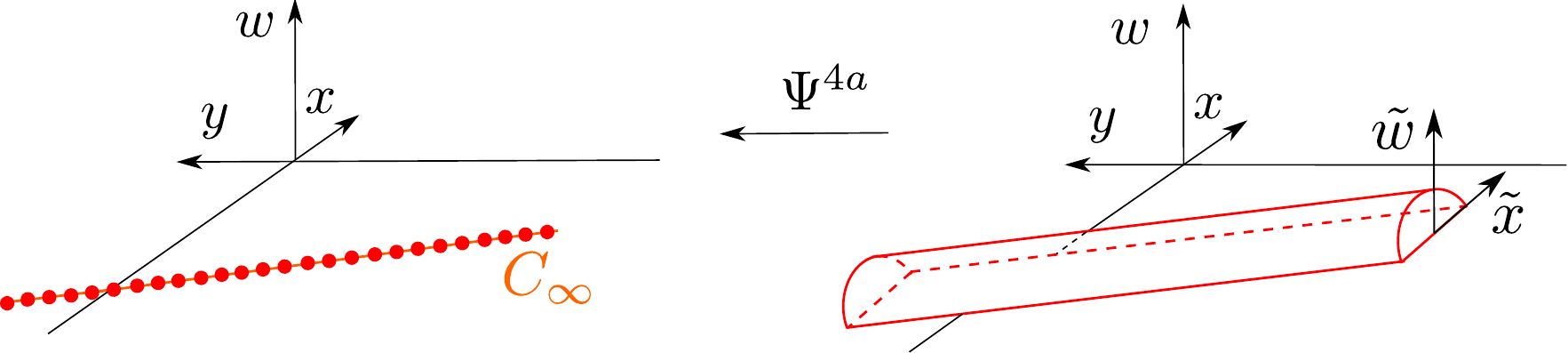}}
\subfigure[]{\includegraphics[width=.995\textwidth]{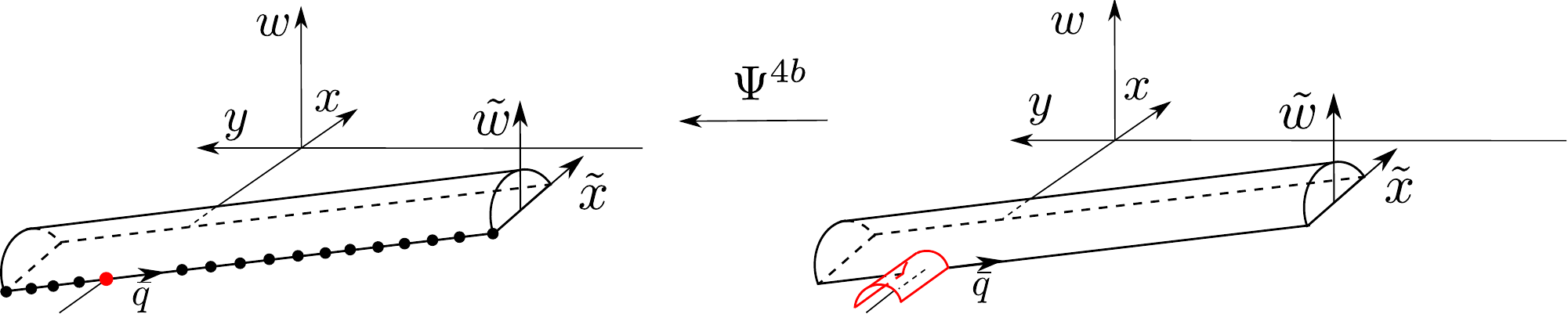}}
\subfigure[]{\includegraphics[width=.995\textwidth]{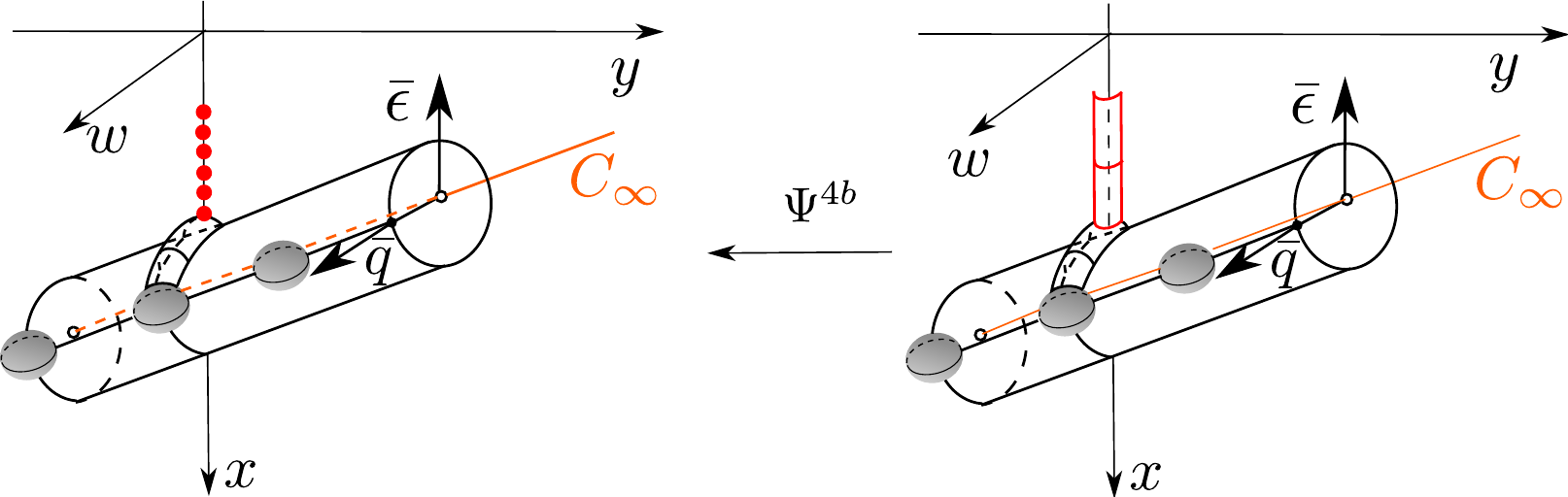}}
\subfigure[]{\includegraphics[width=.995\textwidth]{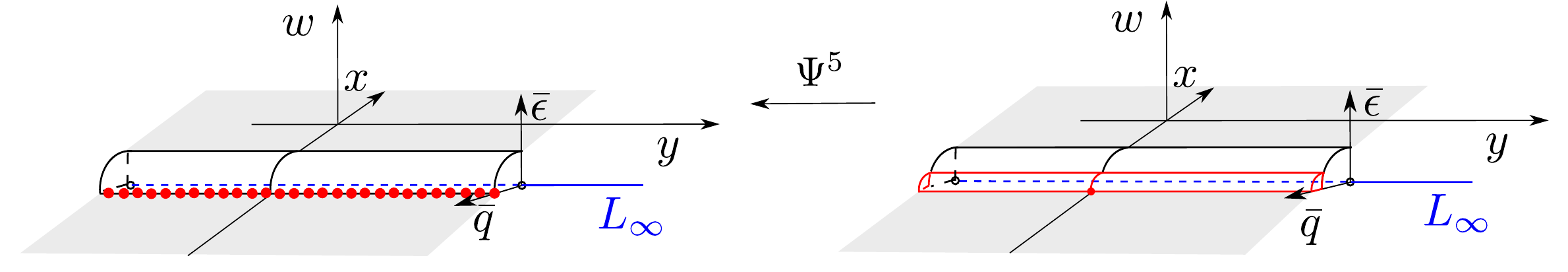}}
\end{center}
\caption{(a) Fourth blowup step, part a, blowing up $C_\infty$ to a cylinder. (b) Fourth blowup step, part b. Here we blow up $\bar q=0$, $\tilde w=0$ to a circle, producing a cylinder along the negative $x$-direction. (c) Similar to (b), but now using the viewpoint in \figref{step1}. (d) Fifth blowup step near $L_\infty$ where this blowup is used. Here we blowup $w=\bar q=0$ to a circle. We indicate two important planes $\bar \epsilon=w=0$ and $\bar q=w=0$ by gray shading. }
\figlab{step2}
\end{figure}
\begin{remark}
Notice that since $(\tilde x,\tilde w)\in S^1$ and $(\bar q,\bar \epsilon)\in S^1$ it follows from simple calculations that the right hand side of \eqref{step4b} can be written as
\begin{align*}
 (\tilde x,\tilde w) &= \left(\chi\left(\pi\tilde{\tilde w}\right),\chi\left(\pi \tilde{\tilde w}\right)^2 \pi \tilde{\tilde w}\right),\\
 (\bar q,\bar \epsilon) &=\left(\frac{\pi \tilde{\bar q}}{\sqrt{1+\pi^2 \tilde{\bar q}^2}},\frac{1}{\sqrt{1+\pi^2 \tilde{\bar q}^2}}\right),
\end{align*}
where $\chi:\mathbb R\rightarrow (-1,0)$ is the unique, negative-valued, smooth function 
\begin{align*}
 \chi: \mathbb R\rightarrow (-1,0),\quad \chi (p) = 
  \left\{\begin{matrix}
                -1& \text{if}\, p=0\\
                 -\frac{\sqrt{\sqrt{4p^2+1}-1}}{\sqrt{2}\vert p\vert} &\text{otherwise}
               \end{matrix}\right.,
\end{align*}
satisfying $\chi(p)^2+\chi(p)^4p^2=1$. 
\end{remark}
$\Psi^{4b}$ gives $\overline X^{4b}=\Psi^{4b*}(\widehat X^{4a})$ on $P^{4b}$ by pull-back of $\widehat X^{4a}$ on $P^{4a}$. Now, $\overline X^{4b}=\pi \widehat X^{4b}$ and it is $\widehat X^{4b}$ that we study.

We now define 
\begin{align*}
\Psi^{14a}:P^{4a}\rightarrow P,\quad 
 \Psi^{14a4b}:P^{4b}\rightarrow P,
\end{align*}
as the compositions
\begin{align*}
\Psi^{14a} &= \Psi^1\circ \Psi^{4a},\quad
 \Psi^{14a4b} = \Psi^1\circ \Psi^{4a}\circ \Psi^{4b}.
\end{align*}
Therefore by \eqref{step1}, \eqref{step4a} and \eqref{step4b}
\begin{align}
 \Psi^{14a}:\, (y,r,\sigma,(\tilde x,\tilde w),(\bar q,\bar \epsilon))&\mapsto  \left\{\begin{matrix}
                                                                                     x&=&-1-\xi y+\sigma \tilde x,\\
                                                                                     y&=&y\\
                                                                                     w&=&\sigma^2\tilde w,\\
                                                                                     q&=&r                                                                            \bar q,\\
                                                                                     \epsilon &=&r\bar \epsilon,
                                                                                    \end{matrix}\right.\nonumber\\
\Psi^{14a4b}:\, (y,r,\sigma,\pi,(\tilde{\tilde w},\tilde{\bar q}))&\mapsto  \left\{\begin{matrix}
                                                                                              x&=&-1-\xi y+\sigma
                                                                                              \chi\left(\pi\tilde{\tilde w}\right),\\
                                                                                              y&=&y\\
                                                                                              w&=&\sigma^2 \chi\left(\pi \tilde{\tilde w}\right)^2 \pi \tilde{\tilde w},\\
                                                                                              q&=&\frac{r\pi \tilde{\bar q}}{\sqrt{1+\pi^2 \tilde{\bar q}^2}},\\
                                                                                              \epsilon &=&\frac{r}{\sqrt{1+\pi^2 \tilde{\bar q}^2}}. \end{matrix}\right.\eqlab{Psi14b}
\end{align}

\begin{lemma}\lemmalab{MMap}
Let $U^3=\{(y,r,\rho,\varrho,(\bar{\bar x},\bar{\bar w}))\in P^3 \vert \rho>0,\varrho>0,\bar{\bar x}<0,\,\bar{\bar w}>0\}$. Then there exists an injective mapping
 $M:U^3\rightarrow P^{4b}$ such that
 \begin{align*}
  \Psi^{13}\vert_{U^3} = \Psi^{14a4b}\circ M.
 \end{align*}
\end{lemma}
\begin{proof}
 Clearly, $M$ fixes $y$ and $r$ and takes $$(r,\rho,\varrho,(\bar{\bar x},\bar{\bar w}))\mapsto (\sigma,\pi,(\tilde{\tilde w},\tilde{\bar q})).$$ We solve for $(\sigma,\pi,(\tilde{\tilde w},\tilde{\bar q}))$ directly using \eqref{Psi13} and \eqref{Psi14b}. This gives, 
 \begin{align}
 x^{-2} w &={\sqrt{1+\varrho^2\bar{\bar x}^2+\varrho^4 \bar{\bar w}^2}} \rho^{-1}\bar{\bar x}^{-2}\bar{\bar w} = \pi\tilde{\tilde w},\eqlab{MEqn1}\\
 \epsilon^{-1} q &={\sqrt{1+\varrho^2\bar{\bar x}^2+\varrho^4 \bar{\bar w}^2}} \rho^{-1} \quad \quad \,\,= \tilde{\bar q} \pi ,\nonumber
 \end{align}
 the first set of equalities due to \eqref{Psi13}, the latter ones due to \eqref{Psi14b}. 
  Therefore by division
  \begin{align*}
   \tilde{\bar q}^{-1} \tilde{\tilde w} = \bar{\bar x}^{-2} \bar{\bar w},
  \end{align*}
and hence we obtain a unique $(\tilde{\bar q},\tilde{\tilde w})\in S^1$ with $\tilde{\bar q}>0,\tilde{\tilde w}> 0$ for every $(\bar{\bar x},\bar{\bar w})\in S^1$ with $\bar{\bar x}>0$ and $\bar{\bar w}>0$.
 From here $\pi$ can be determined by
\begin{align*}
 \pi  &=\tilde{\tilde w}^{-1} {\sqrt{1+\varrho^2\bar{\bar x}^2+\varrho^4 \bar{\bar w}^2}} \rho^{-1}\bar{\bar x}^{-2}\bar{\bar w},
\end{align*}
using \eqref{MEqn1}.
Finally,
\begin{align*}
 \sigma = \chi(\pi\tilde{\tilde w})^{-1}{\rho\varrho \bar{\bar x}}/{\sqrt{1+\varrho^2\bar{\bar x}^2+\varrho^4 \bar{\bar w}^2}}.
\end{align*}
Similar calculations gives the inverse of $M$ on $$M(U^3) = \{(y,r,\sigma,\pi,(\tilde{\tilde w},\tilde{\bar q}))\vert \pi>0,\,\tilde{\tilde w}>0,\,\tilde{\tilde q}>0\}.$$
\end{proof}
This result means that the diagram in \figref{MDiagram} commutes and that we can study $\widehat X^3$ on $U^3$ using $\widehat X^{4b}$ on $M(U^3)$ since $\overline X^{4b}=M_*(\overline X^{4})$ there. The latter property is important for connecting results for $\widehat X^3$ on $P^3$ with results for $\widehat X^{4b}$ on $P^{4b}$.

\begin{figure}[h!]
\begin{center}
{\includegraphics[width=.575\textwidth]{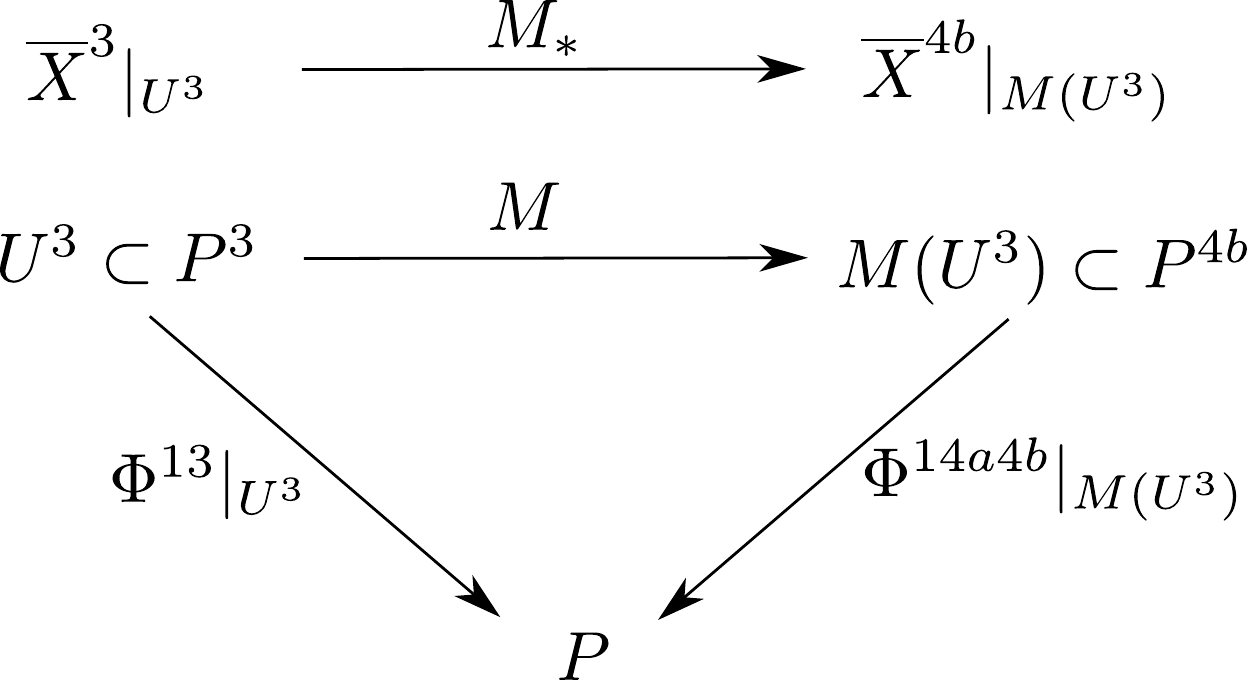}}
\end{center}
\caption{Commutative diagram.}
\figlab{MDiagram}
\end{figure}

In the analysis of the fourth blowup, we will find that $\sigma$ eventually increases while $\pi$ remains small. To cover this part, where $C_\infty$ plays no role, it is easiest to skip the first part of the fourth blowup \eqref{step4a}, see also local form in \eqref{step4aLocal} below, and just do a polar blowup of $w=0,(\bar \epsilon,\bar q)=(1,0)$ as follows: 
\begin{align*}
\Psi^5:\,(\mu,(\tilde w,\tilde{\bar q}))\mapsto (w,\bar \epsilon^{-1}\bar q)=\mu (\tilde w,\tilde{\bar q}), \quad \mu\ge 0, \,(\tilde w,\tilde{\bar q})\in S^1,
\end{align*}
fixing $x$, $y$ and $r_2$. Here $\Psi^5:P^5\rightarrow P^1$ where
\begin{align*}
 P^5 = \{(x,y,r_2,\mu,(\tilde w,\tilde{\bar q}))\in \mathbb R^2 \times [0,\infty]^2 \times S^1\}. 
\end{align*}
We put $\Psi^{15}=\Psi^1\circ \Psi^5$.  We illustrate this final blowup in \figref{step2}(d)  near $L_\infty$. 
%
%
%

\subsection{Charts}\seclab{chartsK3}
We use separate directional charts to describe the blowup transformations defined in the previous section. For the first blowup $\Psi^1$, for example, we will use two separate charts obtained by central projections onto the planes $\bar q=1$ and $\bar \epsilon=1$, respectively. We call these charts $(\bar q=1)_1$ and $(\bar \epsilon=1)_2$, respectively. The mapping from local coordinates to $(q,\epsilon)$ is obtained by setting $\bar q=1$ and $\bar \epsilon=1$, respectively, in \eqref{step1}. These charts therefore give the following local forms of the blowup $\Psi^1$:
\begin{align}
  \Psi^1_{1}: (r_1,\epsilon_1)&\mapsto \left\{\begin{matrix} q&=& r_1,\\\epsilon &=& r_1\epsilon_1,\end{matrix}\right.\eqlab{Psi11}\\
 \Psi^1_{2}: (r_2,q_2)&\mapsto \left\{\begin{matrix}  q&=&r_2q_2,\\ 
 \epsilon&=&r_2.
 \end{matrix}\right.\eqlab{Psi12}
\end{align}
where $(x,y,w,r_1,\epsilon_1)$ and $(x,y,w,r_2,q_2)$ are the local coordinates in the two charts. We can change coordinates between these charts through the following expressions:
\begin{align}
r_2 &= r_1\epsilon_1,\eqlab{firstcc}\\
 q_2 &= \epsilon_1^{-1},\nonumber
\end{align}
for $\epsilon_1>0$.
For the second blowup $\Psi^2$, described by the equations \eqref{step2}, we work in the chart $(\bar q=1)_1$ such that $\bar q^{-1}\bar\epsilon=\epsilon_1\ge 0$. Subsequently we then use local charts to describe $(\bar x,\bar w,\bar{\bar \epsilon})\in S^2$ by setting $\bar w=1$, $\bar{\bar{\epsilon}} =1$ and finally $\bar x=1$. We refer to each of these local charts as $(\bar q=1,\bar w=1)_{11}$, $(\bar q=1,\bar{\bar \epsilon}=1)_{12}$ and $(\bar q=1,\bar x=1)_{13}$, respectively. They produce the following local forms of the second blowup $\Psi^{12}=\Psi^1\circ \Psi^2$:
\begin{align}
 \Psi^{12}_{11}:(y,r_1,\rho_1,x_1,\epsilon_{11})\mapsto &\left\{\begin{matrix}  x &=&-1-\xi y +\rho_1 x_1,\\ 
  w &=&\rho_1\\
  q&=&r_1,\\
                                      \epsilon &=& r_1 \rho_1 \epsilon_{11},\end{matrix}\right.\eqlab{Psi1211}\\
 \Psi^{12}_{12}:(y,r_1,\rho_2,x_2,w_2)\mapsto & \left\{\begin{matrix} 
 x&=&-1-\xi y+\rho_2 x_2,\\
 w&=&\rho_2w_2,\\
q &=&r_1,\\
\epsilon &=& r_1 \rho_2,
\end{matrix}\right.\eqlab{Psi1212},
\end{align}
using $(y,r_1,\rho_1,x_1,\epsilon_{11})$ and $(y,r_1,\rho_2,x_2,w_2)$, as the local coordinates in these charts $(\bar q=1,\bar w=1)_{11}$, $(\bar q=1,\bar{\bar \epsilon}=1)_{12}$, respectively. We can change coordinates between $ (\bar q=1,\bar w=1)_{11}$ and $ (\bar q=1,\bar{\bar \epsilon}=1)_{12}$ through the following expressions:
\begin{align}
 \rho_2 &=\rho_1 \epsilon_{11},\eqlab{ccK31112}\\
 x_2 &=\epsilon_{11}^{-1} x_1,\nonumber\\
 w_2 &=\epsilon_{11}^{-1},\nonumber
\end{align}
for $\epsilon_{11}>0$. We summarize the information about the charts used for the first two blowups in \tabref{tbl1}.
   \begin{table}[h]
    \renewcommand\arraystretch{2}
\begin{tabular}{|c|c|c|c|c|}
\hline
    & \multicolumn{2}{c|}{1st blowup} & \multicolumn{2}{c|}{2nd blowup} 
        \\
    \hline    
        Charts & $(\bar q=1)_1$ & $(\bar \epsilon=1)_2$ & $(\bar q=1,\bar w=1)_{11}$ &  $(\bar q=1,\bar{\bar \epsilon}=1)_{12}$\\
    \hline
Coordinates & $(x,y,w,r_1,\epsilon_1)$ &$(x,y,w,q_2,r_2)$ & $(y,r_1,\rho_1,x_1,\epsilon_{11})$ &$(y,r_1,\rho_2,x_2,w_{2})$\\
\hline 
Local Blowup & $\Psi^1_1$ \eqref{Psi11} &$\Psi^1_2$ \eqref{Psi12} & $\Psi^{12}_{11}$ \eqref{Psi1211}& $\Psi^{12}_{12}$ \eqref{Psi1212}\\
\hline
Equations & \eqref{xyweps1}, \secref{exitPhi3} &   \done &  \eqref{eqn211}, \secref{sec11eqns}&  \eqref{chartBarEpsBarQ}, \secref{sec12eqns}\\
\hline 
Coordinate changes & \multicolumn{2}{c|}{\eqref{firstcc}} & \multicolumn{2}{c|}{\eqref{ccK31112}} \\
\hline
\end{tabular}
\caption{Details about the charts used for the first two blowups. The second to last row (``Equations'') contains the equation numbers of the local forms of the desingularized vector-fields, and the corresponding section numbers where these systems are analyzed. The last row (``Coordinate changes'') contains the equation numbers for the coordinate changes between the corresponding columns.}
\tablab{tbl1}
    \end{table}

For the third blowup $\Psi^{3}$, we work in the chart $(\bar q=1,\bar{\bar \epsilon}=1)_{12}$ where 
\begin{align*}
 \bar{\bar \epsilon}^{-1} \bar x=x_2,\quad  \bar{\bar \epsilon}^{-1} \bar w=w_2.
\end{align*}
Then we plug in $\bar{\bar w}=1$ into \eqref{step3} and obtain the chart $(\bar q=1,\bar{\bar \epsilon}=1,\bar{\bar w}=1)_{122}$, respectively. Within this charts we obtain the following local form of the blowup $\Psi^{123}=\Psi^1\circ \Psi^2\circ \Psi^3$
\begin{align}
 \Psi^{123}_{122}:\,(y,r_1,\rho_2,\varrho_2,x_{22})\mapsto &\left\{\begin{matrix} x&=&-1-\xi y +\rho_2 \varrho_2 x_{22},\\w&=&\rho_2 \varrho_2^2\\
%
q&=&r_1,\\
\epsilon&=&r_1\rho_2,\end{matrix}\right.\eqlab{Psi13Charts}
\end{align}
using $(y,r_1,\rho_2,\varrho_2,x_{22})$ as local coordinates. 

For the fourth blowup $\Psi^4$, we first work in the chart $(\bar \epsilon=1)_2$. Then we plug in $\tilde x=-1$ into \eqref{step4a} to obtain a chart for the description of $(\tilde x,\tilde w)\in S^1$ in a neighborhood of $(\tilde x,\tilde w)=(-1,0)$. This produces the local chart $(\bar \epsilon=1,\tilde x=-1)_{21}$ in which $\Psi^{14a}=\Psi^1\circ \Psi^{4a}$ takes the following local form
\begin{align}
\Psi_{21}^{14a}:\,(y,r_2,\sigma_1,w_1,q_2)\mapsto &\left\{\begin{matrix} 
%
x&=&-1-\xi y -\sigma_1,\\
w&=&\sigma_1^2w_1,\\
q&=&r_2q_2 ,\\
\epsilon&=&r_2,\end{matrix}\right. \eqlab{step4aLocal}
\end{align}
using $(y,r_2,\sigma_1,w_1,q_2)$ as coordinates in this chart. Within $(\bar \epsilon=1,\tilde x=-1)_{21}$ we have
\begin{align*}
\tilde x^{-2} \tilde w&=w_1,\\
\bar \epsilon^{-1}\bar q&=q_2
\end{align*}
and therefore \eqref{step4b} becomes 
\begin{align*}
 w_1 &= \pi \tilde{\tilde w},\\
 q_2 &=\pi \tilde{\bar q}.
\end{align*}
We therefore plug in $\tilde{\bar q}=1$ and obtain the chart $(\bar \epsilon=1,\tilde x=-1,\tilde{\bar q}=1)_{211}$ and the following local form of $\Psi^{14a4b}=\Psi^1\circ \Psi^{4a}\circ \Psi^{4b}$:
\begin{align}
 \Psi_{211}^{14a4b}:\,(y,r_2,\sigma_1,\pi_1,w_{11})\mapsto&\left\{\begin{matrix}x&=&-1-\xi y -\sigma_1,\\ w&=&\sigma_1^2\pi_1 w_{11},\\ q&=&r_2\pi_1 ,\\ \epsilon&=&r_2,\end{matrix}\right.\eqlab{Psi21114a14b}
\end{align}
using $(y,r_2,\sigma_1,\pi_1,w_{11})$ as local coordinates.

Following \lemmaref{MMap}, we can change coordinates between 
$(\bar \epsilon =1,\tilde x=-1,\tilde{\bar q}=1)_{211}$ and $(\bar q=1,\bar{\bar \epsilon}=1,\bar{\bar{w}} = 1)_{122}$ through the following expressions:
\begin{align}
\pi_1&=\rho_2^{-1},\eqlab{cc3bto4}\\
r_2 &=r_1\rho_2,\nonumber\\
\sigma_1 &=\rho_2\varrho_2x_{22},\nonumber\\
w_{11}&=x_{22}^{-1}.\nonumber
\end{align}
for $\rho_2>0$ and $x_{22}>0$.

We describe the fifth blowup transformation $\Psi^5$ using the chart $(\bar \epsilon=1,\tilde w=1)_{21}$ and $(\bar \epsilon=1,\tilde{\bar q}=1)_{22}$ such that $\Psi^{15}=\Psi^1\circ \Psi^5$ becomes
\begin{align}
\Psi_{21}^{15}:(r_2,\mu_1,q_{21})&\mapsto \left\{\begin{matrix} q&=&r_2\mu_1q_{21},\\ \epsilon&=&r_2,\\ w&=&\mu_1,\end{matrix}\right.\eqlab{step5}\\
\Psi_{22}^{15}:(r_2,\mu_2,w_{2})&\mapsto \left\{\begin{matrix} q&=&r_2\mu_2,\\ \epsilon&=&r_2,\\ w&=&\mu_2w_2,\end{matrix}\right.\eqlab{step5b}
\end{align}
in the local coordinates $(x,y,r_2,\mu_1,q_{21})$ and $(x,y,r_2,\mu_2,w_{2})$, respectively.  Notice, that we can change coordinates between $(\bar \epsilon=1,\tilde w=1)_{21}$ and $(\bar \epsilon =1,\tilde x=-1,\tilde{\bar q}=1)_{211}$ through the following expressions
\begin{align}
\mu_1 &= \sigma_1^2\pi_1 w_{11},\eqlab{cc4to5}\\
q_{21}&=\sigma_1^{-2} w_{11}^{-1},\nonumber\\
x &=-1-\xi y-\sigma_1.\nonumber
\end{align}
Also, between $(\bar \epsilon=1,\tilde w=1)_{21}$ and $(\bar \epsilon=1,\tilde{\bar q}=1)_{22}$ we have the following equations
\begin{align}
\mu_2 &=\mu_1 q_{21},\eqlab{fifthcc}\\
 w_2 &=q_{21}^{-1}. \nonumber
\end{align}

We summarize the information about the charts used for the third, fourth and fifth blowup in \tabref{tbl2}.
   \begin{table}[h]
    \renewcommand\arraystretch{2}
\begin{tabular}{|c|c|c|c|}
\hline
  {3rd blowup} & 4th blowup, part b & \multicolumn{2}{c|}{5th blowup}
        \\
    \hline    
     $(\bar q=1,\bar{\bar \epsilon}=1,\bar{\bar x}=-1)_{122}$ & $(\bar \epsilon=1,\tilde x=-1,\tilde{\bar q}=1)_{211}$ &  $(\bar \epsilon=1,\tilde w=1)_{21}$ &  $(\bar \epsilon=1,\tilde{\bar q}=1)_{22}$\\
    \hline
$(y,r_1,\rho_2,\varrho_2,x_{22})$ & $(y,r_2,\sigma_1,\pi_1,w_{11})$ &$(x,y,r_2,\mu_1,q_{21})$ & $(x,y,r_2,\mu_2,w_2)$\\
\hline 
$\Psi^{13}_{122}$ \eqref{Psi13Charts} & $\Psi^{14a14b}_{211}$ \eqref{Psi21114a14b}& $\Psi^{15}_{21}$ \eqref{step5} &$\Psi^{15}_{22}$ \eqref{step5b} \\
\hline 
\eqref{eqn122Here}, \secref{sec122eqns} & \eqref{eqeqns211}, \secref{seceqns211} & \eqref{eqeqns21}, \secref{seceqns21} &\eqref{eqeqns22}, \secref{seceqns22} \\
\hline
 \multicolumn{2}{|c|}{\eqref{cc3bto4}} & \multicolumn{2}{c|}{\eqref{fifthcc}} \\
\hline
 \done  & \multicolumn{2}{c|}{\eqref{cc4to5}} & \done\\
\hline
\end{tabular}
\caption{Details about the charts used for the third, fourth and fifth blowup. The rows have the same meaning as in \tabref{tbl1}. In particular, the last two rows contain the equation numbers for the coordinate changes between the corresponding columns. }
\tablab{tbl2}
    \end{table}
 
%

\subsection{Summary of results in $\phi_3$}
The full details of the analysis of the blowup systems in chart $\phi_3$ is available in \secref{detailsK3}. Here we will try to summarize the findings. 

By our blowup approach we obtain improved hyperbolicity properties of parts of the singular cycle visible in the chart $\phi_3$. In doing so, we also identify segments that are only visible upon blowup. We illustrate all the segments in \figref{gamma} using the viewpoint in \figref{step2} (c) and \figref{step2}(d). In particular, $\gamma^1$ is a heteroclinic connection on the sphere $(\bar x,\bar w,\bar{\bar \epsilon})\in S^2$ obtained from the second blowup $\Psi^2$, see also \figref{step1}(b). By standard hyperbolic methods, we can guide a neighborhood of $W^c(Q^6)$ close to $\gamma^1$. In fact, we show that the contraction of the slow flow on $C$ towards $Q^1$ produces a contraction towards $\gamma^1$ for $\epsilon\ll 1$. In turn, this provides the contraction of the return mapping $\Pi=\Pi^1\circ \Pi^0$, which we use to prove the existence of an attracting limit cycle. 

By the third blowup, we gain hyperbolicity of the forward limit point of $\gamma^1$ and subsequently follow a $1D$ unstable manifold $\gamma^2$ towards $(\bar q,\bar \epsilon)=(0,1)$. We gain hyperbolicity of the forward limit point of $\gamma^2$ by the fourth blowup transformation and follow an unstable manifold $\gamma^3$. Along $\gamma^3$, $x$ is decreasing towards the center-like manifold $L_\infty$ at $x=-1-\alpha$, $w=0,y\in I$, recall \eqref{Lhere}. On this center manifold, we desingularize the slow flow and follow $\gamma^4$. Along $\gamma^4$, $y$ is increasing, recall also \remref{ElenaArgument}. At $y=2\alpha/\xi$, $\gamma^4$ ends along a line of equilibria of saddle-structure. We subsequently follow the unstable manifold $\gamma^5$, along which $\bar q$ is increasing. By the fifth blowup, we gain hyperbolicity of the forward limit point of $\gamma^5$ and subsequently follow an unstable manifold $\gamma^6$. $\gamma^6$ is asymptotic to a center-like manifold. Upon desingularization we obtain a slow flow which produces $\gamma^7$. $\gamma^7$ is asymptotic to $C_\infty$, but this part is better described in chart $\phi_1$, see \secref{phi1Section}.

In conclusion, we obtain the following: let $\Pi^{17}$ be the mapping obtained by the first intersection of the forward flow from 
\begin{align}
\Sigma^1=\{(x,y,w)\vert w = \delta,\,x-1-\xi y\in [-\beta_1,0),y\in [-\beta_2,\beta_2]\},\eqlab{Sigma1Here}
\end{align}
to 
\begin{align}
 \Sigma^{7}_1 = \{(x,y,w,\epsilon_1)\vert y = 2\alpha(1+ \nu)/\xi ,\,x+1+\alpha \in [-\beta_3,\beta_3],\, w\in [0,\beta_4],\,\epsilon_1\in[0,\beta_5]\}.\eqlab{Sigma7}
\end{align}
Here $\nu>0$. Notice the following:
\begin{itemize}
 \item We restrict $\Sigma^1$ to $x-1-\xi y\in [-\beta_1,0)$ so that the flow is transverse to the section, see \eqref{phi3Eqs}. This is clearly a subset containing $\Pi^0(\Sigma^0)$, recall \eqref{Pi1Eqn}.
  \item We use the coordinates $(x,y,w,\epsilon_1)$ in chart $(\bar q=1)_1$ to describe the image of $\Pi^{17}$ in $\Sigma^7_1$. Using \eqref{q3K3} and \eqref{Psi11} we have
\begin{align*}
 \epsilon_1  = \epsilon e^{2w^{-1}}.
\end{align*}
By describing the image in these variables, we therefore at the same time keep track of how small $w$ is. If $w$ were to be too small then $\epsilon_1$ would not be be small enough for us to compose it with the subsequent mapping $\Pi^{71}$, see \lemmaref{Pi70}. 
\end{itemize}
We then have
\begin{lemma}\lemmalab{Pi17}
 The mapping $\Pi^{17}$ is well-defined for appropriately small $\delta>0$, $\nu>0$ and $\beta_i>0$, $i=1,\ldots, 5$ and all $0< \epsilon\ll 1$. In particular,
 \begin{align*}
  \Pi^{17}(x,y,\delta;\epsilon) = (x_+(x,y,\epsilon),2\alpha(1+\nu)/\xi,w_+(x,y;\epsilon),\epsilon_{1+}(x,y;\epsilon)),
 \end{align*}
where $x_+$, $w_+$ and $\epsilon_{1+}(x,y;\epsilon)$ are $C^1$-functions in $x$ and $y$, satisfying the following $C^1$-estimates
\begin{align*}
 x_+(x,y;\epsilon) &= -(1+\alpha)(1+\nu)+\mathcal O(\log^{-1} \epsilon^{-1} \log \log \epsilon^{-1}),\\
 w_+(x,y;\epsilon) &=\mathcal O(\log^{-1} \epsilon^{-1} \log \log \epsilon^{-1}),\\
 \epsilon_{1+}(x,y;\epsilon) &=\mathcal O(e^{-c\log \epsilon^{-1}}),
\end{align*}
for $c>0$ sufficiently small, 
as $\epsilon\rightarrow 0$. 
\end{lemma}
\begin{proof}
 The result follows from a series of lemmas (\lemmaref{Pi111}, see also \corref{Pi111Cor}, \lemmaref{Pi22}, \lemmaref{Pi2113}, \lemmaref{Pi2145}, \lemmaref{Pi226}, \lemmaref{Pi17New}) working in the local charts described in \secref{chartsK3}, and applying standard hyperbolic methods to follow the segments $\gamma^1-\gamma^7$ described above. Notice that the mappings between the different local sections are diffeomorphism that do not change the order See details in \secref{detailsK3}.
\end{proof}


\begin{figure}[h!]
\begin{center}
{\includegraphics[width=.995\textwidth]{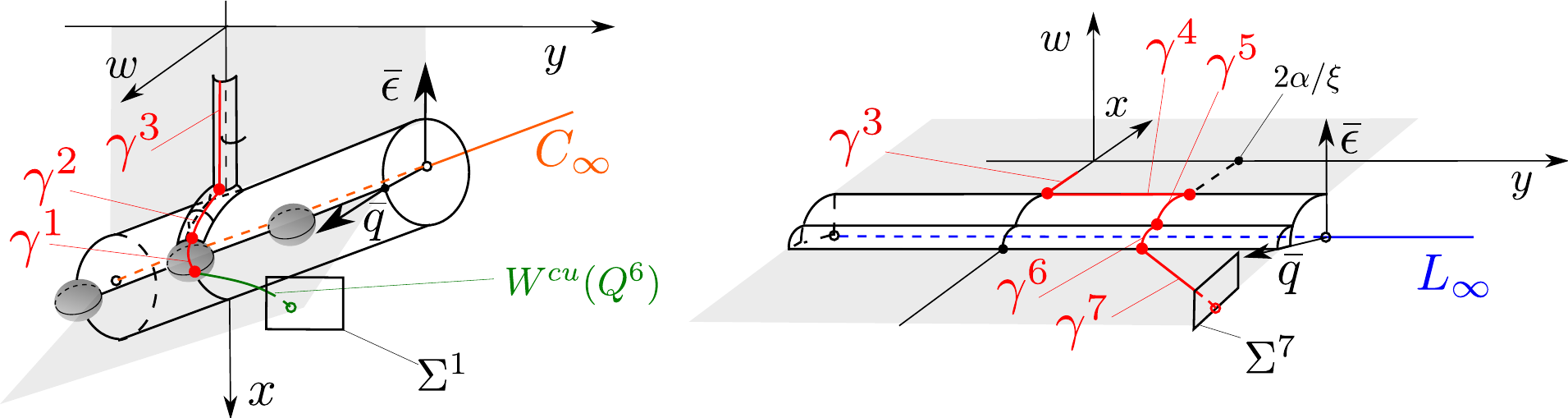}}
\end{center}
 \caption{Parts of the blown up singular cycle visible in chart $\phi_3$. The improved hyperbolicity properties allow us to prove \lemmaref{Pi17}. Notice that, following the viewpoint in \figref{step2}(d), $\gamma^3$ and $\gamma^4$ are contained within the plane $\bar q=w=0$, whereas $\gamma^7$ is contained within $\bar \epsilon=w=0$. $\gamma^5$ and $\gamma^6$ connect these orbit segments. Along these orbits, $\bar \epsilon$ is therefore decreasing. See further details in \secref{detailsK3}.}
\figlab{gamma}
\end{figure}


\section{Chart $\phi_1$}\seclab{phi1Section}
In this chart we obtain the following equations
\begin{align}
 \dot w_1 &=-\epsilon w_1^2 F(z_1 w_1^{-1}),\eqlab{phi1Eqns}\\
 \dot x_1 &=-\epsilon \left( w_1x_1 F(z_1 w_1^{-1})+\left(x_1+(1+\alpha)z_1 \right)\right),\nonumber\\
 \dot z_1 &=-\epsilon w_1z_1 F(z_1 w_1^{-1})-e^{-2z_1 w_1^{-1}} \left(1+\frac{x_1+z_1}{\xi}\right),\nonumber\\
 \dot \epsilon &=0,\nonumber
\end{align}
where $F(s) = 1-e^{-s}$. Therefore also 
\begin{align*}
 F(-s) = -e^s F(s).
 \end{align*}
 Henceforth we drop the subscripts. 
In this chart, we then have
\begin{align*}
 C = \{(w,x,z,\epsilon) \vert x = -\xi - z,\,w>0,\,\epsilon=0\},
\end{align*}
and 
\begin{align}
 C_\infty  =\{(w,x,z,\epsilon)\vert x=-\xi-z,w=0,\,\epsilon=0\}.\eqlab{CInftyK1}
\end{align}

We first notice that $e^{-zw^{-1}}$ and $e^{-2zw^{-1}}$ appearing in \eqref{phi1Eqns} are not defined along $w=0$ for $z\le 0$. We shall therefore introduce a new system by blowing up $w=z=0$ by the polar blowup transformation
\begin{align}
 (w,z) = \theta (\bar w,\bar z),\,\theta \ge 0,\,(\bar w,\bar z)\in S^1,\eqlab{blowupBadPoint}
\end{align}
and apply appropriate desingularization of the transformed vector-field to have a well-defined vector-field within $\theta=0$. In particular, we will divide the right hand side by $e^{-2\bar z\bar w^{-1}}$ whenever $\bar z<0$. 

We will use three separate charts $(\bar z=1)_1$, $(\bar w=1)_2$ and $(\bar z=-1)_3$ obtained by setting $\bar z=1$, $\bar w=1$ and $\bar z=-1$, respectively, so that we have the following local forms of \eqref{blowupBadPoint}:
\begin{align}
 w &= \theta_1 w_1,\,z= \theta_1,\eqlab{theta1Eqn}\\
 w &= \theta_2,\,z= \theta_2z_2,\nonumber\\
 w&=\theta_3w_3,\,z=-\theta_3,\eqlab{theta3Eqn}
\end{align}
where $(\theta_1,w_1) \in [0,\infty)^2$, $(\theta_2,z_2)\in [0,\infty)\times \mathbb R$, and $(\theta_3,w_3)\in [0,\infty)^2$ are the local coordinates, 
respectively. We consider each of these charts in the following.
\subsection{Chart $(\bar z=1)_1$}
Working in the chart $(\bar z=1)_1$ is similar to the analysis of \eqref{K3Ext} in chart $\phi_3$. 
Indeed, here we have $e^{-2zw^{-1}}=e^{-2w_1^{-1}}$, and as in chart $\phi_3$, we therefore put
\begin{align}
 q_1 = e^{-2w_1^{-1}}.\eqlab{q1app}
\end{align}
This gives the following equations
\begin{align}
\dot x &=- \epsilon \theta_1w_1\left(\theta_1w_1 x F(w_1^{-1})+(x+(1+\alpha)\theta_1)\right),\eqlab{K1BarZ1Ext}\\
\dot \theta_1 &=-\theta_1 w_1 \left(\epsilon \theta_1^2w_1 F(w_1^{-1})+q \left(1+\frac{x+\theta_1}{\xi}\right)\right),\nonumber\\
\dot w_1 &=w_1^2 q \left(1+\frac{x+\theta_1}{\xi}\right),\nonumber\\
\dot q &=2q^2 \left(1+\frac{x+\theta_1}{\xi}\right),\nonumber\\
\dot \epsilon &=0,\nonumber
\end{align}
by implicit differentiation and dropping the subscript on $q$. 
Here we have multiplied the right hand side by $\theta_1w_1$ to desingularize along $\theta_1=0$ and $w_1=0$. 
For this system,
\begin{align*}
C = \left\{(x,\theta_1,w,q,\epsilon)\vert 1 + \frac{x+\theta_1}{\xi}=0,w>0,q>0,\epsilon=0\right\},
\end{align*}
is a normally hyperbolic set of equilibria, but still not compact. As in chart $\phi_3$, the system is very degenerate near
\begin{align*}
 C_\infty = \left\{(x,\theta_1,w,q,\epsilon)\vert 1 + \frac{x+\theta_1}{\xi}=0,w_1=q=0,\epsilon=0,\theta_1\ge 0\right\}.
\end{align*} 
Then we proceed as in chart $\phi_3$: Let $P_1= \{(x,\theta_1,w,q,\epsilon)\in \mathbb R\times [0,\infty)^4\}$, $P_1^1 = \{(x,\theta_1,w,r,(\bar q,\bar \epsilon))\in \mathbb R\times [0,\infty)^3 \times S^1\}$ and blowup $q=\epsilon=0$ through the blowup transformation $$\Psi^1_1:P_1^1\rightarrow P_1,$$ which fixes $x$, $\theta_1$ and $w$ and takes
\begin{align*}
 (r,(\bar q,\bar \epsilon))\mapsto \left\{\begin{matrix} q &=& r\bar q,\\
 \epsilon &=& r\bar \epsilon,
 \end{matrix}\right.\quad r\ge 0,\,(\bar q,\bar \epsilon)\in S^1.
\end{align*}
Notice that we use a notion for the blowups (i.e. $\Psi_j^i$) that is similar to the one used in \secref{phi3Section}. However, we believe it is be clear from the context what blowup we are referring to. 
In the second blowup step, we set $P^2_1 = \{(\theta_1,r,\rho,(\bar x,\bar w,\bar{\bar \epsilon}))\in \mathbb R\times [0,\infty)^3 \times S^2\}$ and blowup $C_\infty$ through the transformation $$\Psi_1^2:P_1^2\rightarrow P_1^1,$$ which fixes $\theta_1$ and $r$ and takes
\begin{align}
(\theta_1,\rho,(\bar x,\bar w,\bar{\bar \epsilon})) \mapsto \left\{\begin{matrix}
 x& =& -\xi -\theta_1 + \rho \bar x,\\
 w_1 &=&\rho \bar w_1,\\
 \bar q^{-1}\bar \epsilon &=& \rho \bar{\bar{\epsilon}},
 \end{matrix}\right. \quad \rho \ge 0, (\bar x,\bar w_1,\bar{\bar \epsilon})\in S^2.\eqlab{Psi12here}
\end{align}
See \figref{step1K1} (a).
Since $(\bar q,\bar \epsilon)\in S^1$ we can write the last equality as 
\begin{align*}
 (\bar q,\bar \epsilon) = \left(\frac{1}{\sqrt{1+\rho^2 \bar{\bar \epsilon}^2}},\frac{\rho \bar{\bar \epsilon}}{\sqrt{1+\rho^2 \bar{\bar \epsilon}^2}}\right).
\end{align*}
Let $\Psi_1^{12}=\Psi_1^{1}\circ \Psi_1^2$.

Due to the multiplication by $\theta_1$ on the right hand side in the derivation of \eqref{K1BarZ1Ext}, the resulting system is still degenerate near $(\bar x,\bar{\bar \epsilon},\bar w_1)=(0,0,1),\,\theta_1=0$. Therefore let $P_1^3=\{(r,\rho,\varrho,(\bar{\bar x},\bar \theta_1,\bar{\bar{\bar \epsilon}}))\in [0,\infty)^3 \times S^2\}$. Then we apply a final blowup transformation $$\Psi_1^3:P_1^3\rightarrow P_1^2,$$ which fixes $r$ and $\rho$ and takes
\begin{align}
(\varrho,(\bar{\bar x},\bar \theta_1,\bar{\bar{\bar \epsilon}}))\mapsto \left\{\begin{matrix}
 \bar w_1^{-1} \bar x&=& \varrho \bar{\bar{x}},\\
 \theta_1 &= &\varrho \bar \theta_1,\\
 \bar w_1^{-1} \bar{\bar \epsilon} &=&\varrho \bar{\bar{\bar{\epsilon}}},
 \end{matrix}\right. \quad \varrho\ge 0, (\bar{\bar{x}},\bar \theta_1,\bar{\bar{\bar \epsilon}}) \in S^2.\eqlab{Psi13here}
\end{align}
See \figref{step1K1} (b). Since $(\bar x,\bar w_1,\bar{\bar \epsilon})\in S^2$ we can write the right hand side as
\begin{align*}
 (\bar x,\bar w_1,\bar{\bar \epsilon})=\left(\frac{ \varrho \bar{\bar x}}{\sqrt{1+\varrho^2 \bar{\bar x}^2+\varrho^2 \bar{\bar{\bar \epsilon}}^2}},\frac{1}{\sqrt{1+\varrho^2 \bar{\bar x}^2+\varrho^2 \bar{\bar{\bar \epsilon}}^2}},\frac{ \varrho \bar{\bar{\bar \epsilon}}}{\sqrt{1+\varrho^2 \bar{\bar x}^2+\varrho^2 \bar{\bar{\bar \epsilon}}^2}}\right).
\end{align*}
Let $\Psi_1^{123}=\Psi_1^{12}\circ \Psi_1^3$. 

\begin{figure}[h!]
\begin{center}
\subfigure[]{\includegraphics[width=.995\textwidth]{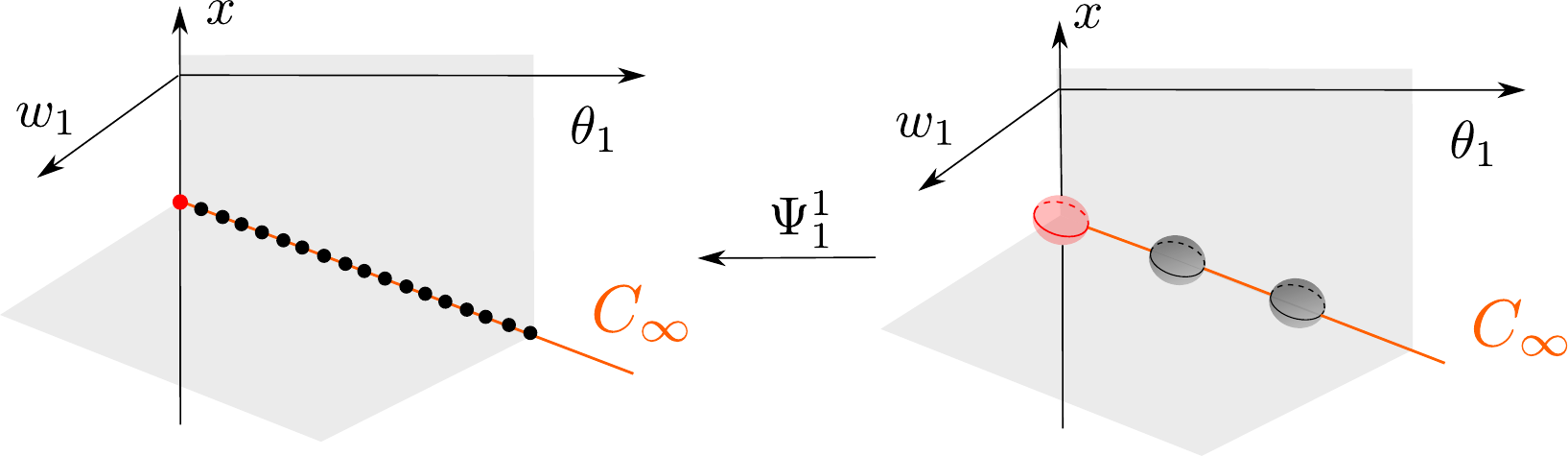}}
\subfigure[]{\includegraphics[width=.995\textwidth]{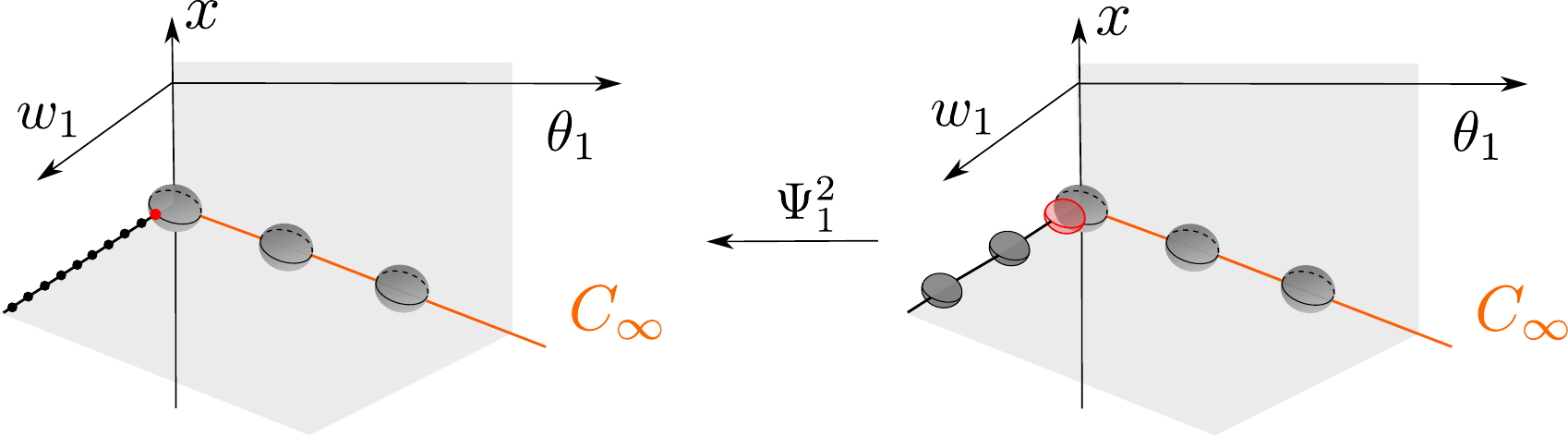}}
\end{center}
\caption{Blowup in chart $\bar z=1$.}
\figlab{step1K1}
\end{figure}
\subsection{Charts}
To describe the blowups $\Psi_1^1$, $\Psi_1^{12}$ and $\Psi_1^{123}$ we again use local directional charts. For $\Psi_1^1$ we will only work in the chart $(\bar z=1,\bar q=1)_{11}$ obtained by setting $\bar q=1$ so that 
\begin{align}
\Psi_{11}^1:\,(r_1,\epsilon_1)\mapsto \left\{\begin{matrix} q&=& r_1,\\
 \epsilon &=&r_1\epsilon_1,
 \end{matrix}\right.\eqlab{Psi111app}
\end{align}
in the local coordinates $(r_1,\epsilon_1)\in [0,\infty)^2$. Then to describe $\Psi_1^{12}$, we use two separate charts $(\bar z=1,\bar q=1,\bar x=1)_{111}$ and $(\bar z=1,\bar q=1,\bar w=1)_{112}$ obtained by setting $\bar x=1$ and $\bar w_1 = 1$ in \eqref{Psi12here}:
\begin{align}
 \Psi_{111}^{12}:\,(r_1,\theta_1,\rho_1,w_{11},\epsilon_{11})&\mapsto \left\{\begin{matrix} x &=& -\xi -\theta_1+\rho_1,\\
 w_1 &=&\rho_1 w_{11},\\
 q&=&r_1,\\
 \epsilon &=&r_1 \rho_1 \epsilon_{11},\end{matrix}\right..\eqlab{Psi11211}\\
 \Psi_{112}^{12}:\,(r_1,\theta_1,\rho_2,x_{2},\epsilon_{12})&\mapsto \left\{\begin{matrix}
 x &=& -\xi -\theta_1+\rho_2x_2,\\
 w_1 &=&\rho_2,\\
 q&=&r_1,\\
 \epsilon&=&r_1 \rho_2 \epsilon_{12}.
 \end{matrix}\right.\eqlab{Psi11212}
\end{align}
The coordinate changes between these two charts are given by 
\begin{align}
 \rho_1 &= \rho_2 x_2,\eqlab{Psi11212cc}\\
 w_{11}&=x_2^{-1},\nonumber\\
 \epsilon_{11}&=\epsilon_{12} x_2^{-1}.\nonumber
\end{align}

For $\Psi_1^{123}$, we will only consider the chart $(\bar z=1,\bar q=1,\bar w_1=1,\bar \theta_1=1)_{1121}$ obtained by setting $\bar \theta_1=1$ in \eqref{Psi13here}:
\begin{align}
\Psi_{1121}^{123}:\,(r_1,\rho_2,\varrho_1,x_{21},\epsilon_{121})\mapsto \left\{\begin{matrix}
x&=&-\xi-\varrho_1+\rho_2\varrho_1 x_{21},\\
\theta_1 &=&\varrho_1,\\
w_1 &=&\varrho_1 \rho_2,\\
q&=&r_1,\\
\epsilon &=&r_1 \rho_2 \varrho_1 \epsilon_{121},
\end{matrix}\right. \eqlab{finalChartBarZ1K1}
\end{align}
such that 
\begin{align}
\theta_1 &= \varrho_1,\eqlab{finalChartBarZ1K1cc}\\
 x_2 &= \varrho_1 x_{21},\nonumber\\
 \epsilon_{12} &=\varrho_1 \epsilon_{121}.\nonumber
\end{align}
See \tabref{tbl3} for a summary.

\subsection{Chart $(\bar w=1)_2$}
The analysis in this chart is more standard because here $e^{-zw^{-1}} = e^{-z_2}$ is regular. We have
\begin{align}
  \dot x &=-\epsilon \theta_2 \left(\theta_2 x F(z_2)+\left(x_1+(1+\alpha)\theta_2 z_2\right)\right),\eqlab{K1Barw1}\\
    \dot \theta_2 &=-\epsilon \theta_2^3 F(z_2),\nonumber\\
  \dot z_2 &=-e^{-2z_2} \left(1+\frac{x+\theta_2 z_2}{\xi}\right),\nonumber\\
  \dot \epsilon &=0,\nonumber
\end{align}
after multiplication of the right hand side by $\theta_2$.
Here $\rho_2=\epsilon=0,\,x=-\xi, \, z_2\in \mathbb R$ is a line of equilibria. For obvious reasons, we shall abuse notation slightly and call this line $C_\infty$ also. It is fully nonhyperbolic and we therefore blowup this set through the following blowup transformation $\Psi_2^1:P_2^1\rightarrow P_2$, where $P_2 = \{(x,\theta_2,z_2,\epsilon)\in \mathbb R\times [0,\infty)\times \mathbb R\times [0,\infty)\}$, $P_2^1=\{(z_2,\sigma,(\bar x,\bar \theta_2,\bar \epsilon))\in \mathbb R\times [0,\infty)\times S^2\}$, which fixes $z_2$ and $\sigma$ and takes
\begin{align*}
(z_2,\sigma, (\bar x,\bar \theta_2,\bar \epsilon))\mapsto \left\{\begin{matrix}
 x &=&-\xi - \sigma\bar \theta_2 z_2 +\sigma \bar x,\\
 \theta_2 &=& \sigma\bar \theta_2,\\
 \epsilon &=& \sigma \bar \epsilon, 
 \end{matrix}\right.\quad \sigma\ge 0,(\bar x,\bar \theta_2,\bar \epsilon)\in S^2
\end{align*}
See \figref{step2K1}. 
\begin{figure}[h!]
\begin{center}
{\includegraphics[width=.995\textwidth]{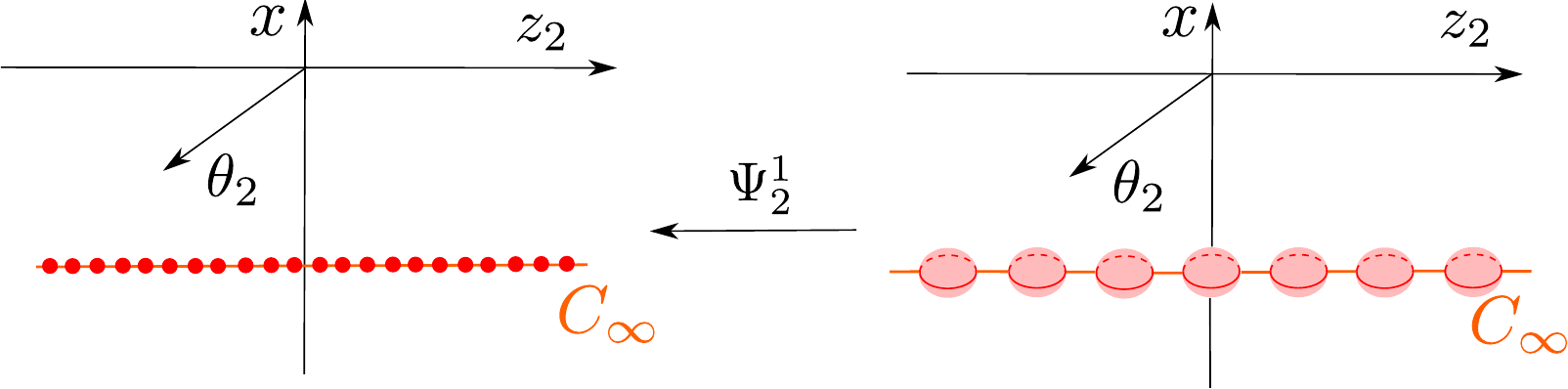}}
\end{center}
\caption{Blowup in chart $\bar w=1$.}
\figlab{step2K1}
\end{figure}
\subsection{Chart}
We only need to consider the single chart $(\bar w=1,\bar \theta_2=1)_{21}$ obtained by setting $\bar \theta_2=1$. This gives the following local form of $\Psi_2^1$
\begin{align}
 \Psi_{21}^1:\,(z_2,\sigma_1,x_{1},\epsilon_{1})\mapsto \left\{\begin{matrix}
                                                                  x &=& -\xi-\sigma_1 z_2+\sigma_1 x_{1}\\
                                                                  \theta_2 &=& \sigma_1,\\
                                                                  \epsilon &= &\sigma_1 \epsilon_{1}.
                                                                 \end{matrix}\right.\eqlab{Psi211}
\end{align}
Notice, that we can change coordinates between the chart $(\bar z=1,\bar q=1,\bar x=1,\bar \theta_1=1)_{1121}$ and $(\bar w=1,\bar \theta_2=1)_{21}$, using  \eqref{finalChartBarZ1K1}, as follows
\begin{align}
 \epsilon_{1} &=\exp(-2\rho_2^{-1})\epsilon_{121},\eqlab{ccBarWE1}\\
 \sigma_1 &=\varrho_1 \rho_2,\nonumber\\
 z_2 &=\rho_2^{-1},\nonumber\\
 x_{1} &=x_{21},\nonumber
\end{align}
for $\rho_2>0$. See \tabref{tbl3} for a summary.
\subsection{Chart $(\bar z=-1)_3$}
In this chart we have
\begin{align}
\dot x &= \epsilon \theta_3 e^{-w_3^{-1}}\left(\theta_3w_3 x F(w_3^{-1})-e^{-w_3^{-1}}(x-(1+\alpha)\theta_3)\right),\eqlab{K3BarZN1}\\
\dot \theta_3 &=\theta_3 \left(\epsilon \theta_3^2 e^{-w_3^{-1}} F(w_1^{-1})+1+\frac{x-\theta_3}{\xi}\right),\nonumber\\
\dot w_3 &=-w_3  \left(1+\frac{x-\theta_3}{\xi}\right),\nonumber\\
\dot \epsilon &=0,\nonumber
\end{align}
after division of the right hand side by $\theta_3^{-1} e^{2w_3^{-1}}$. For the analysis in this chart, we will have to keep track of exponentially small remainders in center manifold calculations. Standard power series expansion will therefore have to be adapted. For this purpose it is therefore useful to again introduce a flat function $q(w_3)$ as follows
\begin{align}
 q_3=w_3^{-1} e^{-w_3^{-1}}.\eqlab{q3}
\end{align}
It is also possible to use the seemingly more natural choice $q_3=e^{-w_3^{-1}}$ but the calculations are slightly simpler with \eqref{q3}. Implicit differentiation of \eqref{q3} gives the following system after multiplication by $w_3$ on the right hand side:
\begin{align*}
 \dot x &= \epsilon \theta_3^2 w_3^2 q\left(\theta_3w_3 x (1-w_3 q)-w_3 q(x-(1+\alpha)\theta_3)\right),\nonumber\\
\dot \theta_3 &=\theta_3 w_3 \left(\epsilon \theta_3^2 w_3 q (1-w_3 q)+1+\frac{x-\theta_3}{\xi}\right),\nonumber\\
\dot w_3 &=-w_3^3  \left(1+\frac{x-\theta_3}{\xi}\right),\nonumber\\
\dot q &= -q \left(1+\frac{x-\theta_3}{\xi}\right)(1-w_3),\nonumber\\
\dot \epsilon &=0,\nonumber
\end{align*}
Here we have dropped the subscript on $q$ and used \eqref{q3} to write $e^{-w_3^{-1}}=w_3 q$.  Let $P_3 =\{(x,\theta_3,w_3,q,\epsilon)\in \mathbb R\times [0,\infty)^4\}$, $P_3^1=\{(\pi,w_3,q,(\bar x,\bar \theta_3,\bar \epsilon))\in [0,\infty)^2 \times S^2\}$. Then we perform the following blowup transformation $$\Psi_3^1: P_3^1\rightarrow P_3,$$ of $x=-\xi,\,\theta_3=0,\,\epsilon=0$ which fixes $w_3$ and $q$ and takes
\begin{align*}
(\pi,(\bar x,\bar \theta,\bar \epsilon))&\mapsto \left\{\begin{matrix}
 x &=& -\xi +\pi \bar \theta +\pi \bar x,\\
 \theta_3 &=&\pi \bar \theta,\\
 \epsilon &=&\pi \bar \epsilon.
 \end{matrix}\right.
\end{align*}
Subsequently, we blowup $\bar\theta_3^{-1} \bar x=0,\,w_3=0,\,\bar \theta_3^{-1} \bar \epsilon=0$ through the blowup transformation $$\Psi_3^2:P_3^2\rightarrow P_3^1,$$ where $P_3^2 = \{(\pi,\mu,q,(\bar{\bar x},{\bar w}_3,\bar{\bar \epsilon}))\in [0,\infty)^2 \times S^2\}$, which fixes $\pi$ and $q$ and takes 
\begin{align*}
(\mu,(\bar{\bar x},\bar{\bar w},\bar{\bar \epsilon}))\mapsto \left\{\begin{matrix}
\bar \theta_3^{-1} \bar x &=&\mu \bar{\bar x},\\
 w_3 &=&\mu {\bar w}_3,\\
 \bar \theta_3^{-1} \bar \epsilon &=&\mu \bar{\bar \epsilon},
 \end{matrix}
 \right.\quad \mu\ge 0,(\bar x,\bar \theta,\bar \epsilon)\in S^2
\end{align*}
We illustrate the blowup in \figref{step3K1}. 
Let $\Psi_3^{12}=\Psi_3^1\circ \Psi_3^2$. 
\begin{figure}[h!]
\begin{center}
{\includegraphics[width=.995\textwidth]{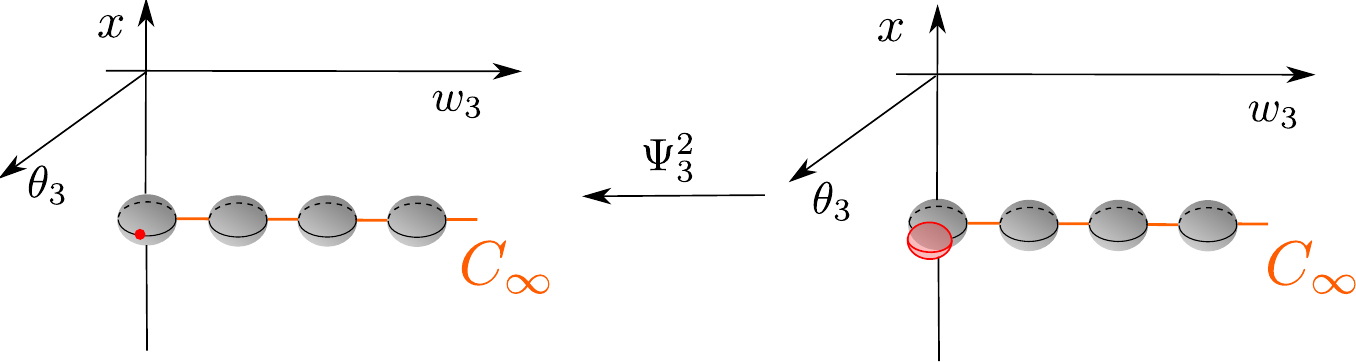}}
\end{center}
\caption{Blowup in chart $\bar z=-1$.}
\figlab{step3K1}
\end{figure}

\subsection{Chart}
To describe the blowup $\Psi_3^1$ we will only consider the following chart $(\bar z=-1,\bar \theta_3=1)_{31}$ obtained by setting $\bar \theta=1$:
\begin{align*}
\Psi_{31}^1:\,(\pi_1,x_1,\epsilon_1) \mapsto \left\{\begin{matrix} x &=&-\xi + \pi_1 +\pi_1 x_1,\\
 \theta_3 &=& \pi_1,\\
 \epsilon &=&\pi_1 \epsilon_1,
 \end{matrix}\right.
\end{align*}
using the local coordinates $\pi_1\ge 0,x_1\in \mathbb R,\epsilon_1\ge 0$. 

For $\Psi_{3}^{12}$ we use the single chart $(\bar z=-1,\bar \theta_3=1,\bar w_3=1)_{311}$ obtained by setting $\bar w_3=1$:
\begin{align}
\Psi_{311}^{12}:\,(\pi_1,\mu_1,x_{11},\epsilon_{11})\mapsto \left\{\begin{matrix}
 x &=& -\mu +\pi_1 + \pi_1 \mu_1 x_{11},\\
 \theta_3&=&\pi_1,\\
 w_3 &=&\mu_1,\\
 \epsilon &=&\pi_1 \mu_1\epsilon_{11}.
 \end{matrix}\right.\eqlab{Psi31112}
\end{align}
using the local coordinates $\pi_1\ge 0,\mu_1\ge 0,x_{11}\in \mathbb R,\epsilon_{11}\ge 0$. 
Notice that we can write the right hand side of \eqref{Psi31112} as
\begin{align*}
 x &=-\xi-z+wx_{11},\\
  \epsilon &=w\epsilon_{11},
\end{align*}
after eliminating $\pi_1$ and $\mu_1$ and using that $\theta_3=-z$.
Notice also that we can change coordinates between $(\bar w=1,\bar \theta_2=1)_{21}$ and $(\bar z=-1,\bar \theta_3=1,\bar w_3=1)_{311}$ as follows
\begin{align}
\pi_1 &=-\sigma_1 z_2,\eqlab{finalCC}\\
\mu_1 &=-1/z_2,\nonumber\\
x_{11}&=x_{1},\nonumber\\
\epsilon_{11} &=\epsilon_{1},\nonumber
\end{align}
for $z_2<0$. We summarize the results on the charts in \tabref{tbl3} below.
   \begin{table}[h]
    \renewcommand\arraystretch{2}
\begin{tabular}{|c|c|c|c|c|}
\hline
    \multicolumn{3}{|c|}{$(\bar z=1)_1$} & $(\bar w=1)_2$ & $(\bar z=-1)_3$
        \\
    \hline    
        $(\cdot,\bar q=1,\bar x=1)_{111}$ & $(\cdot,\bar q=1,\bar x=1)_{112}$ & $(\cdot,\bar q=1,\bar w_1=1,\bar \theta_1=1)_{1121}$ & $(\cdot,\bar \theta_2=1)_{21}$ &  $(\cdot,\bar \theta_3=1,\bar{w}_3=1)_{311}$\\
    \hline
$(r_1,\theta_1,\rho_1,w_{11},\epsilon_{11})$ & $(r_1,\theta_1, \rho_2,x_2,\epsilon_{12})$ &$(r_1,\rho_2,\varrho_1,x_{21},\epsilon_{121})$ & $(z_2,\sigma_1,x_1,\epsilon_{1})$ &$(\pi_1,\mu_1,x_{11},\epsilon_{11})$\\
\hline 
$\Psi^{12}_{111}$ \eqref{Psi11211}  & $\Psi^{12}_{112}$ \eqref{Psi11212} &$\Psi^{123}_{1121}$\eqref{finalChartBarZ1K1} & $\Psi^{1}_{21}$ \eqref{Psi211}& $\Psi^{12}_{311}$ \eqref{Psi31112}\\
\hline
\eqref{eqeqns12111}, \secref{seceqns12111} & \eqref{eqns112}, \secref{seceqns12112} &  \eqref{eqeqns121121}, \secref{seceqns121121} &  \eqref{eqhere21}, \secref{sechere21}&  \eqref{eqnsBlowupK1}, \secref{sechere311}\\
\hline 
 \multicolumn{2}{|c|}{\eqref{Psi11212cc}} &  \done & \multicolumn{2}{c|}{\eqref{finalCC}} \\
\hline
\done & \multicolumn{2}{c|}{\eqref{finalChartBarZ1K1cc}} & \done & \done\\
\hline 
\end{tabular}
\caption{Details about the charts used to describe the blowups in $\phi_1$. The first row divide the table into three parts, corresponding to the three directional charts associated with the initial blowup \eqref{blowupBadPoint}. The remaining rows have the same meaning as in \tabref{tbl1}. In particular, the last two rows provide the equation numbers for the equations describing the coordinate changes between the charts in the corresponding columns. }
\tablab{tbl3}
    \end{table}
\subsection{Summary of results in $\phi_1$}
The full details of the analysis of the blowup systems in chart $\phi_1$ is available in \secref{detailsK1}. Here we will try to summarize the findings. For simplicity, we restrict to the case where 
\begin{align}
\alpha<1.\eqlab{alpha1}
\end{align}
$\alpha>1$ is easier, while $\alpha=1$ is a special case, see \appref{alphaGE1}.

The blowup approach provides improved hyperbolicity properties of parts of the singular cycle visible in the chart $\phi_1$. We illustrate all the segments, including new segments only visible upon blowup, in \figref{gammaK1} using the viewpoint in \figref{step1K1}(b), \figref{step2K1} and \figref{step3K1}. In \figref{gammaK1}(a) we illustrate the parts visible in the chart $(\bar z=1)_1$. All orbits are contained within the subset $(\bar q,\bar \epsilon)=(1,0)$. $\gamma^7$ is asymptotic to a partially hyperbolic point $(1,0,0)$ on the sphere $(\bar x,\bar w,\bar{\bar \epsilon})\in S^2$. From here $\gamma^8$ is an unstable manifold which is asymptotic to a point  within $\bar{w}=1$ on a center manifold. This center manifold provides an extension of the slow manifold in the usual way, see e.g. \cite{krupa_extending_2001}. By desingularization of the slow flow on this center manifold we obtain an orbit $\gamma^9$ which is asymptotic to a partially hyperbolic equilibrium on $\theta_1=0$. From here $\gamma^{10}$ is an unstable manifold that we follow forward into $(\bar w=1)_2$, see \figref{gammaK1}(b), by following the slow flow on the center manifold. This orbit eventually brings us into $(\bar z=-1)_3$ where we finally obtain a heteroclinic $\gamma^{11}$ connecting the end of $\gamma^{10}$ with $W^{cu}(Q^6)$, obtained as a center submanifold of the reduced problem on the larger center manifold that extends the slow manifold in the usual way. In fact, $\gamma^{11}$ is only visible upon further use of a blowup involving exponentially small terms. The illustration in \figref{gammaK1}(c) is therefore (extra) caricatured. We combine the information in each of the charts into a single figure in \figref{gammaK1}(d).


In conclusion, we obtain the following: let 
$$\Pi^{70}=\Sigma^7_1\rightarrow \Sigma^0,$$ be the mapping obtained by the first intersection of the forward flow where
\begin{align*}
\Sigma^{7}_1=\left\{(x,z,w)\vert z = \frac{\xi}{2\alpha(1+ \nu)},\,x+\frac{(1+\alpha)\xi}{2 \alpha} \in [-\beta_3,\beta_3],\, w\in [0,\beta_4],\epsilon_1\in [0,\beta_5]\right\},
\end{align*}
recall \eqref{Sigma7} and \eqref{phi13xyz}, and
\begin{align}
 \Sigma^{0} = \{(x,z,w)\vert w = \delta, (\delta x,\delta z) \in N^0\}.\eqlab{Sigma0Here}
\end{align}
Here $N^0$ is the small neighborhood of $(x^0, z^0)$ in \eqref{Sigma0}. Notice also that $q^0 = (\delta x^0,\delta z^0,\delta)$ in our present $(x,z,w)=(x_1,z_1,w_1)$-coordinates, see \eqref{q0Here} and \eqref{phi13xyz}.
Then we have
\begin{lemma}\lemmalab{Pi70}
 The mapping $\Pi^{70}$ is well-defined for appropriately small $\delta>0$, $\nu>0$ and $\beta_i>0$, $i=1,\ldots, 5$ and all $0< \epsilon\ll 1$. In particular,
 \begin{align*}
  \Pi^{70}(x,w,\delta,\epsilon_1) = (x_+(x,w,\epsilon_1),z_+(x,w,\epsilon_1),\delta),
 \end{align*}
where $x_+$ and $z_+$ are $C^1$-functions in $x$ and $w$, satisfying
\begin{align*}
 x_+(x,w,\epsilon_1) &= \delta x^0+\mathcal O(\epsilon_1+w),\\
 z_+(x,w,\epsilon_1) &=\delta z^0+\mathcal O(\epsilon_1+w).
\end{align*}
\end{lemma}
\begin{proof}
 The result follows from a series of lemmas (see \lemmaref{Pi1118}, \lemmaref{Pi1129}, \lemmaref{Pi112110} and \lemmaref{Pi31111}) working in the local charts described above, and applying standard hyperbolic methods to follow the segments $\gamma^7-\gamma^{11}$, described above, and $W^{cu}(Q^6)$. Notice that the mappings between the different local sections are diffeomorphism that do not change the order. See details in \secref{detailsK1}. 
\end{proof}


\begin{figure}[h!]
\begin{center}
\subfigure[]{\includegraphics[width=.495\textwidth]{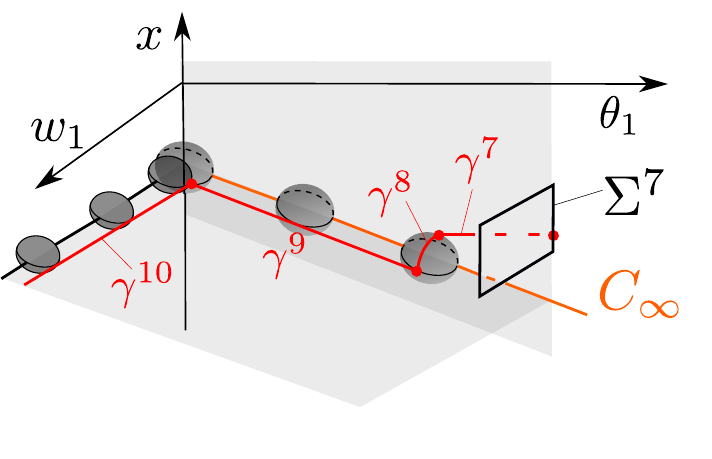}}
\subfigure[]{\includegraphics[width=.495\textwidth]{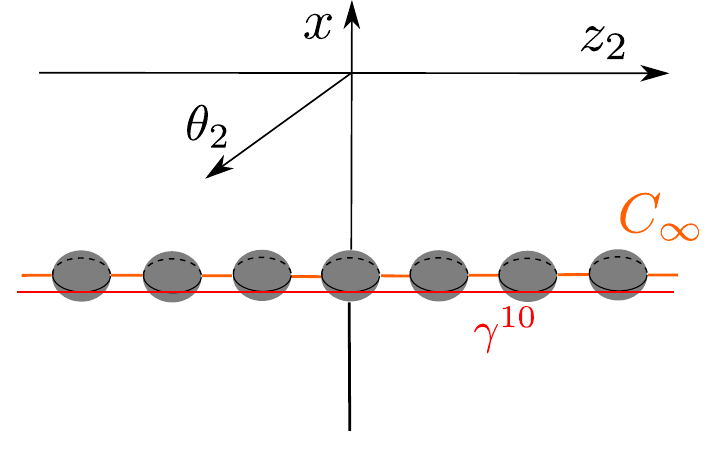}}
\subfigure[]{\includegraphics[width=.495\textwidth]{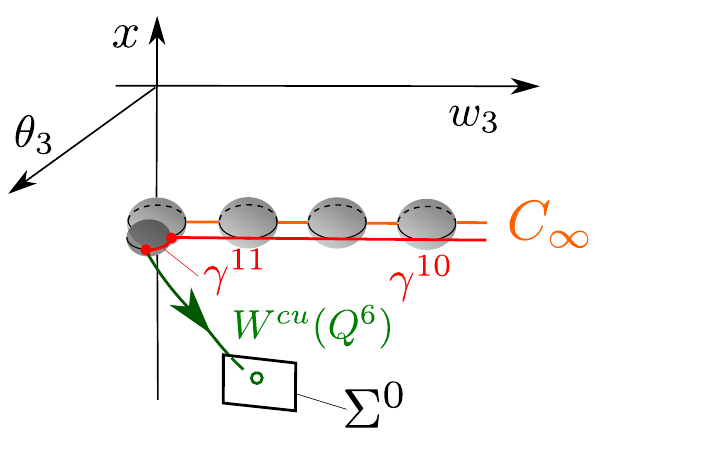}}
\subfigure[]{\includegraphics[width=.495\textwidth]{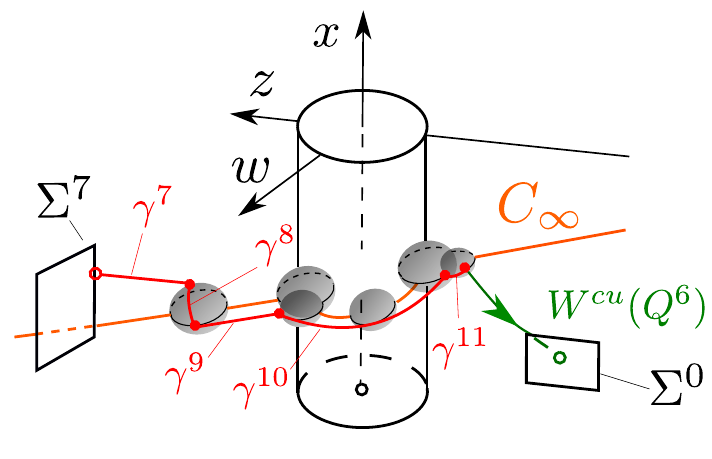}}
\end{center}
\caption{Singular orbit segments of the blowup in $\phi_1$. In (a), (b), (c) from the viewpoint of the three direction charts associated with \eqref{blowupBadPoint}. In (d) we collect (a), (b) and (c) into a global picture.} 
\figlab{gammaK1}
\end{figure}

\section{Proof of \lemmaref{kristianLemma}}\seclab{proofofkristianLemma}
We prove \lemmaref{kristianLemma} by first writing $\Pi^1 = \Pi^{70}\circ \Pi^{17}:\Sigma^1\rightarrow \Sigma^0$, where $\Sigma^1$ is defined in $\phi_3$, see \eqref{Sigma1Here}, and $\Sigma^0$ is defined in $\phi_1$, see \eqref{Sigma0Here}. Then by \lemmaref{Pi17} and \lemmaref{Pi70}
\begin{align*}
\Pi^1(x,y,\delta) = \left(\delta x^0+\mathcal O(\log^{-1} \epsilon^{-1} \log \log \epsilon^{-1}),\delta z^0+\mathcal O(\log^{-1} \epsilon^{-1} \log \log \epsilon^{-1}),\delta\right),
\end{align*}
the estimates being $C^1$-small with respect to $x$ and $y$. Transforming the result back to the original variables $(x,y,z)$ completes the proof of \lemmaref{kristianLemma}. 


\section{Analysis in chart $\phi_3$}\seclab{detailsK3}
In this section we describe the dynamics in chart $\phi_3$ using the blowup and the charts presented in \secref{phi3Section}.

\subsection{Chart $(\overline q=1,\overline w=1)_{11}$}\seclab{sec11eqns}
In this chart, we obtain the following equations
\begin{align}
\dot y &=\rho_1 \left(\epsilon_{11} \rho_1+y\frac{x_1}{\xi}\right),\eqlab{eqn211}\\
\dot r_1 &=2r_1\frac{x_1}{\xi},\nonumber\\
 \dot \rho_1 &=\rho_1^2 \frac{x_1}{\xi},\nonumber\\
 \dot x_1 &=-\frac{x_1}{\xi}-\epsilon_{11}(\rho_1x_1-\xi y+\alpha)+\rho_1\epsilon_{11}\xi F(\rho_1),\nonumber\\
  \dot \epsilon_{11} &=-\epsilon_{11} \frac{x_1}{\xi}(2+\rho_1).\nonumber
\end{align}
by \eqref{K3Ext} using $\Psi^{12}_{11}$, see \eqref{Psi1211}.
Notice that $r_1$-decouples. At this stage, we therefore proceed with the $(y,\rho_1,x_1,\epsilon_{11})$-subsystem only.  We notice that $(y,\rho_1,x_1,\epsilon_{11})=0$ is an equilibrium of the system with $-\xi^{-1}$ as a single non-zero eigenvalue. We therefore obtain an extension of the slow manifold as a center manifold using standard center manifold theory:
\begin{proposition}
Fix $\eta\in (0,1)$. Then there exists a $\delta>0$ and a small neighborhood $\mathcal U_{11}$ of $(y,\rho_1,\epsilon_{11})=0$ in $\mathbb R^3$ such that the following holds. There exists a locally invariant center manifold $M_{11}$ as a graph $$x_1 = -\epsilon_{11} \xi h_{11}(y,\rho_1,\epsilon_{11}),$$ over $(y,\rho_1,\epsilon_{11})\in \mathcal U_{11}$. Here $h_{11}$ is a smooth function of the following form
\begin{align}
 h_{11}(y,\rho_1,\epsilon_{11}) = \alpha-\xi y+\mathcal O(\rho_1+\epsilon_{11}).\eqlab{h11}
\end{align}
Furthermore, there exists a smooth stable foliation with base $M_{11}$ and one-dimensional fibers as leaves of the foliation. Within $x_1\in [-\delta,\delta], (\rho_1,y,\epsilon_{11})\in \mathcal U_{11}$, the contraction along any of these fibers is at least $e^{-\eta \xi^{-1} t}$. 
\end{proposition}
Now, consider the following sections:
\begin{align*}
 \Sigma_{11}^1 &= \{(y,\rho_1,x_1,\epsilon_{11})\vert \rho_1 = \delta,\,x_1\in [-\beta_1,0),\,y \in [-\beta_2,\beta_2],\,\epsilon_{11} \in (0,\beta_3]\},\\
 \Sigma_{11}^{1,\textnormal{out}} &= \{(y,\rho_1,x_1,\epsilon_{11})\vert \epsilon_{11} = \nu,\,\rho_1 \in [0,\beta_4],\,x_1\in [-\beta_1,0),\,y \in [-\beta_2,\beta_2]\},
\end{align*}
transverse to the flow. Notice that $\rho_1=\delta$ in $\Sigma_{11}^{1}$ becomes $z=1/\delta$ in the original variables using \eqref{Psi1211}, in agreement with $\Sigma^1$, see \eqref{Sigma1Here}. 

Now, the $1D$ manifold $$\gamma_{11,\text{loc}}^1 \equiv M_{11}\cap \{\rho_1=y=0\},$$ is invariant and by inserting $x_1=-\epsilon_{11}\xi h_{11}(0,0,\epsilon_{11})$ into \eqref{eqn211}, using \eqref{h11}, it follows that $\epsilon_{11}$ is increasing along this set for $\epsilon_{11}\ne 0$. $\gamma_{11,\text{loc}}^1$ intersects $\Sigma_{11}^{1,\textnormal{out}}$ in a point $q_{11}^{1,\textnormal{out}}$ with coordinates $(y,\rho_1,x_1,\epsilon_{11})=(0,0,-\nu \xi h_{11}(0,0,\nu),\nu)$. 
We now consider the mapping $\Pi_{11}^1:\,\Sigma_{11}^1\rightarrow \Sigma_{11}^{1,\textnormal{out}}$ defined as the first intersection by the forward flow.  See \figref{gamma1}. 
We have the following.
\begin{lemma}\lemmalab{Pi111}
The mapping $\Pi_{11}^1$ is well-defined for appropriately small $\delta$, $\nu$ and $\beta_i>0$, $i=1,4$. In particular,
\begin{align*}
 \Pi_{11}^1(y,x_1,\delta,\epsilon_{11}) &= \bigg(y_+(y,x_1,T(y,x_1,\epsilon_{11})),\rho_{1+}(y,x_1,T(y,x_1,\epsilon_{11})),\\
 &x_{1+}(y,x_1,T(x_1,\epsilon_{11})),\nu\bigg),
\end{align*}
with  $y_+$, $\rho_{1+}$, and $x_{1+}$ being $C^1$ in each of their arguments:
\begin{align*}
 \rho_{1+}(y,x_1,T)&=\frac{\delta}{1+\delta T}(1+\mathcal O(\delta)),\\
 y_+(y,x_1,T) &=\frac{y}{1+\delta T}+\frac{\delta \log\left(1+\delta T\right)}{\alpha (1+\delta T)}+\mathcal O\left(\frac{\delta}{1+\delta T}+e^{-ce^{2T}}\right),\\
 x_{1+}(y,x_1,T) &= -\nu \xi  h_{11}(\rho_{1+}(y,x_1,T),y_+(y,x_1,T),\nu)+\mathcal O(e^{-ce^{2T}}),
 \end{align*}
for some $c(\delta,\nu)>0$ sufficiently small and where $T(y,x_1,\epsilon_{11})>0$ is the unique solution of the following equation
\begin{align}
 \tilde \epsilon_{11}(\rho_{1+},y_+,x_{1+},\nu) = \tilde \epsilon_{11}(\delta,y,x_1,\epsilon_{11}) e^{2T(y,x_1,\epsilon_{11})} (1+\delta T(y,x_1,\epsilon_{11})). \eqlab{Teqn}
\end{align}
Here $\tilde \epsilon_{11}(\rho, y,x_1,\epsilon_{11})=\epsilon_{11} (1+\mathcal O(x_1+\epsilon_{11} \xi h_{11}(\delta,y,\epsilon_{11})))$ is a smooth function. 
\end{lemma}
\begin{proof}
First, we straighten out the stable fibers of $M_{11}$ by a smooth transformation of the form $(y,\rho_1,x_1,\epsilon_{11})\mapsto (\tilde \rho_1,\tilde y,\tilde \epsilon_{11})$ with
\begin{align*}
 \tilde \rho_1 &= \rho_1(1+\mathcal O(\rho_1)),\\
 \tilde y &=y+\mathcal O(\rho_1),\\
 \tilde \epsilon_{11}&=\epsilon_{11}(1+\mathcal O(x_1+\epsilon_{11} \xi h_{11}(\rho_1,y,\epsilon_{11}))).
\end{align*}
The transformation is close to the identity for $\rho_1,x_1$ and $\epsilon_{11}$ sufficiently small and hence invertible by the inverse function theorem. 
Then the dynamics of $(\tilde \rho_1,\tilde y,\tilde \epsilon_{11})$ becomes independent of $x_1$. We drop the tildes henceforth and obtain the following equations
\begin{align}
  \dot \epsilon_{11} &=\epsilon_{11}(2+\rho_1),\eqlab{reducedM11}\\
 \dot \rho_1 &=-\rho_1^2,\nonumber\\
 \dot  y &= -\rho_1\left(y-\frac{\rho_1}{h_{11}(\rho_1,y,\epsilon_{11})}\right).\nonumber
\end{align}
after division of the right hand side by $\epsilon_{11} h_{11}(\rho_1,y,\epsilon_{11})$.
We then use the following lemma.
\begin{lemma}
  There exists two $C^1$, locally defined functions $H_{11}$ and $\tilde  H_{11}$ such that
\begin{align}
 y\mapsto \tilde  y = H_{11}(y,\rho_1,\epsilon_{11}) = y+\mathcal O(\rho_1), \eqlab{H11}
\end{align}
with inverse
\begin{align*}
 \tilde y\mapsto y = \tilde  H_{11}(\tilde  y,\rho_1,\epsilon_{11}) = \tilde  y+\mathcal O(\rho_1), 
\end{align*}
transforms the system \eqref{reducedM11} into
\begin{align}
 \dot \epsilon_{11} &=\epsilon_{11}(2+\rho_1),\eqlab{normalForm}\\
 \dot \rho_1 &=-\rho_1^2,\nonumber\\
 \dot{\tilde  y} &= -\rho_1\left(\tilde  y-\frac{\rho_1}{\alpha}\right).\nonumber
\end{align}
\end{lemma}
\begin{proof}
 The transformation \eqref{H11} is composed of two steps. First we notice that $\epsilon_{11}=0$ is a normally hyperbolic set with smooth unstable fibers. We can straighten out these fibers through a smooth transformation of the form $(\rho_1,y,\epsilon_{11})\mapsto y_1$. Then the $y_1$ equation is independent of $\epsilon_{11}$:
 \begin{align*}
  \dot y_1 &=-\rho_1 \left(y_1-\frac{\rho_1}{ h_{11}(\rho_1,y,0)}\right).
 \end{align*}
The $(\rho_1,y_1)$-system therefore decouples, and with respect to the time $\tau$ defined by
\begin{align*}
 \frac{d\tau}{dt} = \rho_1
\end{align*}
this planar systems has a stable, hyperbolic node at the origin. Therefore we can linearize this system by a $C^1$ transformation. This gives the desired result. 
\end{proof}
Now we integrate \eqref{normalForm}. This gives
\begin{align*}
  \rho_1(T) &=\frac{\delta}{1+\delta T},\\
 \tilde y(T) &=\frac{\tilde y_0}{1+\delta T}+\frac{\delta \ln (1+\delta T)}{\alpha (1+\delta T)},
\end{align*}
where $T$ is defined by $\epsilon_{11}(T)=\tilde \nu$:
\begin{align}
 \tilde \nu = \tilde \epsilon_0 e^{2T} (1+\delta T).\eqlab{Teqn2}
\end{align}
We now transform $T$ back to the original time. This gives the duration of the transition in terms of this time. Using the contraction along the stable fibers, then gives the desired result. 
\end{proof}
Now, the function $T$, given implicitly by \eqref{Teqn2}, can be expressed in terms of the smooth Lambert W function $W:(-e^{-1},\infty)\rightarrow (-1,\infty)$, defined by $z=W(ze^z)$ for all $z\in (-1,\infty)$, as follows
\begin{align*}
 T(y,x_1,\epsilon_{11}) = \frac{1}{2}\left(W\left(2\tilde \nu e^{2\delta^{-1}}/\tilde \epsilon_{0}\right) -2\delta^{-1}\right).
\end{align*}
Using the asymptotics:  $$W(w) = \log w (1+o(1)),$$ of $W$ for $w\rightarrow \infty$, we obtain the following asymptotics of $T$ in \eqref{Teqn} as $\epsilon_{11}\rightarrow 0$:
\begin{align*}
 T(y,x_1,\epsilon_{11}) = \frac12 \log \epsilon_{11}^{-1} (1+o(1)),
\end{align*}
after substituting $\tilde \nu = \tilde \epsilon_{11}(\rho_{1+},y_+,x_{1+},\nu)$, $\tilde \epsilon_0 =  \tilde \epsilon_{11}(\delta,y,x_{1},\epsilon_{11})$. 
In fact, the partial derivatives of $T$ with respect to $y$ and $x_1$ satisfy an identical estimate. 
We therefore have the following corollary:
\begin{cor}\corlab{Pi111Cor}
 The mapping $(y,x_1)\mapsto \Pi_{11}^1(y,x_1,\delta,\epsilon_{11})$ is $C^1$ $\mathcal O\left(\frac{\log \log \epsilon_{11}}{\log \epsilon_{11}}\right)$-close to the constant mapping $(y,x_1)\mapsto q_{11}^{1,\textnormal{out}}$ as $\epsilon_{11}\rightarrow 0$. 
\end{cor}



\begin{figure}[h!]
\begin{center}
{\includegraphics[width=.525\textwidth]{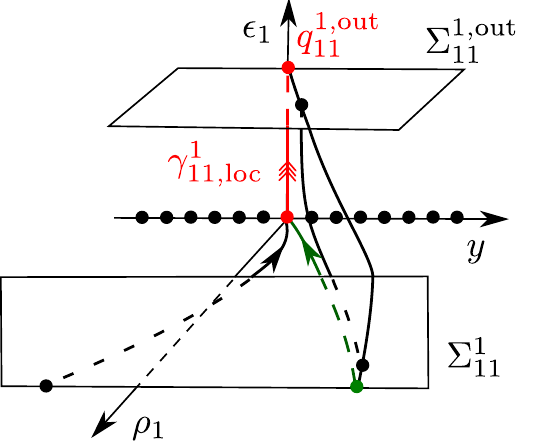}}
\end{center}
\caption{Illustration of the result in \lemmaref{Pi111}. $\gamma_{11,loc}^1$ is a local unstable manifold of the point $(y,\rho_1,\epsilon_{11}) =(0,0,0)$ within the center manifold $M_{11}$.}
\figlab{gamma1}
\end{figure}
\subsection{Chart $(\overline q=1,\bar{\bar \epsilon}=1)_{12}$}\seclab{sec12eqns}
In this chart we obtain the following equations
\begin{align}
  \dot y &=\rho_2 \left(\rho_2w_2^2+yw_2\frac{x_2}{\xi}\right),
\eqlab{chartBarEpsBarQ}\\
\dot \rho_2 &=-2\rho_2 \frac{x_2}{\xi},\nonumber\\
 \dot x_2 &=\frac{2x_2^2}{\xi}-w_2\left(\alpha-\xi y+\rho_2x_2\right)+w_2\frac{x_2}{\xi}\left(\rho_2x_1-1\right)\nonumber\\
  &+\xi\rho_2w_2^2F(\rho_2w_2), \nonumber\\
 \dot w_2 &=w_2 \frac{x_2}{\xi}(2+\rho_2 w_2),\nonumber
 \end{align}
from \eqref{K3Ext} using \eqref{Psi1212}.
We then transform $q_{11}^{1,\textnormal{out}}$ from above to this chart and obtain $q_{12}^{1,\textnormal{out}}$ with coordinates
\begin{align*}
 (y,\rho_2,x_2,w_2) = (0,0, -\xi h_{11}(0,0,\nu),\nu^{-1}),
\end{align*}
cf. \eqref{ccK31112}.
Setting $y=\rho_2=0$ in \eqref{chartBarEpsBarQ} gives
\begin{align*}
 \dot x_2 &=\frac{2x_2^2}{\xi}-w_2\left(\alpha-\xi y\right)-w_2\frac{x_2}{\xi},\\
 \dot w_2 &=2w_2 \frac{x_2}{\xi}. 
\end{align*}
In \cite{bossolini2017a}, it was shown, using a simple phase portrait analysis that $\gamma_{12}^{1}$, which is $\gamma_{11}^1$ in the present coordinates, is asymptotic to the nonhyperbolic equilibrium $x_2=w_2=0$ within the invariant subset $y=\rho_2=0$. Fix a large $T>0$. Then by regular perturbation theory, we can map a sufficiently small neighborhood of $q_{12}^{1,\textnormal{out}}$ diffeomorphically onto a neighborhood of $\phi_T(q_{12}^{1,\textnormal{out}})$ using the flow $\phi_t$. Due to the loss of hyperbolicity we apply the third blowup transformation, see \eqref{step3} and \eqref{Psi13}, of $x_2=w_2=0$. We describe this in the following using the chart $(\bar q=1,\bar{\bar \epsilon}=1,\bar{\bar w}=1)_{122}$ and the local form of the blowup \eqref{Psi13Charts}. 

\subsection{Chart $(\bar q=1,\bar{\bar \epsilon}=1,\bar{\bar w}=1)_{122}$}\seclab{sec122eqns}
In this chart, we obtain the following equations:
\begin{align}
 \dot y &=\rho_2 \varrho_2^2 \left(\rho_2\varrho_2+y\frac{x_{22}}{\xi}\right),\eqlab{eqn122Here}\\
 \dot \rho_2 &=-2\rho_2 \frac{x_{22}}{\xi},\nonumber\\
 \dot \varrho_2 &=\frac12 \varrho_2 \frac{x_{22}}{\xi}\left(2+\rho_2\varrho_2^2\right),\nonumber\\
 \dot x_{22} &= \frac{x_{22}^2}{\xi}-\left(\alpha-\xi y+\rho_2\varrho_2 x_{22}\right)\nonumber\\
 &+\varrho_2 \frac{x_{22}}{\xi}\left(\frac12 \rho_2 \varrho_2 x_{22}-1\right)+\xi \rho_2 \varrho_2 F(\rho_2 \varrho_2^2).\nonumber
\end{align}
Now, $\gamma_{12}^{1}$ in these coordinates become $\gamma_{122}^1$ which is asymptotic to the equilibrium $x_{22} = -\sqrt{\alpha \xi}$, $y=\rho_2=\varrho_2=0$. This point is a stable node within the invariant $(\varrho_2,x_{22})$-plane. We therefore work in a neighborhood of this equilibrium and consider the sections
\begin{align*}
\Sigma_{122}^{2,\text{in}} &= \{(y,\rho_2,\varrho_2,x_{22})\vert \varrho_2=\delta,\,\rho_2 \in [0,\beta_1],\,x_{22}-\sqrt{\alpha \xi} \in [-\beta_2,\beta_2],\,y\in [-\beta_3,\beta_3]\},
\end{align*}
and
\begin{align*}
\Sigma_{122}^{2,\textnormal{out}} &=\{(y,\rho_2,\varrho_2,x_{22})\vert \rho_2 = \nu,\,\varrho_2\in [0,\beta_4],\,\,x_{22}-\sqrt{\alpha \xi} \in [-\beta_2,\beta_2],\,y\in [-\beta_5,\beta_5]\}.
\end{align*}
Notice the graph $x_{22} = -\sqrt{\xi( \alpha-\xi y)}$, $y<\frac{\alpha}{\xi}$, $\rho_2=\varrho_2=0$ is a curve of equilibria of \eqref{eqn122Here}. It is normally hyperbolic with a $3D$ stable manifold within $\rho_2=0$ and a $2D$ unstable manifold $x_{22} = -\sqrt{\xi( \alpha-\xi y)},\,\varrho_2=0,\,\rho_2\ge 0$, $y<\frac{\alpha}{\xi}$. In particular, 
\begin{align*}
\gamma_{122}^2 = \{(y,\rho_2,\varrho_2,x_{22})\vert y=\varrho_2=0,\,x_{22} = -\sqrt{\alpha\xi},\,\rho_2\ge 0\},
\end{align*}
is contained within the unstable manifold and is invariant. See \figref{gamma2}. 

Consider the local mapping $\Pi_{122}^2$ from $\Sigma_{122}^{2,\text{in}}$ to $\Sigma_{122}^{2,\textnormal{out}}$ obtained from the first intersection by following the forward flow. 

\begin{lemma}\lemmalab{Pi22}
$\Pi_{122}^2$ is well-defined for appropriately small $\delta>0$, $\nu>0$ and $\beta_i>0$, $i=1,\ldots,5$. In particular,
\begin{align*}
 \Pi_{122}^2(y,\rho_2,\delta,x_{22}) = (y_+(y,\rho_2,x_{22}),\nu,\varrho_{2+}(\sqrt{\rho_2}),x_{22+}(y,\sqrt{\rho_2},x_{22})),
\end{align*}
with  $\varrho_{2+}$ a $C^1$-function, 
\begin{align*}
 x_{22+}(y,\sqrt{\rho_2},x_{22}) = H_{122}(y,\rho_2)+\mathcal O(\sqrt{\rho_2}),
\end{align*}
with $H_{122}$ smooth satisfying $H_{122}(0,0)=-\sqrt{\alpha \xi}$. Also 
\begin{align*}
 y_+(y,\rho_2,x_{22}) =y+\mathcal  O(\ln(\rho_2)\rho_{2})
\end{align*}
Furthermore, the remainder terms in $x_{22+}$ and $y_+$ are $C^1$ with respect to $x_{22}$ and $y$ and the orders of these terms as $\rho_2\rightarrow 0$ do not change upon differentiation.  

%
\end{lemma}
\begin{proof}
 We divide the right hand side by $-x_{22}/\xi$. This gives
  \begin{align*}
 \dot \varrho_2 &=-\frac12 \varrho_2 \left(2+\rho_2\varrho_2^2\right),\\
 \dot \rho_2 &=2\rho_2,
 \end{align*}
 and new equations for $x_{21}$ and $y$. It is possible to $C^1$ linearize the $(\varrho_2,\rho_2)$-subsystem by a transformation of the form $(\varrho_2,\rho_2)\mapsto \tilde \varrho_2=\varrho_2(1+\mathcal O(\rho_2))$ with $(\varrho_2,0)\mapsto \tilde \varrho_2=\varrho_2$. 
 Now, for the $\rho_2=0$ subsystem $y$ is constant and $x_{22}=-\sqrt{\xi(\alpha-\xi y)},\,\tilde \varrho_2=0$ is a hyperbolic stable node for any $y<\frac{\alpha}{\xi}$ sufficiently small. We can therefore linearize this subsystem by a $C^1$ transformation $(\tilde \varrho_2,x_{22})\mapsto \tilde x_{22}$. Applying these transformations to the full system produces
 \begin{align*}
 \dot y &=\mathcal O(\rho_2\tilde \varrho_2^2),\\
  \dot \rho_2 &=2\rho_2,\\
  \dot{\tilde{\varrho}}_2 &=-\tilde \varrho_2 ,\\
  \dot{\tilde x}_{22} &=-2\tilde x_{22}+\mathcal O(\rho_2\tilde \varrho_2),
   \end{align*}
Integrating these equations gives
\begin{align*}
 \tilde x_{22}(T) &= e^{-2T}\tilde x_{22}(0) + \int_0^T\mathcal O( e^{-2(T-s)}e^{2s}\rho_{20} e^{-s}\varrho_{20})ds \\
 &=\frac{\rho_{20}}{\nu}\tilde x_{22}(0)+\mathcal O(\sqrt{\rho_{20}}\varrho_{20}) = \mathcal O(\sqrt{\rho_{20}}),\\
 y(T) &= y(0)+\mathcal  O(\ln(\rho_{20}^{-1}\rho_2)\rho_{20} \varrho_{20}^2),
 \end{align*}
 using that $e^{2T} = \nu \rho_{20}^{-1}$ and hence $\rho_{20} e^T \sim \sqrt{\rho_{20}}$. 
We obtain similar estimates for the derivatives. 
\end{proof}
Notice that $\Pi_{122}^2(0,\delta,x_{22},0)=\gamma_{122}^2\cap \Sigma_{122}^{2,\textnormal{out}}$ for every $x_{22}-\sqrt{\alpha \xi} \in [-\beta_2,\beta_2]$. Now, along $\gamma_{122}^2$, $\rho_2$ is increasing for $\rho_2\ne 0$. We therefore study the dynamics in a neighborhood of this orbit moving to chart $(\bar \epsilon =1,\tilde x=-1,\tilde{\bar q}=1)_{211}$.

\begin{figure}[h!]
\begin{center}
{\includegraphics[width=.525\textwidth]{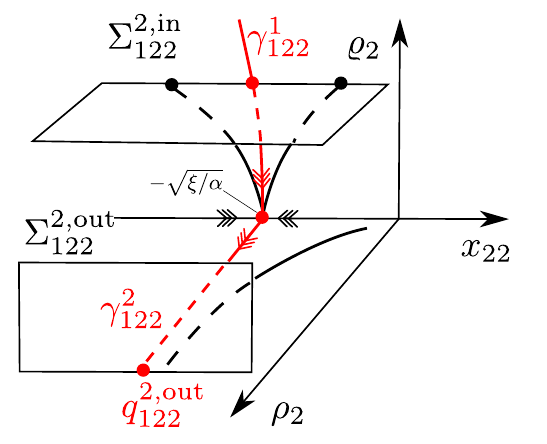}}
\end{center}
\caption{Illustration of the result in \lemmaref{Pi22}. $\gamma_{122}^2$ is the unstable manifold of the point $x_{22} = -\sqrt{\xi \alpha},\,\varrho_2=\rho_2=y=0$.}
\figlab{gamma2}
\end{figure}

\subsection{Chart $(\bar \epsilon =1,\tilde x=-1,\tilde{\bar q}=1)_{211}$}\seclab{seceqns211}
In this chart, we obtain the following equations
\begin{align}
 \dot y &=\sigma_1^2 \pi_1 w_{11} \left(-\frac{y}{\xi}+\sigma_1 w_{11}\right),\eqlab{eqeqns211}\\
 \dot \pi_1 &=-\frac{2\pi_1}{\xi},\nonumber\\
  \dot \sigma_1 &=\sigma_1 w_{11} G_{211}(y,\pi_{1},\sigma_1,w_{11}),\nonumber\\
 \dot w_{11} &=w_{11}\left(\frac{2}{\xi}-w_{11} \left(2G_{211}(y,\pi_{1},\sigma_1,w_{11})+\sigma_1^2 \frac{\pi_1}{\xi}\right)\right),\nonumber
\end{align}
and $\dot r_2=0$. Notice that $r_2\ge 0$ decouples and we shall therefore work within the $(y,\pi_{1},\sigma_1,w_{11})$-space. Here
\begin{align*}
 G_{211}(y,\pi_{1},\sigma_1,w_{11}) = \alpha-\sigma_1-\xi y-\sigma_1\pi_1 (\sigma_1+1)-\sigma_1^2\xi \pi_1 w_{11}F(\sigma_1^2\rho_1w_{11}).
\end{align*}
Also $\gamma_{122}^2$ from chart $(\bar q=1,\bar{\bar{\epsilon}}=1,\bar{\bar{w}} = 1)_{122}$ becomes
\begin{align*}
 \gamma_{211}^2 =\{(y,\pi_1,\sigma_1,w_{11})\vert \pi_1>0,w_{11} = 1/(\alpha \xi),\,\sigma_1 = 0,\,y=0\},
\end{align*}
using \eqref{cc3bto4}. It is contained within the invariant set $\sigma_1=0$ where
\begin{align*} \dot y &=0,\\
 \dot \pi_1 &=-\frac{2\pi_1}{\xi},\\
 \dot w_{11} &=2w_{11}\left(\frac{1}{\xi}-w_{11} (\alpha-\xi y) \right).
\end{align*}
Here the $2D$ graph $w_{11} = 1/(\xi (\alpha-\xi y))$, over $y<\frac{\alpha}{\xi}$, $\pi_1\ge 0$, is invariant. This set is foliated by $1D$ stable manifolds $w_{11}=1/(\xi (\alpha-\xi y)),y=\text{const}$, $\pi_1\ge 0$ of points on the line of equilibria $w_{11}= 1/(\xi (\alpha-\xi y))$, $y<\frac{\alpha}{\xi}$, $\pi_1=0$. In particular, $\gamma_1^2$ is contained within the stable manifold with $y=0$, being asymptotic under the forward flow to $w_{11} = 1/(\alpha \xi)$, $y=0$, $\pi_1=0$ within this line. 

Next, within the invariant set $\pi_1=0$ we have
\begin{align*}
\dot y&=0,\\
\dot \sigma_1 &=\sigma_1 w_{11} \left(\alpha-\sigma_1-\xi y\right),\\
 \dot w_{11} &=w_{11}\left(\frac{2}{\xi}-2w_{11} \left(\alpha-\sigma_1-\xi y\right)\right).
\end{align*}
For this subsystem, the line of equilibria $w_{11}= 1/(\xi (\alpha-\xi y))$, $y<\frac{\alpha}{\xi}$, $\sigma_1=0$ is of saddle type. Indeed, the linearization about any point in this set, gives $-2/\xi$ and $1/\xi$ as eigenvalues with stable space purely in the $w_{11}$-direction and unstable space contained in the $(\sigma_1,w_{11})$-plane. It is possible to write the individual unstable manifolds as graphs:
\begin{align*}
 w_{11} = H_{211}(y,\sigma_1),
\end{align*}
with $H_{211}$ smooth, such that $H_{211}(y,0)=1/(\xi(\alpha-\xi y)$, for $\sigma_1\le \nu$ with $\nu>0$ sufficiently small. Let $\gamma_{211}^3$ be the unstable manifold of $w_{11}=1/(\alpha \xi)$, $y=0$. Locally it is given as
\begin{align*}
\gamma_{211,\text{loc}}^{3} = \{(y,\pi_{1},\sigma_1,w_{11})\vert w_{11}=H(0,\sigma_1),0\le \sigma_1\le \nu,\,y=0,\,\pi_1=0\}.
\end{align*}

Therefore, we consider the following sections transverse to the flow:
\begin{align*}
 \Sigma_{211}^{3,\text{in}} &= \{(y,\pi_{1},\sigma_1,w_{11})\vert \pi_1=\delta,\,w_{11}-1/(\alpha \xi) \in [-\beta_1,\beta_1],\,\sigma_1 \in[0,\beta_2],\,y\in [-\beta_3,\beta_3]\},\\
 \Sigma_{211}^{3,\textnormal{out}} &=\{(y,\pi_{1},\sigma_1,w_{11})\vert \pi_1\in [0,\beta_4], \sigma_1=\nu,\,w_{11}-1/(\alpha \xi) \in [-\beta_1,\beta_1],\,y\in [-\beta_5,\beta_5] \}.
\end{align*}
Let $\Pi_{211}^3:\Sigma_{211}^{3,\text{in}}\rightarrow \Sigma_2^{3,\textnormal{out}}$ be the associated map obtained by the first intersection by applying the forward flow.

\begin{lemma}\lemmalab{Pi2113}
 $\Pi_{211}^3$ is well-defined for appropriately small $\delta>0$, $\nu>0$ and $\beta_i>0$, $i=1,\ldots,5$. In particular
 \begin{align*}
  \Pi_{211}^3(y,\delta,\sigma_1,w_{11}) = (y_+(y,\sigma_1,w_{11}),\pi_{1+}(y,\sigma_1,w_{11}),\nu,w_{11+}(y,\sigma_1,w_{11})),
 \end{align*}
with 
\begin{align*}
 \pi_{1+}(y,\sigma_1,w_{11}) &= \mathcal O(\sigma_1^2),\\
 w_{11+}(y,\sigma_1,w_{11}) &=H_{211}(y, \nu)+\mathcal O(\sigma_1^2),\\
 y_{11+}(y,\sigma_1,w_{11})&=y+\mathcal O(\ln (\sigma^{-1}) \sigma_1^2).
\end{align*}
Furthermore, the remainder terms in $\pi_{1+}$, $w_{11+}$ and $y_{11+}$ are $C^1$ with respect to $y$ and $w_{11}$ and the orders of these terms as $\sigma_1\rightarrow 0$ do not change upon differentiation.  
\end{lemma}
\begin{proof}
The proof is similar to the proof of \lemmaref{Pi22}, using partial linearization and Gronwall-like estimation of the remainder. We leave out the details.
\end{proof}

Notice that $\Pi_{211}^3(\delta,w_{11},0,0)=\gamma_{211}^3\cap \Sigma_{211}^{3,\textnormal{out}}$. See \figref{gamma3}. Notice also that $\gamma_{211,\text{loc}}^3$ in the $(x,y,w)$-variables becomes:
\begin{align*}
 \gamma_{\text{loc}}^3 = \{(x,y,w)\vert x\in [-1-\nu,-1],y=w=0\},
\end{align*}
using \eqref{Psi21114a14b},
in agreement with \eqref{gamma33}. ($\gamma^{1}$ and $\gamma^2$, on the other hand, both ``collapse'' to $Q^1$ at $(x,y,w)=(0,0,0)$ upon blowing down. See also \figref{gamma}.) To follow $\gamma_{211}^3$ forward, we move to chart $(\bar \epsilon=1,\tilde w=1)_{21}$, see \eqref{step5}.

\begin{figure}[h!]
\begin{center}
{\includegraphics[width=.525\textwidth]{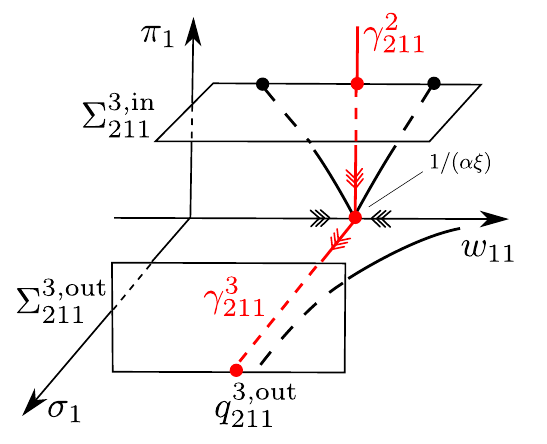}}
\end{center}
\caption{Illustration of the result \lemmaref{Pi2113}. $\gamma^3_{211}$ is the unstable manifold of $(y,\pi_{1},\sigma_1,w_{11})=(0,0,0,1/(\alpha\xi))$.}
\figlab{gamma3}
\end{figure}
\subsection{Chart $(\bar \epsilon=1,\tilde w=1)_{21}$}\seclab{seceqns21}
 In this chart we obtain the following equations
\begin{align}
\dot x &=-(x+1+\alpha)+x \mu_1 q_{21} \left(y+\frac{x+1}{\xi}\right),\eqlab{eqeqns21}\\
 \dot y &=\mu_1 F(\mu_1)+y \mu_1 q_{21} \left(y+\frac{x+1}{\xi}\right),\nonumber\\
 \dot \mu_1 &=\mu_1^2 q_{21} \left(y+\frac{x+1}{\xi}\right),\nonumber\\
\dot q_{21} &= 2 q_{21}^{2} \left(y+\frac{x+1}{\xi}\right)(2-\mu_1),\nonumber
\end{align}
and $\dot r_2=0$. Again, $r_2$ decouples and we shall therefore only work with the $(x,y,\mu_1,q_{21})$-system. 
Also $\gamma_{211,\text{loc}}^{3}$ becomes
\begin{align*}
 \gamma_{21,\text{loc}}^{3} = \bigg\{(x,y,\mu_1,q_{21})\vert &q_{21} = \sigma_1^{-2} H(\sigma,0),\,x=-1-\sigma_1,\, \sigma_1\in (0, \nu),\\
 &y=0,\,\mu_1=0\bigg\},
\end{align*}
using \eqref{cc4to5}, in the present chart. It is therefore contained within the invariant set $\mu_1=y=0$ where
\begin{align*}
\dot x &=-(x+1+\alpha),\\
\dot q_{21} &= 2 q_{21}^{2} \frac{x+1}{\xi},
\end{align*}
Notice, that starting from $x=-1-\nu$ with $\nu>0$ small, $x$ and $q_{21}$ are both monotonically decreasing towards their steady-state values $(x,q_{21})=(-1-\alpha,0)$.  Therefore, by extending 
$\gamma_{2,\text{loc}}^{3}$ by the forward flow, we obtain an orbit that is asymptotic to $(x,q_{21})=(-1-\alpha,0)$. Since the $x$-direction is a stable space and the $q$-direction is a center space, the orbit $\gamma_2^{3}$ approaches the steady state as a center manifold $x=h(q)$ over $0\le q\le \delta$ which is flat at $q=0$: $h^{(i)}(0)=0$ for all $i\in \mathbb N$. In particular, by center manifold theory, we have the following:
\begin{lemma}
Fix $\eta\in (0,1)$. Then there exists a $\delta>0$ and a neighborhood $\mathcal U_{21}$ of $(\mu_1,q_{21},y)=0$ in $\mathbb R^3$ such that the following holds. There exists a locally invariant center manifold $N_{21}$ as a graph $$x = -1-\alpha+\mu_1 h_{21}(\mu_1,q_{21},y),$$ over $(\mu_1,q_{21},y)\in \mathcal U_{21}$. Here $h_{21}$ is a smooth function.
Furthermore, there exists a smooth stable foliation with base $N_{21}$ and one-dimensional fibers as leaves of the foliation. Within $x+1+\alpha \in [-\delta,\delta], (\mu_1,q_{21},y)\in \mathcal U_{21}$, the contraction along any of these fibers is at least $e^{-\eta t}$. 
\end{lemma}

Notice that $N_{21}\cap \{\mu_1=0\}$ becomes $L_\infty:\,x=-1-\alpha,\,w=0,y\in I$ upon blowing down using \eqref{step5}. 
%
Next, consider the following sections
\begin{align*}
 \Sigma_{21}^{4,\text{in}} &= \{(x,y,\mu_1,q_{21})\vert q_{21} = \delta,\, x+1+\alpha \in [-\beta_1,\beta_1],\,\mu_1 \in [0,\beta_2],\,y\in [-\beta_3,\beta_3]\},\\
 \Sigma_{21}^{5,\textnormal{out}} &= \{(x,y,\mu_1,q_{21})\vert q_{21} = \delta,\, x+1+\alpha \in [-\beta_1,\beta_1],\,\mu_1 \in [0,\beta_4],\,y-\frac{2\alpha}{\xi }\in [-\beta_5,\beta_5]\},
\end{align*}
and let $\Pi_{21}^{45}:\,\Sigma_{21}^{4,\text{in}}\rightarrow \Sigma_{21}^{5,\textnormal{out}}$ be the associated mapping obtained by the first intersection of the forward flow. By reducing the dynamics to the center manifold $N_{21}$ (and applying a subsequent blowup) we will then show that we can guide the forward flow along the following lines
\begin{align}
 \gamma_{21}^4 &= \{(x,y,\mu_1,q_{21}) \vert x=-1-\alpha,\,y\in [0,2\alpha/\xi),\mu_1=q_{21}=0\},\eqlab{gamma214}\\
 \gamma_{21}^5 &=\{(x,y,\mu_1,q_{21}) \vert x=-1-\alpha,\,y=2\alpha/\xi,\mu_1=0,q_{21}\ge 0\}.\nonumber
\end{align}
See \figref{gamma45}.
\begin{figure}[h!]
\begin{center}
{\includegraphics[width=.855\textwidth]{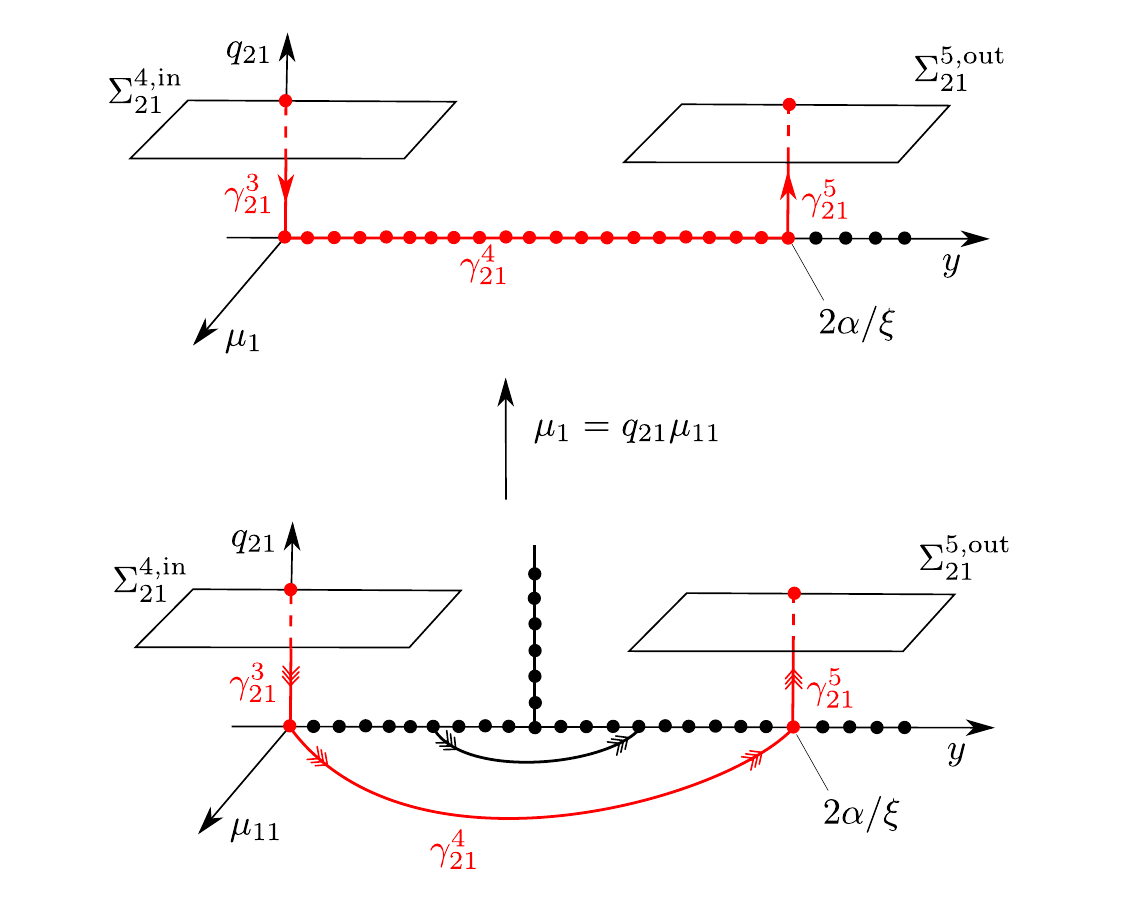}}
\end{center}
\caption{Illustration of the result in \lemmaref{Pi2145}. In the $(y,\mu_1,q_{21})$-variables, the $y$-axis is a line of equilibria. Upon the blowup $(q_{21},\mu_{11})\mapsto \mu_1 = q_{21}\mu_{11}$, this line has improved hyperbolicity properties, in particular we obtain a heteroclinic orbit $\gamma_{21}^4$ (which blows down to the expression \eqref{gamma214}) between equilibria $(0,0,0)$ and $(2\alpha/\xi,0,0)$ on this axis. It connects $\gamma_{21}^3$ with $\gamma_{21}^5$, the former being the stable manifold of $(0,0,0)$ while the latter is the unstable manifold of $(2\alpha/\xi,0,0)$. For simplicity, we use the same symbols in the two figures (although the axes are different).}
\figlab{gamma45}
\end{figure}
Then we have
\begin{lemma}\lemmalab{Pi2145}
 $\Pi_{21}^{45}$ is well-defined for appropriately small $\delta>0$, $\nu>0$ and $\beta_i>0$, $i=1,\ldots,5$. In particular,
\begin{align*}
 \Pi_{21}^{45}(x,y,\mu_1,\delta) = (x_{+}(x,y,\mu_1),y_+(x,y,\mu_1),\mu_{1+}(x,y,\mu_1),\delta),
\end{align*}
with $\mu_{1+}$ a $C^1$-function with $\mu_{1+}=\mu_1(1+o(1))$, 
\begin{align*}
x_+ (x,y,\mu_1)  &=-1-\alpha+\mu_{1+}(x,y,\mu_1) \nu h_{21}(\mu_{1+}(x,y,\mu_1),\delta,y_+(x,y,\mu_1))+\mathcal O(e^{-\eta/\mu_1}),\\
y_+(x,y,\mu_1) &= \frac{2\alpha}{\xi}-y+\mathcal O(\mu_1 \log \mu_1),
\end{align*}
as $\mu_1\rightarrow 0$. 
Furthermore, the remainder terms in $x_+$ and $y_+$ are $C^1$ with respect to $x$ and $y$ and the orders of these terms as $\mu_1\rightarrow 0$ do not change upon differentiation. 
\end{lemma}
\begin{proof}
 Working in a small neighborhood of $N_{21}$, we can straighten out the stable fibers by a smooth transformation of the form $(x,y,\mu_1,q_{21})\mapsto (\tilde y,\tilde \mu_1,\tilde q_{21})$ satisfying
 \begin{align*}
   \tilde y &=y+\mathcal O(\mu_1),\\
  \tilde \mu_1 &=\mu_1+\mathcal O(\mu_1^2q_{21}),\\
  \tilde q_{21} &= q_{21}+\mathcal O(q_{21}^{2}).
 \end{align*}
We drop the tildes henceforth and therefore consider the following reduced system on $N_{21}$.
\begin{align*}
 \dot y &=\mu_1 \left(F(\mu_1)+y q_{21} \left(\xi y-\alpha +\mu_1 h_2(\mu_1,q_{21},y)\right)/\xi\right),\\
 \dot \mu_1 &= \mu_1^2q_{21} \left(\xi y-\alpha +\mu_1 h_2(\mu_1,q_{21},y)\right)/\xi,\\
 \dot q_{21} &=q_{21}^{2} \left(\xi y-\alpha +\mu_1 h_2(\mu_1,q_{21},y)\right)/\xi (2-\mu_1).
 \end{align*}
Here $\mu_1=q_{21}=0$, $y\in I$, where $I$ is some appropriate interval, is a line of equilibria. It is not normally hyperbolic since the linearization about any point only has zero as an eigenvalue. We can gain hyperbolicity by applying the directional blowup, setting:
\begin{align*}
 \mu_1 = q_{21} \mu_{11}.
\end{align*}
Inserting this into the reduced equations we obtain 
\begin{align}
 \dot y &=\mu_{11} \left(F(q_{21}\mu_{11})+yq_{21} \left(\xi y-\alpha +q_{21} \mu_{11} h_2(q_{21} \mu_{11},q_{21},y)\right)/\xi\right),\eqlab{ReducedOnM211}\\
 \dot \mu_{11} &=2\mu_{11} \left(\xi y-\alpha +q_{21} \mu_{11} h_2(q_{21} \mu_{11},q_{21},y)\right)(-1+q_{21} \mu_{11})/\xi, \nonumber\\
 \dot q_{21} &= q_{21} \left(\xi y-\alpha +q_{21} \mu_{11} h_2(q_{21} \mu_{11},q_{21},y)\right)(2-\mu_1)/\xi ,\nonumber
\end{align}
after division of the right hand side by $q_{21}$. Now the line $\mu_1=q_{21}=0$, $y\in I$, is normally hyperbolic for any $y\ne \alpha/\xi$: The linearization about any point gives $\pm 2 (\xi y-\alpha)$ as nonzero eigenvalues. Within the invariant set $\mu_1 = 0$ we obtain
\begin{align*}
 \dot y &=0,\\
 \dot q_{21} &=q_{21} \left(\xi y-\alpha\right)/\xi. 
\end{align*}
Along $y=\alpha/\xi,q_{21}\ge 0$ every point is an equilibrium. For $y<\alpha/\xi$, $q_{21}$ contracts exponentially towards $q_{21}=0$. On the hand, for $y>\alpha/\xi$, $q_{21}$ expands exponentially. 
Next, within $q_{21}=0$ we obtain from \eqref{ReducedOnM211}
\begin{align*}
 \dot y &= \mu_{11},\\
 \dot \mu_{11} &= 2\mu_{11} \left(\xi y-\alpha\right)/\xi.
 \end{align*}
Writing 
\begin{align*}
 \frac{d\mu_{11}}{dy} = 2\left(\xi y-\alpha\right)/\xi,
\end{align*}
we realise that every point $\mu_{11}=0,y=y_0<\alpha/\xi$, is heteroclinic with $\mu_{11}=0, y=y_1>\alpha/\xi$ where $y_1=2\alpha/\xi-y_0$. See \figref{gamma45}. 

Now, to describe the mapping $\Pi_{21}^{45}$, we proceed as follows. We first work locally near $y=0$ and consider a mapping from $q_{21}=\delta$ to $\mu_{11}=\nu$. From there we then apply a finite time flow map by following the heteroclinic orbits within $\mu_{11}=0$ up to a neighborhood of $\mu_{11} = 0$, $q_{21}=0$, $y= 2\alpha/\xi$. From here, we then consider a mapping $\mu_{11}=\nu$ to $q_{21} = \delta$ working near the normally hyperbolic line $\mu_{11} = q_{21}=0,\,y\approx 2\alpha/\xi $. 

For the first part, near $y=0$, we divide the right hand side by $$\left(\xi y-\alpha +q_{21} \mu_{11} h_2(q_{21} \mu_{11},q_{21},y)\right) (-1+q_{21} \mu_{11})/\xi>0.$$ This gives
\begin{align*}
 \dot y &=\mu_{11}(-1+q_{21} \mu_{11})^{-1}  \left(\left(\xi y-\alpha +q_{21} \mu_{11} h_2(q_{21} \mu_{11},q_{21},y)\right)^{-1}  \xi F(q_{21}\mu_{11})+yq_{21}\right),\\
 \dot \mu_{11} &=2\mu_{11} ,\\
 \dot q_{21} &= q_{21} (1-q_{21} \mu_{11})^{-1} (-2+\mu_1).
\end{align*}
Now we straighten out the unstable fibers within the unstable manifold $q_{21}=0$ by performing a transformation of the form $(y,\mu_{11})\mapsto \tilde y$ such that 
\begin{align*}
 \dot{\tilde y} &=\mathcal O(\mu_{11} q_{21}),\\
 \dot \mu_{11} &=2\mu_{11} ,\\
 \dot q_{21} &= q_{21} (1-q_{21} \mu_{11})^{-1} (-2+\mu_1).
 \end{align*}
 The $y$-variables decouples and the $(\mu_{11},q_{21})$-subsystem has a saddle at $p_{21}=q_{21}=0$. We can therefore linearize this subsystem through a $C^1$-transformation of the form $(\mu_{11},q_{21})\mapsto \tilde q_{21}=q_{21}(1+\mathcal O(\mu_{11}))$ such that 
\begin{align*}
 \dot{\tilde y} &=\mathcal O(\mu_{11} \tilde q_{21}),\\
 \dot \mu_{11} &=2\mu_{11} ,\\
 \dot{\tilde q}_{21} &= -2\tilde q_{21}.
\end{align*}
 We then integrate this system from $\tilde q_{21} = \tilde \delta$ to $\mu_{11}=\nu$. This gives 
 \begin{align*}
 (\tilde y,\mu_{11},\tilde \delta) \mapsto (y+\mathcal O(\mu_{11}\log \mu_{11}),\nu,\mu_{11}\nu^{-1} \tilde \delta).
\end{align*}
We then return to $(y,\mu_{11},q_{21})$, by applying the $C^1$-inverses, and proceed with the second and third step. In the third, final step, we proceed using a similar approach to the first part, now working near $y=2\alpha/\xi,\,\mu_{11} = q_{21}=0$. We leave out the details, but in combination, this gives the desired result. 
\end{proof}



\subsection{Chart $(\bar \epsilon=1,\tilde{\bar q}=1)_{22}$}\seclab{seceqns22}
In this chart, we obtain 
\begin{align}
\dot x &= w_2\left(-(x+1+\alpha)+x\mu_2 \left(y+\frac{x+1}{\xi}\right)\right),\eqlab{eqeqns22}\\
 \dot y&=\mu_2w_2 \left(w_2 F(\mu_2 w_2)+y\left(y+\frac{x+1}{\xi}\right)\right),\nonumber\\
 \dot \mu_2 &=2\mu_2 \left(y+\frac{x+1}{\xi}\right),\nonumber\\
  \dot w_2 &=-w_2 \left(y+\frac{x+1}{\xi}\right)\left(2-\mu_2w_2\right),\nonumber
\end{align}
and $\dot r_2=0$. 
In this chart $\gamma_{21}^5$ becomes
 \begin{align*}
 \gamma_{22}^5 = \left\{(x,y,\mu_2,w_2)\vert w_2>0,x=-1-\alpha,y=2\alpha/\xi,\mu_2=0\right\},
\end{align*}
contained within the invariant set $x=-1-\alpha,\mu_2=0$ where
\begin{align*}
 \dot y &=0,\\
 \dot w_2 &= - 2w_2 \left(y-\frac{\alpha}{\xi}\right).
\end{align*}
$\gamma_{22}^5$ is asymptotic to the point $(x,y,\mu_2,w_2)=(-1-\alpha,2\alpha/\xi,0,0)$ within the set of equilibria $\mu_2=w_2=0$, $(x,y)$ in a neighborhood of $(-1-\alpha,2\alpha/\xi)$. Any point within this ``plane'' of equilibria has $\pm 2$ as non-zero eigenvalues. Indeed, within $w_2=0$, we have 
\begin{align*}
 \dot x &=0,\\
 \dot y &=0,\\
 \dot \mu_2 &= 2\mu_2 \left(y+\frac{x+1}{\xi}\right).
\end{align*}
In particular, 
\begin{align*}
 \gamma_{22}^6 = \left\{(x,y,\mu_2,w_2)\vert w_2=0,x=-1-\alpha,y=2\alpha/\xi,\mu_2\ge 0\right\},
\end{align*}
is contained within the unstable manifold of the set of equilibria $\mu_2=w_2=0$, $(x,y)$ in a neighborhood of $(-1-\alpha,2\alpha/\xi)$, as the individual unstable manifold of the base point of $\gamma_{22}^5$, $(x,y,\mu_2,w_2)=(-1-\alpha,2\alpha/\xi,0,0)$. 
We therefore consider the following sections
\begin{align*}
 \Sigma_{22}^{6,\text{in}} &= \{(x,y,\mu_2,w_2)\vert w_2 = \nu,\,x+1+\alpha\in [-\beta_1,\beta_1],\,y-2\alpha/\xi \in [-\beta_2,\beta_2],\,\mu_2 \in [0,\beta_2]\},\\
 \Sigma_{22}^{6,\textnormal{out}}&=\{(x,y,\mu_2,w_2)\vert \mu_2 = \delta ,\,x+1+\alpha\in [-\beta_3,\beta_3],\,y-2\alpha/\xi \in [-\beta_4,\beta_4],\,w_2 \in [0,\beta_5]\},
\end{align*}
and let $\Pi_{22}^6:\Sigma_{22}^{6,\text{in}}\rightarrow \Sigma_{22}^{6,\textnormal{out}}$ the associated local mapping obtained by the forward flow.
We then have
\begin{lemma}\lemmalab{Pi226}
 The mapping $\Pi_{22}^6$ is well-defined for appropriately small $\nu>0,\delta>0$ and $\beta_i>0$, $i=1,\ldots,5$. In particular,
 \begin{align*}
  \Pi_{22}^6(x,y,\mu_2,\nu) =\left(x_+(x,y,\mu_2),y_+(x,y,\mu_2),\delta,w_{2+}(\mu_2)\right),
 \end{align*}
where $w_{2+}$ is $C^1$ satisfying $w_{2+}(\mu_2) = \mu_2(1+o(1))$ and
\begin{align*}
 x_+(x,y,\mu_2) &=H_{22}(x,\delta)+\mathcal O(\mu_2 \ln \mu_2^{-1}),\\
 y_+(x,y,\mu_2)&=y+\mathcal O(\mu_2 \ln \mu_2^{-1}),
\end{align*}
 as $\mu_2\rightarrow 0$. Here $H_{22}$ is smooth and satisfy $H_{22}(-1-\alpha,\delta)=-1-\alpha$. 
 
 Furthermore, the remainder terms in $x_+$ and $y_+$ are $C^1$ with respect to $x$ and $y$ and the orders of these terms as $\mu_2\rightarrow 0$ do not change upon differentiation. 

\end{lemma}
\begin{proof}
We straighten out the individual stable manifolds within $\mu_2=0$ by a transformation of the form $(w_2,x)\mapsto \tilde x = H_{22}(x,w_2)$. Here by the invariance of $\gamma_{22}^5$ we have $H_{22}(-1-\alpha,w_2)=-1-\alpha$ for any $w_2$. Then
\begin{align*}
 \dot{\tilde x}_2 &=\mathcal O(w_2\mu_2).
\end{align*}
Straightforward estimation gives the desired result. 
\end{proof}
See \figref{gamma} for illustration of $\gamma^6$.  
\subsection{Exit of chart $\phi_3$}\seclab{exitPhi3}
To follow $\gamma_{22}^6$ forward, we return to the chart $(\overline \epsilon=1)_1$ and the coordinates $(x,y,w,\epsilon_1,r_1)$.
 In this chart, we obtain the following equations
\begin{align}
 \dot x &=w\left(-\epsilon_1 \left(x+1+\alpha\right)+x\left(y+\frac{x+1}{\xi}\right)\right),\eqlab{xyweps1}\\
 \dot y &=w\left(\epsilon_1 w F(w)+y\left(y+\frac{x+1}{\xi}\right)\right),\nonumber\\
 \dot w &=w^2 \left(y+\frac{x+1}{\xi}\right),\nonumber\\
 \dot \epsilon_1 &=-2\epsilon_1 \left(y+\frac{x+1}{\xi}\right).\nonumber
\end{align}
Also $\dot r_1  = 2\epsilon_1 \left(y+\frac{x+1}{\xi}\right)$ but this decouples and we shall therefore (again) just work with the $(x,y,w,\epsilon_1)$-subsystem.
In these coordinates $\gamma_{22}^6$ becomes
\begin{align*}
 \gamma_{1}^6 = \left\{(x,y,w,\epsilon_1)\vert \epsilon_1>0,\,x=-1-\alpha,y=\frac{2\alpha}{\xi},w=0\right\}.
\end{align*}
It is asymptotic to the point $q_{1}^7$ with coordinates
\begin{align}
 (x,y,w,\epsilon_1)=(-1-\alpha,\frac{2\alpha}{\xi},0,0). \eqlab{q26}
\end{align}
We work in a neighborhood of this point where $$y+\frac{x+1}{\xi}\approx \alpha/\xi>0.$$ We therefore divide the right hand side of \eqref{xyweps1} by this quantity and consider the following system
\begin{align*}
 \dot x &=w\left(- \frac{\epsilon_1(x+1+\alpha)}{y+\frac{x+1}{\xi}}+x\right),\\
 \dot y &=w\left( \frac{\epsilon_1 wF(w)}{y+\frac{x+1}{\xi}}+y\right),\\
 \dot w &=w^2 ,\\
 \dot \epsilon_1 &=-2\epsilon_1 .
\end{align*}

Notice that $\epsilon_1=w=0$ is invariant. Also the linearization about any point in this set gives $-2$ as a single zero eigenvalue. Therefore $\epsilon_1=0, w\in [0,\beta]$ and $(x,y)$ in a neighborhood of $(-1-\alpha,\frac{2\alpha}{\xi})$ is a local center manifold with smooth foliation by $1D$ fibers, along which orbits contract towards the center manifold with $e^{-2 t}$. Therefore there exists a smooth, local transformation $(x,y,w,\epsilon_1)\mapsto (\tilde x,\tilde y)=(x,y)+\mathcal O(w\epsilon_1)$ such that 
\begin{align*}
 \dot{\tilde x} &= w\tilde x,\\
 \dot{\tilde y}&=w\tilde y. 
\end{align*}
In the following, fix $y_1>\frac{2\alpha}{\xi}$ and consider the following sections:
\begin{align*}
\Sigma_1^{7,\text{in}} &= \{(x,y,w,\epsilon_1)\vert \epsilon_1=\delta,\,x+1+\alpha \in [-\beta_1,\beta_1],\,y-2\alpha/\xi \in [-\beta_2,\beta_2],\,w\in [0,\beta_3]\},\\
\Sigma_1^{7} &=\{(x,y,w,\epsilon_1)\vert \epsilon_1\in [0,\beta_4],\,x+1+\alpha \in [-\beta_5,\beta_5],\,y = y_1,\,w\in [0,\beta_6]\},
\end{align*}
Let $\Pi_1^7:\,\Sigma_1^{7,\text{in}}\rightarrow \Sigma_1^{7}$. Then, by integrating the $(\tilde x,\tilde y,w,\epsilon_1)$-system and transforming the result back to the $(x,y,w,\epsilon_1)$-system using the implicit function theorem, we obtain the following:
\begin{lemma}\lemmalab{Pi17New}
 $\Pi_1^7$ is well-defined for appropriately small $y_1-\frac{2\alpha}{\xi}$, $\delta>0$ and $\beta_i>0$, $i=1,\ldots, 6$. In particular,
 \begin{align*}
  \Pi_1^7(x,y,w,\delta) = (x_+(x,y,w),y_1,w_+(x,y,w),\epsilon_{1+}(x,y,w)),
 \end{align*}
with $x_+$, $w_+$ and $\epsilon_{1+}$ all $C^1$ satisfying
\begin{align*}
 x_+(x,y,w) &= \frac{xy_1}{y}(1+\mathcal O(w)),\\
 w_+(x,y,w) &= \frac{wy_1}{y}(1+\mathcal O(w)),\\
 \epsilon_{1+}(x,y,w)&=\mathcal O(e^{-c/w}),
\end{align*}
for some sufficiently small $c>0$. 
\end{lemma}
Define $\gamma_{1,\text{loc}}^7$ by
\begin{align}
 \gamma_{1,\text{loc}}^7 = \left\{(x,y,w,\epsilon_1) \vert \epsilon_1 = w=0,\,x = -\frac{\xi(1+\alpha)}{2\alpha} y,\,y\in \left[\frac{2\alpha}{\xi},\frac{2\alpha(1+\nu)}{\xi}\right]\right\}.\eqlab{gamma17loc}
\end{align}
It is obtained from the reduced problem of \eqref{xyweps1} within $\epsilon_1=0$:
\begin{align}
 \dot x &=x\left(y+\frac{x+1}{\xi}\right),\eqlab{xyReduced}\\
 \dot y&=y\left(y+\frac{x+1}{\xi}\right),\nonumber
\end{align}
using $x(0)=-1-\alpha,y(0)=2\alpha/\xi$, see \eqref{q26}, 
upon desingularization through division by $w$, and subsequently letting $w=0$. See \figref{gamma67}. Then it follows that 
\begin{align*}
 \Pi_1^7(x,y,0,\delta) = \gamma_{1,\text{loc}}^7\cap \Sigma_1^{7}.
\end{align*}
We can extend $\gamma_{1,\text{loc}}^7$ by the forward flow of \eqref{xyReduced} within $\epsilon_1=w=0$. We then have
\begin{lemma}
 Under the forward flow of \eqref{xyReduced} within $\epsilon_1=w=0$, $\gamma_1^7$ is bounded if and only if $\alpha<1$. In the affirmative case, $\gamma_1^7$ is asymptotic to the point $Q^5\in C_\infty$ with
 \begin{align*}
  x=-\frac{1+\alpha}{1-\alpha},\,y=\frac{2\alpha}{\xi (1-\alpha)}.
 \end{align*}
\end{lemma}

\begin{figure}[h!]
\begin{center}
{\includegraphics[width=.725\textwidth]{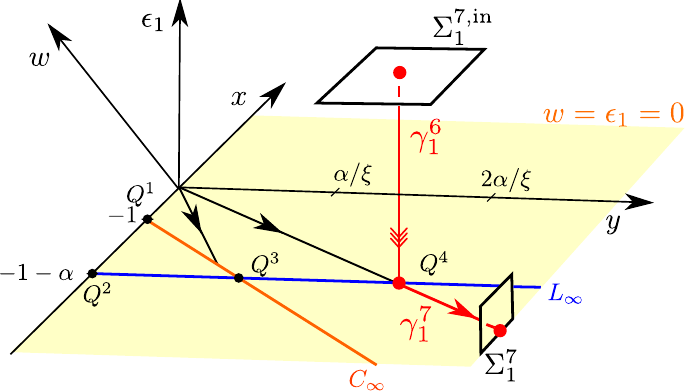}}
\end{center}
\caption{Illustration of the result in \lemmaref{Pi2145}. The set $\{\epsilon_1=0\}$ is normally hyperbolic and attracting within the region $y+(x+1)/\xi>0$. We can desingularize the flow within $\epsilon_1=w=0$ by division by $w$. This produces \eqref{xyReduced} and $\gamma^7_1$ (in red) as the flow of the base point $(x,y,w,\epsilon_1)=(-1-\alpha,2\alpha/\xi,0,0)$ of $\gamma^6_1$ (also in red). The projection of this point onto $(x,y,w)$ is $Q^4$.  }
\figlab{gamma67}
\end{figure}


\section{Analysis in chart $\phi_1$}\seclab{detailsK1}
In this section we describe the dynamics in chart $\phi_1$ using the blowup and the charts presented in \secref{phi1Section}.
\subsection{Chart $(\overline z=1,\overline q= 1,\overline x=1)_{111}$}\seclab{seceqns12111}
In this chart we obtain, using \eqref{Psi11211}, the following equations
\begin{align}
\dot \theta_1 &=\rho_1 w_{11} \theta_1 \left(-1/\xi -\epsilon_{11}\theta_1^2 \rho_1  w_{11} F(\rho_1^{-1} w_{11}^{-1})\right),\eqlab{eqeqns12111}\\
 \dot \rho_1 &=\rho_1 w_{11} G_{111}(\theta_1,\epsilon_{11},\rho_1, w_{11}),\nonumber\\
 \dot w_{11} &=-w_{11}^2 G_{111}(\theta_1,\epsilon_{11},\rho_1, w_{11})+\rho_1 w_{11}^2/\xi,\nonumber\\
 \dot \epsilon_{11} &=\epsilon_{11} \left(-2/\xi -G_{111}(\theta_1,\epsilon_{11},\rho_1 ,w_{11}) w_{11}\right), \nonumber
\end{align}
where
\begin{align*}
 G_{111}(\theta_1,\rho_1, w_{11},\epsilon_{11}) = -\theta_1/\xi -\epsilon_{11} \theta_1 \left(-\xi +\rho_1+\alpha \theta_1\right)+\epsilon_{11} \rho_1 w_{11} F(\rho_1^{-1} w_{11}^{-1})(\xi-\rho_1).
\end{align*}
In these coordinates, $\gamma_1^7$ becomes
\begin{align*}
 \gamma_{111}^7 = \left\{(\theta_1,\rho_1,w_{11},\epsilon_{11}) \vert \rho_1 = \frac{\xi(\alpha-1)}{2\alpha}+\theta_1,\,\theta_1 \in \left(\frac{\xi (1-\alpha)}{2\alpha},\frac{\xi}{2\alpha}\right),\,w_{11}=\epsilon_{11}=0\right\},
\end{align*}
for $\alpha<1$, recall the assumption \eqref{alpha1}. It is asymptotic to the point $q_{11}^8$ with coordinates
\begin{align}
(\theta_1,\rho_1,w_{11},\epsilon_{11}) = \left(\frac{\xi (1-\alpha)}{2\alpha},0,0,0\right).\eqlab{q118}
\end{align}
 Now, we notice that $\{\epsilon_{11}=0,w_{11}\in [0,\beta_5]\}$, with $\beta_5>0$ sufficiently small, is an attracting center manifold. The (center-)stable manifold has a smooth foliation by stable fibers as leaves of the foliation. We can straighten out these fibers through a transformation of the form $(\theta_1,\rho_1,w_{11},\epsilon_{11})\mapsto (\tilde \theta_1,\tilde \rho_1,\tilde w_{11}) = (\theta_1,\rho_1,w_{11}) + \mathcal O(w_{11} \epsilon_{11})$. This gives
 \begin{align}
 \dot \theta_1 &=-\rho_1 w_{11}\theta_1/\xi,\eqlab{Pi7Eqns}\\
\dot \rho_1 &=-\rho_1 w_{11} \theta_1/\xi,\nonumber\\
\dot w_{11} &=w_{11}^2 \left( \theta_1+\rho_1 \right)/\xi,\nonumber
\end{align}
upon dropping the tildes. We see that $w_{11}$ is a common factor and therefore divide this out on the right hand side. This gives
\begin{align*}
 \dot \rho_1 &= -\rho_1 \theta_1/\xi,\\
 \dot w_{11} &= w_{11} \left(\theta_1+\rho_1\right)/\xi,\\
 \dot \theta_1 &=-\rho_1 \theta_1/\xi,
\end{align*}
with respect to the new time. Now, $\rho_1=0$, $w_{11}=0$, $\theta_1>0$ is a line of equilibria. It is normally hyperbolic, being of saddle type. $\gamma_{111}^7$ is contained in the stable manifold within $w_{11}=0$, being asymptotic to the base point $q_{111}^8$ with $\rho_1=0$, $\theta_1 = \frac{\xi (1-\alpha)}{2\alpha}$, recall \eqref{q118}. From this point, there is also a unstable manifold
\begin{align*}
\gamma_{111}^8 = \left\{(\theta_1,\rho_1,w_{11},\epsilon_{11})\vert \rho_1=\epsilon_{11}=0,\,\theta_1= \frac{\xi(1-\alpha)}{2\alpha},\,w_{11}\ge 0\right\}.
\end{align*}
In the following, we work in a neighborhood of the point $q_{111}^8$. Let 
\begin{align*}
 \Sigma_{111}^{8,\text{\text{in}}} &=\left\{(\theta_1,\rho_1,w_{11},\epsilon_{11}) \vert \rho_1 = \delta,\,\theta_1-\frac{\xi (1-\alpha)}{2\alpha}\in [-\beta_1,\beta_1],w_{11} \in [0,\beta_3],\,\epsilon_{11}\in [0,\beta_4]\right\},\\
 \Sigma_{111}^{8,\text{\text{out}}} &=\left\{(\theta_1,\rho_1,w_{11},\epsilon_{11}) \vert w_{11} = \nu,\,\rho_1 \in [0,\beta_5],\,\theta_1-\frac{\xi (1-\alpha)}{2\alpha}\in [-\beta_6,\beta_6],\,\epsilon_{11}\in [0,\beta_4]\right\}, 
\end{align*}
transverse to the flow and $\Pi_{111}^8:\Sigma_{111}^{8,\text{\text{in}}} \rightarrow \Sigma_{111}^{8,\text{\text{out}}}$ the associated mapping obtained by the first intersection of the forward flow. Then we have 
\begin{lemma}\lemmalab{Pi1118}
 $\Pi_{111}^8$ is well-defined for appropriately small $\delta$, $\nu$ and $\beta_i>0$, $i=1,5$. In particular, 
 \begin{align*}
 \Pi_{111}^8(\theta_1,\delta,w_{11},\epsilon_{11}) = (\theta_{1+}(\theta_1,w_{11},\epsilon_{11}),\rho_{1+}(\theta_1,w_{11},\epsilon_{11}),\nu,\epsilon_{11+}(\theta_1,w_{11},\epsilon_{11})),
\end{align*}
where $\rho_{1+}$, $\epsilon_{11+}$ and $\theta_{1+}$ are $C^1$ and satisfy
\begin{align*}
 \theta_{1+}(\theta_1,w_{11},\epsilon_{11}) &=\theta_1-\delta+\mathcal O(w_{11}),\\
 \rho_{1+} (\theta_1,w_{11},\epsilon_{11}) &= \mathcal O(w_{11}),\\
 \epsilon_{11+}(\theta_1,w_{11},\epsilon_{11}) &=\mathcal O(\epsilon_{11}e^{-cw_{11}^{-1}}),
\end{align*}
for some $c>0$ sufficiently small. 
\end{lemma}
\begin{proof}
We integrate \eqref{Pi7Eqns} from $\rho_1(0) = \delta$ to $w_{11}(T)=\nu$. This gives
\begin{align*}
 \theta_1(T(\theta_1(0),w_{11}(0))) &=\theta_1(0)-T(\delta_1(0),w_{11}(0)),\\
 \rho_1(T(\theta_1(0),w_{11}(0))) &= \delta-T(\theta_1(0),w_{11}(0)),
\end{align*}
where 
\begin{align*}
 T(\theta_1,w_{11})=\frac{1}{2}(\delta+\theta_1)-\frac12 (\theta_1-\delta)\sqrt{1+\frac{4\delta \theta_1w_{11}}{\nu(\theta_1^2-\delta^2)}}.
\end{align*}
Notice that $T(\theta_1,0) = \delta$. 
Returning to the original variables gives the desired result upon using the exponential contraction towards $\epsilon_{11}=0$. 
\end{proof}
It follows that the image of $\gamma_{111}^7\cap \Sigma_{111}^{8,\text{\text{in}}}$ under $\Pi_{111}^8$ is $\gamma_1^7\cap \Sigma_{111}^{8,\text{\text{out}}}$. See \figref{gamma8}.
\begin{figure}[h!]
\begin{center}
{\includegraphics[width=.525\textwidth]{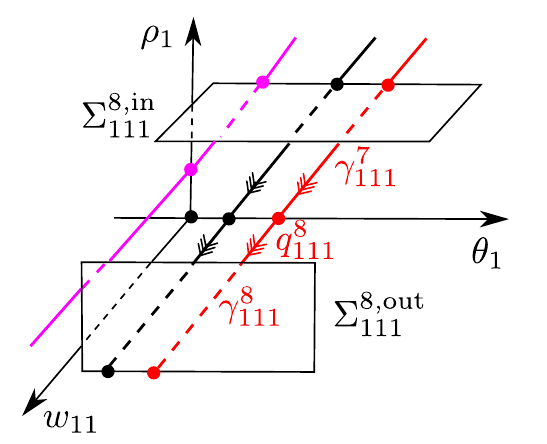}}
\end{center}
\caption{Illustration of the result in \lemmaref{Pi1118} within $\epsilon_{11}=0$. $\rho_1=w_{11}=0$ is a line of equilibria. The purple orbit is for $\alpha>1$, see \appref{alphaGE1}.}
\figlab{gamma8}
\end{figure}
\subsection{Chart $(\overline z=1,\overline q= 1,\overline w_1=1)_{112}$}\seclab{seceqns12112}
In this chart, we obtain the following:
\begin{align}
 \dot \theta_1 &=\rho_2 \theta_1 \left(\rho_2 \theta_1^2 \epsilon_{12} F(\rho_2^{-1})+x_2/\xi\right),\eqlab{eqns112}\\
 \dot \rho_2 &=\rho_2^2 \frac{x_2}{\xi},\nonumber\\
 \dot x_2 &=\rho_2 \epsilon_{12} \theta_1^2 F(\rho_2^{-1})(\xi-\rho_2 x_2)-\theta_1 \epsilon_{12} \left(-\xi +\rho_2 x_2 + \alpha \theta_1\right)-x_2\left(\theta_1 +\rho_2 x_2\right)/\xi,\nonumber\\
 \dot \epsilon_{12} &= -\epsilon_{12} x_2/\xi \left(2+\rho_2\right). \nonumber
\end{align}
In these coordinates, $\gamma_{111}^8$ takes the following form:
\begin{align*}
 \gamma_{112}^8= \left\{(\theta_1,\rho_2,x_2,\epsilon_{12})\vert x_2>0,\,\rho_2=0,\,\theta_1 = \frac{\xi(1-\alpha)}{2\alpha},\,\epsilon_{12}=0\right\},
\end{align*}
using the coordinate change $x_2=w_{11}^{-1}$ between the charts, recall \eqref{Psi11212cc}. The dynamics on $\gamma_{112}^8$ is asymptotic to the point
\begin{align*}
 (\theta_1,\rho_2,x_2,\epsilon_{12}) = \left(\frac{\xi(1-\alpha)}{2\alpha},0,0,0\right).
\end{align*}
This point becomes $Q^5$ upon blowing down, see \eqref{Q5expr}. 
The set defined by $\rho_2 = \epsilon_2 = x_2 = 0,\,\theta_1 \in [0,\infty)$ is a line of equilibria for \eqref{eqns112}. Upon blowing down, using \eqref{Psi11212} and \eqref{theta1Eqn}, it becomes the subset of $C_\infty$, see \eqref{CInftyK1}, with $z\ge 0$.  But within this blowup chart, the linearization about any point on this line now has one single non-zero eigenvalue $-\theta_1/\xi$ for $\theta_1>0$. This produces an extension of the slow manifold as a center manifold by standard center manifold theory in the usual way: 
\begin{proposition}
 Fix a closed interval $I\subset (0,\infty)$. Then there exists a $\delta>0$ and a neighborhood $\mathcal U_{112}$ of $(\rho_2,\epsilon_{12})=0$ in $\mathbb R^2$ such that the following holds. There exists a locally invariant center manifold $M_{112}$ as a graph $$x_2 = \epsilon_{12}\xi^2 \left(1-\xi^{-1} \alpha \theta_1+\epsilon_{12} h_{112}(\epsilon_{12},\rho_2,\theta_1)\right),$$ over $(\theta_1,\rho_2,\epsilon_{12})\in I\times \mathcal U_{112}$. Here $h_{112}$ is a smooth function.
Furthermore, there exists a smooth stable foliation with base $M_{112}$ and one-dimensional fibers as leaves of the foliation. Within $x_2 \in [-\delta,\delta], (\theta_1,\rho_2,\epsilon_{12})\in I\times \mathcal U_{112}$, the contraction along any of these fibers is at least $e^{-c t}$ with $c(I)>0$.
\end{proposition}

The reduced problem on $M_{112}$ is 
\begin{align}
\dot \theta_1 &=-\rho_2 \theta_1 \left(1+\frac{\rho_2\theta_1^2 F(\rho_2^{-1})}{\xi (1-\xi^{-1} \alpha \theta_1+\epsilon_{12} h_{112}(\epsilon_{12},\rho_2,\theta_1))}\right),\eqlab{reduced8}\\
\dot \rho_2 &= \rho_2^2,\nonumber\\
\dot \epsilon_{12} &=-\epsilon_{12} \left(2+\rho_2\right),\nonumber
\end{align}
after division of the right hand side by $x_2/\xi$. Notice that this quantity is positive for $\theta_1$ sufficiently small and $\epsilon_{12}>0$ sufficiently small. For $\epsilon_{12}=0$, \eqref{reduced8}, after division by $\rho_2$ on the right hand side, therefore provides a desingularized system on the center manifold, w we shall study in the following.  Within $\epsilon_{12}=\rho_2=0$, we therefore see that $\theta_1$ is decreasing and hence we put
\begin{align*}
 \gamma_{112,\textnormal{loc}}^9 = \left\{(\theta_1,\rho_2,x_2,\epsilon_{12})\vert \theta_1 \in \left[\nu, \frac{\xi(1-\alpha)}{2\alpha}\right],x_1 = \epsilon_{12}=\rho_2=0\right\},
\end{align*}
and consider the sections
\begin{align*}
\Sigma_{112}^{9,\text{in}} &= \left\{(\theta_1,\rho_2,x_2,\epsilon_{12})\vert x_2 = \delta,\, \rho_2 \in[0,\beta_1],\,\epsilon_{12} \in [0,\beta_2],\theta_1-\frac{\xi(1-\alpha)}{2\alpha}\in [-\beta_3,\beta_3]\right\},\\
\Sigma_{112}^{9,\text{out}} &= \left\{(\theta_1,\rho_2,x_2,\epsilon_{12})\vert \theta_1 = \nu,\, \rho_2 \in[0,\beta_4],\,\epsilon_{12} \in [0,\beta_2],x_2\in [-\beta_5,\beta_5]\right\},
\end{align*}
and let $\Pi_{112}^9:\Sigma_{112}^{9,\text{in}}\rightarrow \Sigma_{112}^{9,\text{out}}$ be the associated mapping obtained by the first intersection of the forward flow. We then have
\begin{lemma}\lemmalab{Pi1129}
 The mapping $\Pi_{112}^9$ is well-defined for appropriately small $\delta>0,\,\nu>0$ and $\beta_i>0$, $i=1,\ldots,5$. In particular,
 \begin{align*}
  \Pi_{112}^9(\theta_1,\rho_2,\nu,\epsilon_{12})=\left(\nu,\rho_{2+}(\rho_2,\theta_1,\epsilon_{12}),x_{2+}(\rho_2,\theta_1,\epsilon_{12}),\epsilon_{12+}(\rho_2,\theta_1,\epsilon_{12})\right),
 \end{align*}
with each $\rho_{2+},\,x_{2+},\,\epsilon_{12+}$ being smooth and satisfying
\begin{align*}
 \rho_{2+}(\rho_2,\theta_1,\epsilon_{12})& = \mathcal O(\rho_2),\\
 x_{2+}(\rho_2,\theta_1,\epsilon_{12}) &=\mathcal O(\epsilon_{12} e^{-c\rho_2^{-1}}),\\
 \epsilon_{1+}(\rho_2,\theta_1,\epsilon_{12}) &=\mathcal O(\epsilon_{12} e^{-c\rho_2^{-1}}),
\end{align*}
for $c>0$ sufficiently small. 
\end{lemma}
\begin{proof}
 We consider the reduced problem \eqref{reduced8}. Here $\epsilon_2=0,\rho_2\in [0,\beta_1]$ is an attracting center manifold with smooth foliation by stable fibers. We straighten out the fibers $(\theta_1,\rho_2,\epsilon_{12})\mapsto \tilde \theta_1 = \theta_1+ \mathcal O(\rho_1^2 \theta_1^3\epsilon_{12})$ and obtain the following
 \begin{align*}
  \dot \rho_2 &= \rho_2,\nonumber\\
\dot \theta_1 &=-\theta_1 \left(1+\frac{\rho_2\theta_1^2 F(\rho_2^{-1})}{\xi (1-\xi^{-1} \alpha \theta_1)}\right),\nonumber
 \end{align*}
after dropping the tildes. On this time scale, the mapping from $\Sigma_{112}^{9,\text{in}}$ to $\Sigma_{112}^{9,\text{out}}$ takes $\mathcal O(1)$ time. We now work our way backwards and obtain the desired result. 
\end{proof}
See illustration in \figref{gamma9}.

\begin{figure}[h!]
\begin{center}
{\includegraphics[width=.5 \textwidth]{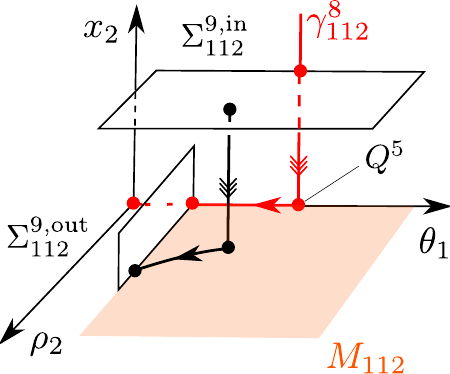}}
\end{center}
\caption{Illustration of the result in \lemmaref{Pi1129} within $\epsilon_{11}=0$. }
\figlab{gamma9}
\end{figure}
%

\subsection{Chart $(\overline z=1,\overline q= 1,\overline w_1=1,\overline \theta_1 = 1)_{1121}$}\seclab{seceqns121121}
In this chart we obtain
\begin{align}
\dot \rho_2 &=\rho_2^2 \frac{x_{21}}{\xi},\eqlab{eqeqns121121}\\
 \dot \varrho_1 &=-\varrho_1^2 \left(\rho_2 x_{21}/\xi+\varrho_1^2 \rho_2^2 \epsilon_{121}F(\rho_2^{-1})\right),\nonumber\\
 \dot x_{21} &=\rho_2 \varrho_1 \epsilon_{121}F(\rho_2^{-1})\xi - \epsilon_{121}\left(-\xi+\rho_2 \varrho_1x_{21} +\alpha \theta_1\right)-x_{21}/\xi,\nonumber\\
 \dot \epsilon_{121}&=-\epsilon_{121}\left(2x_{21}/\xi -\varrho_1^2\rho_2 \epsilon_{12}F(\rho_2^{-1})\right),\nonumber
\end{align}
from \eqref{K1BarZ1Ext} using \eqref{finalChartBarZ1K1}. 
Here $x_{21}=\epsilon_{121}=\rho_2=0,\varrho_1\in [0,\delta]$ is a line of equilibria but now the linearization gives on single non-zero eigenvalue $-1/\xi$ for any $\varrho_1\ge 0$. This provides an extension of the center manifold $M_{112}$ as follows.
\begin{proposition}
 Fix $\eta\in (0,1)$ and a  neighborhood $\mathcal U_{1121}$ of $(\varrho_1, \rho_2,\epsilon_{121})=0$ in $\mathbb R^3$ such that the following holds. There exists a locally invariant center manifold $M_{1121}$ as a graph $$x_{21} = \epsilon_{121}\xi^2 \left(1-\xi^{-1} \alpha \varrho_1+\epsilon_{121} h_{21}(\epsilon_{121},\rho_2,\varrho_1)\right),$$ over $(\varrho_1, \rho_2,\epsilon_{121})\in \mathcal U_{1121}$. Here $h_{21}$ is a smooth function.
Furthermore, there exists a smooth stable foliation with base $M_{1121}$ and one-dimensional fibers as leaves of the foliation. Within $x_{21}\in [-\delta,\delta],\,(\varrho_1, \rho_2,\epsilon_{121})\in \mathcal U_{1121}$, the contraction along any of these fibers is at least $e^{-\eta/\xi t}$.
\end{proposition}
The reduced problem on $M_{1121}$ is
\begin{align}
\dot \rho_2 &= \rho_2^2,\eqlab{reduced821}\\
\dot \varrho_1 &=-\rho_2 \varrho_1 \left(1+\frac{\rho_2\varrho_1^2 F(\rho_2^{-1})}{\xi (1-\xi^{-1} \alpha \varrho_1+\epsilon_{121} h_{21}(\epsilon_{121},\rho_2,\varrho_1))}\right),\nonumber\\
\dot \epsilon_{121} &=-\epsilon_{121} \left(2-\frac{\varrho_1^2 \rho_2 F(\rho_2^{-1})}{\xi (1-\xi^{-1} \alpha \varrho_1+\epsilon_{121} h_{21}(\epsilon_{121},\rho_2,\varrho_1))}\right).\nonumber
\end{align}
Here $\epsilon_{121}=0,\,\rho_2\in [0,\beta_1]$ is a center manifold with smooth foliation by stable fibers. We straighten out these fibers by a transformation of the form $(\rho_2,\varrho_1,\epsilon_{121})\mapsto \tilde \varrho_1 = \varrho_1+\mathcal O(\rho_2^2 \varrho_1^3\epsilon_{121})$ such that 
\begin{align}
\dot \rho_2 &= \rho_2,\eqlab{reduced821New}\\
\dot \varrho_1 &=- \varrho_1 \left(1+\frac{\rho_2\varrho_1^2 F(\rho_2^{-1})}{\xi (1-\xi^{-1} \alpha \varrho_1)}\right),\nonumber
\end{align}
after dropping the tilde, and dividing the right hand side by $\rho_2$. 
In these coordinates, $\gamma_{112,\textnormal{loc}}^9$ therefore becomes
\begin{align*}
 \gamma_{1121}^9 = \left\{(\rho_2,\varrho_1,x_{21},\epsilon_{121})\vert x_{21}= \epsilon_{121}=\rho_2=0,\,\varrho_1\in \left(0, \frac{\xi(1-\alpha)}{2\alpha}\right]\right\},
\end{align*}
upon using the flow of \eqref{reduced821New} to extend the forward orbit. It is asymptotic to $x_{21} =\epsilon_{121}=\rho_2=\varrho_1=0$ and becomes $\gamma^9$ in \eqref{gamma91} upon blowing down using \eqref{finalChartBarZ1K1}. From \eqref{reduced821New}, we have an unstable manifold
\begin{align*}
 \gamma_{1121,\textnormal{loc}}^{10} =\left\{(\rho_2,\varrho_1,x_{21},\epsilon_{121})\vert x_{21}= \epsilon_{121}=\varrho_1=0,\,\rho_2 \in [0,\nu]\right\},
\end{align*}
with $\nu>0$ sufficiently small. We therefore consider the following sections 
\begin{align*}
 \Sigma_{1121}^{10,\text{in}} &=\left\{(\rho_2,\varrho_1,x_{21},\epsilon_{121})\vert \varrho_1 = \delta,\rho_2\in [0,\beta_1],\,x_{21} \in (\beta_2],\,\epsilon_{121}\in [0,\beta_3]\right\},\\
 \Sigma_{1121}^{10,\text{out}} &=\left\{(\rho_2,\varrho_1,x_{21},\epsilon_{121})\vert \rho_2 = \nu,\,\varrho_1 \in [0,\beta_4],\rho_2\in [0,\beta_1],\,x_{21} \in [-\beta_2,\beta_2],\,\epsilon_{121}\in [0,\beta_3]\right\}.
\end{align*}
transverse to $\gamma_{1121}^9$ and $\gamma_{1121}^{10}$, respectively.
We let $\Pi_{1121}^{10}$ be the associated mapping obtained by the first intersection of the forward flow.
\begin{lemma}\lemmalab{Pi112110}
 $\Pi_{1121}^{10}$ is well-defined for appropriately small $\delta>0$, $\nu>0$ and $\beta_i>0$, $i=1,\ldots,4$. In particular, 
\begin{align*}
  \Pi_{1121}^{10}(\rho_2,x_{21},\nu,\epsilon_{121})=\left(\nu,x_{21+}(\rho_2,x_{21},\epsilon_{121}),\varrho_{1+}(\rho_2,x_{21},\epsilon_{121}),\epsilon_{121+}(\rho_2,\theta_1,\epsilon_{12})\right),
 \end{align*}
with each $x_{21+},\,\varrho_{1+},\,\epsilon_{121+}$ being $C^1$ and satisfying
\begin{align*}
 x_{21+}(\rho_2,\theta_1,\epsilon_{12}) &=\mathcal O(\epsilon_{121} e^{-c\rho_2^{-1}}),\\
 \varrho_{1+}(\rho_2,x_{21},\epsilon_{121})&=\mathcal O(\rho_2),\\
 \epsilon_{1+}(\rho_2,\theta_1,\epsilon_{12}) &=\mathcal O(\epsilon_{121} e^{-c\rho_2^{-1}}),
\end{align*}
for some $c>0$ sufficiently small.
\end{lemma}
\begin{proof}
 Similar to previous results in the manuscript, we perform a $C^1$-linearization of the $(\rho_1,\varrho_1)$-subsystem in \eqref{reduced821New}. Working backwards we then obtain the result. 
\end{proof}
Notice that the image of $\gamma_{1121}^9\cap  \Sigma_{1121}^{10,\text{in}}$ under $\Pi_{1121}^{10}$ is $\gamma_{1121,\textnormal{loc}}^{10}\cap  \Sigma_{1121}^{10,\text{out}}$, as desired.

\subsection{Chart $(\overline w=1,\overline \theta_2=1)_{21}$ }\seclab{sechere21}
In this chart, we obtain the following equations:
\begin{align}
\dot \sigma_1 &=-\sigma_1^3 \epsilon_1 F(z_2),\eqlab{eqhere21}\\
\dot z_2 &=-e^{-2z_2} x_1/\xi,\nonumber\\
\dot x_1&=\sigma_1 \epsilon_1 \xi-\epsilon_1 \left(-\xi +\sigma_1 x_1+\alpha \sigma_1 z_2\right)\nonumber\\
&- e^{-2z_2} x_1/\xi+\sigma_1^2\epsilon_1 x_1F(z_2),\nonumber\\
\dot \epsilon_1 &=\sigma_1^2 \epsilon_1^2 F(z_2)\nonumber
\end{align}
from \eqref{K1Barw1} using \eqref{Psi211}. 
Let $I=[-c_1,c_1]\subset \mathbb R$ be a fixed, large interval. Then there is a sufficiently small neighborhood $\mathcal U_{21}$ of $(0,0)$ in $\mathbb R^2$ such that there exists a center manifold $M_{21}$ as a graph
\begin{align*}
 x_1 =\epsilon_1 e^{2z_2} \xi^2 (1+\sigma_1 F(z_2)-\xi^{-1} \alpha \sigma_1 z_2+ \epsilon_1 h_{21}(\sigma_1,z_2,\epsilon_1)),
\end{align*}
over $(z_2,\sigma_1,\epsilon_1)\in I\times \mathcal U_{21}$. This is an extension of the slow manifold into this chart. On this center manifold, we obtain the following reduced problem 
\begin{align}
 \dot \sigma_1 &=-\frac{\sigma_1^3 e^{-2z_2} F(z_2)}{\xi (1+\sigma_1-\xi^{-1} \alpha \sigma_1 z_2 + \epsilon_1 h_{21}(\sigma_1,z_2,\epsilon_1))},\eqlab{reducedM3}\\
 \dot z_2 &=-e^{-2z_2},\nonumber\\
 \dot \epsilon_1 &=\frac{\sigma_1^2e^{-2z_2}  \epsilon_1 F(z_2)}{\xi (1+\sigma_1-\xi^{-1} \alpha \sigma_1 z_2 + \epsilon_1 h_{21}(\sigma_1,z_2,\epsilon_1))}.\nonumber
\end{align}
Clearly, $z_2$ is decreasing. We then get a mapping from $\{z_2=\nu^{-1}\}$ to $\{z_2=- \nu^{-1}\}$ using regular perturbation theory. In particular, we notice that by using \eqref{ccBarWE1}, $\gamma_{1121}^{10}$ becomes
\begin{align}
 \gamma_{21}^{10} = \left\{(\sigma_1,z_2,x_1,\epsilon_1)\vert \sigma_1=x_1=\epsilon_1=0,\,z_2\in \mathbb R\right\},\eqlab{gamma2110}
\end{align}
upon extension by the forward flow of \eqref{reducedM3}.
\subsection{Chart $(\overline z=-1,\overline \theta=1,\overline w_3=1)_{311}$}\seclab{sechere311}
In this chart, we obtain the following
\begin{align}
 \dot x_{11} &=-x_{11}/\xi+\epsilon_{11}\mu_1^2 q\left(-\xi \pi_1 (1-\mu_1q)-q(-\xi +\pi_1 \mu_1 x_{11}-\alpha \pi_1)\right),\eqlab{eqnsBlowupK1}\\
 \dot \pi_1 &= \pi_1 \mu_1 \left(\pi_1^2 \epsilon_{11}\mu_1^2 q(1-\mu_1 q)+x_{11}/\xi\right),\nonumber\\
 \dot \mu_1 &=-\mu_1^2 x_{11}/\xi,\nonumber\\
 \dot q &=-qx_{11}(1-\mu_1)/\xi,\nonumber\\
 \dot \epsilon_{11}&=-\pi_1^2 \mu_1^3 q \epsilon_{11}^2 (1-\mu_1 q).\nonumber
\end{align}
Recall that 
\begin{align}
q=\mu_1^{-1} e^{-\mu_1^{-1}},\eqlab{q3New}
\end{align}
see \eqref{q3} and \eqref{Psi31112}, but we shall use this only when required. Then $\gamma_{21}^{10}$ becomes 
\begin{align*}
 \gamma_{311}^{10} =\left\{(x_{11},\pi_1,\mu_1,\epsilon_{11})\vert \pi_1=x_{11}=\epsilon_{11}=0,\,\mu_1>0\right\},
\end{align*}
in this chart,
using \eqref{gamma2110} and the coordinate change described by \eqref{finalCC}.
Now, $x_{11}=0,\,\pi_1 = 0,\epsilon_{11}=0, \mu_1=0$ is an equilibrium. The linearization has $-1/\xi$ as a single non-zero eigenvalue. Therefore there exists a small neighborhood $\mathcal U_{311}$ of $(\pi_1,\mu_1,q,\epsilon_{11})=0$ in $\mathbb R^4$ such that there exists a center manifold $M_{311}$ as a graph
\begin{align}
 x_{11} =\xi \epsilon_{11} \mu_1^2 qH_{311}(\pi_1,\mu_1, q,\epsilon_{11}) ,\eqlab{M1ManifoldK1}
\end{align}
over $(\pi_1,\mu_1,q,\epsilon_{11})\in \mathcal U_{311}$. Here
 \begin{align}
 H_{311}(\pi_1,\mu_1, q,\epsilon_{11})=-\pi_1 \xi (1-\mu_1 q) + q(\xi+\alpha \pi_1)+\epsilon_{11}\mu_1^2 q h_{311}(\pi_1,\mu_1,q,\epsilon_{11}),\eqlab{H311Function}
 \end{align}
 with $h_{311}$ smooth. 
 On this center manifold we obtain the following reduced problem
\begin{align}
 \dot \mu_1 &=-\mu_1^2  H_{311}(\pi_1,\mu_1,q,\epsilon_{11}),\eqlab{ReducedM1ManifoldK1}\\
 \dot q &=-q H_{311}(\pi_1,\mu_1,q,\epsilon_{11})(1-\mu_1),\nonumber\\
 \dot \pi_1 &=\pi_1 \mu_1 \left(\pi_1^2(1-\mu_1 q)+H_{311}(\pi_1,\mu_1,q,\epsilon_{11}) \right),\nonumber\\
 \dot \epsilon_{11}&=-\pi_1^2 \mu_1\epsilon_{11}(1-\mu_1 q),\nonumber
 \end{align}
 after division on the right hand side by $\epsilon_{11}\mu_1^2 q$. 
 For this system, $\mu_1=q=\pi_1=\epsilon_{11}=0$ is fully nonhyperbolic. We therefore apply a subsequent blowup transformation, setting 
\begin{align}
 \pi_1 &= q \pi_{11}.\eqlab{finalBlowupK1}
\end{align}
This gives
\begin{align*}
H_{311}(q\pi_{11},\mu_1, q,\epsilon_{11}) = q \tilde H_{311}(\pi_{11},\mu_1,q,\epsilon_{11}),
\end{align*}
with 
\begin{align*}
 \tilde H_{311}(\pi_{11},\mu_1,q,\epsilon_{11}) = -\pi_{11} \xi (1-\mu_1 q) + \xi+\alpha q\pi_{11}+\epsilon_{11}\mu_1^2 h_{311}(q\pi_{11},\mu_1,q,\epsilon_{11}).
\end{align*}
See \eqref{H311Function}. 
Therefore
\begin{align}
 \dot \mu_1 &=-\mu_1^2 \tilde H_{311}(\pi_{11},\mu_1,q,\epsilon_{111}) ,\eqlab{ReducedM1ManifoldK1Pi11}\\
 \dot q &=- q \tilde H_{311}(\pi_{11},\mu_1,q,\epsilon_{111})(1-\mu_1),\nonumber\\
 \dot \pi_{11}&=\pi_{11}\left(\tilde H_{311}(\pi_{11},\mu_1,q,\epsilon_{11})+\pi_1^2 q\mu_1(1-\mu_1 q)\right),\nonumber\\
 \dot{\epsilon}_{11}&=-\pi_{11}^2 \mu_1 q\epsilon_{11}(1-\mu_1 q),\nonumber
\end{align}
after division of the right hand side by $q$. Now, we have gained hyperbolicity. In particular, $\pi_{11}=\mu_1 =q=\epsilon_{11}=0$ is partially hyperbolic, the linearization having a single non-zero eigenvalue $\xi>0$ with corresponding unstable eigenspace along the invariant $\pi_{11}$-axis. Also, $\mu_1 =q=\epsilon_{11}=0,\pi_{11}= 1$ is a partially hyperbolic equilibrium and therefore we have the following by standard center manifold theory.
\begin{lemma}\lemmalab{Wcu311Q6}
There exists a center manifold $K_{311}$ as a graph
\begin{align}
 \pi_{11} = G_{311}(\mu_1,q,\epsilon_{11}),\eqlab{M11ManifoldK1}
\end{align}
over $(\mu_1,q,\epsilon_{11})\in \mathcal V_{311}$, where $\mathcal V_{311}$ is a small neighborhood of $(0,0,0)$ in $\mathbb R^3$. Here
\begin{align*}
G_{311}(\mu_1,q,\epsilon_{11})=1+\epsilon_{11} \mu_1^2 h_{311}(0,\mu_1,0,\epsilon_{11})+ q\left(\frac{\alpha}{\xi}+\frac{\xi+1}{\xi}\mu_1 +\mathcal O(\mu_1^2\epsilon_{11}+\mu_1^2+q)\right),
\end{align*}
is smooth. $h_{311}$ is the smooth function in \eqref{H311Function}.  

The submanifold of $K_{311}$ within the invariant subset $\{\epsilon_{11}=0,q=\mu_1^{-1} e^{-\mu_1^{-1}}\}$, recall \eqref{q3New}, is a unique center manifold $W_{311}^{cu}$. In particular, its image under the coordinate transformation $(\mu_{11},\pi_1)\mapsto (y,z)$ defined by \eqref{finalBlowupK1}, \eqref{Psi31112}, \eqref{theta3Eqn} and \eqref{phi13xyz} produce $W^{cu}(Q^6)$  with the asymptotics in \lemmaref{Wcsu} for $y\gg 1$.  
\end{lemma}
See \figref{gamma10}. 
\begin{figure}[h!]
\begin{center}
{\includegraphics[width=.995\textwidth]{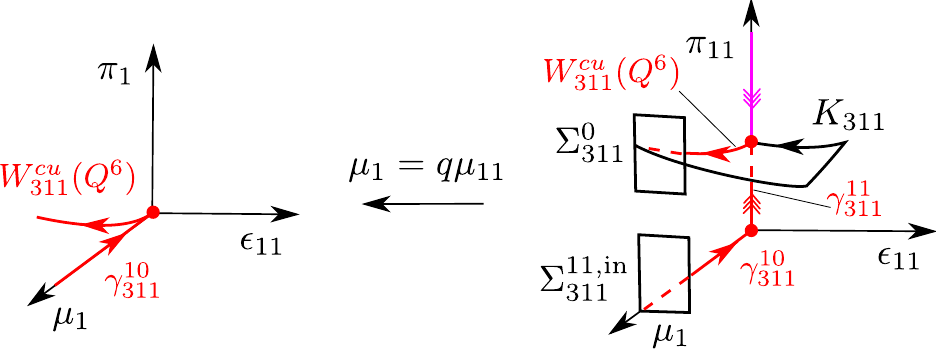}}
\end{center}
\caption{Illustration of the result in \lemmaref{Wcu311Q6}. Using the blowup \eqref{finalBlowupK1} we gain hyperbolicity and obtain an attracting center manifold $K_{311}$, its $\epsilon_1=0$-limit is the unique center manifold $W^{cu}(Q^6)$ in \lemmaref{Wcsu}. The purple orbit is relevant for $\alpha>1$. }
\figlab{gamma10}
\end{figure}
Notice that the invariant graph \eqref{M11ManifoldK1} passes through the set of equilibria given as the graph $\pi_{11}=1+\epsilon_{11}\mu_1^2 h_{311}(0,\mu_1,0,\epsilon_{11})$ over $(\mu_1,\epsilon_{11})$ within $q=0$. On the center manifold \eqref{M11ManifoldK1}, we have
\begin{align*}
 \tilde H_{311}(\pi_{11},\mu_1,q,\epsilon_{11}) = - q\left(\mu_1+\mathcal O(\mu_1^2 \epsilon_{11}+\mu_1^2 + q)\right).
\end{align*}
Therefore, upon returning to the variables $(\pi_{11},\mu_1,\epsilon_{11})$ and the set defined by $q= \mu_1^{-1} e^{-\mu_1^{-1}}$, recall \eqref{q3New}, we have
\begin{align*}
 \tilde H_{311}(\pi_{11},\mu_1,q(\mu_1),\epsilon_{11}) = -\mu_1 q(\mu_1) \left(1+\mathcal O(\mu_1)\right),
\end{align*}
and hence obtain the following reduced problem on the center manifold 
\begin{align}
 \dot \mu_1 &=\mu_1^2 \left(1+\mathcal O(\mu_1)\right),\eqlab{ReducedM11ManifoldK1}\\
 \dot \epsilon_{11} &=-\epsilon_{11} \left(1+\mathcal O(\epsilon_{11}+\mu_1)\right),\nonumber
\end{align}
after division by $\mu_1 q$ on the right hand side. 

Consider the following sections
\begin{align*}
 \Sigma_{311}^{11,\text{\text{in}}} &= \{(x_{11},\pi_{11},\mu_1,\epsilon_{11})\vert \mu_1 = \nu,\,x_{11}\in [-\beta_1,\beta_1],\,\pi_{11} \in [0,\beta_2],\,\epsilon_{11}\in [0,\beta_3]\},\\
 \Sigma_{311}^0 &=\{(x_{11},\pi_{11},\mu_1,\epsilon_{11})\vert \mu_1 = \delta,\,x_{11}\in [-\beta_1,\beta_1],\,\pi_{11}-1 \in [-\beta_4,\beta_4],\,\epsilon_{11}\in [0,\beta_3]\},
\end{align*}
and consider the associated mapping $\Pi_{311}^{11}:\Sigma_{311}^{0}\rightarrow \Sigma_{311}^{11,\text{\text{out}}}$. Notice that $\mu_1=\delta$ in $\Sigma_{311}^{0}$ becomes $y=1/\delta$ in the original variables using \eqref{Psi31112} and \eqref{phi13xyz}, in agreement with $\Sigma^0$, see \eqref{Sigma0}. 
Setting 
\begin{align*}
 \gamma_{311}^{11} = \{(x_{11},\pi_{11},\mu_1,\epsilon_{11})\vert x_{11} = 0,\,\pi_{11}\in [0,1),\,\mu_1=0,\epsilon_{11}=0\},
\end{align*}
we can then describe $\Pi_{311}^{11}$ by following $\gamma_{311}^{10}$, $\gamma_{311}^{11}$ and $W_{311}^{cu}(Q^6)$. 
\begin{lemma}\lemmalab{Pi31111}
 $\Pi_{311}^{11}$ is well-defined for appropriately small $\delta>0$, $\nu>0$ and $\beta_i>0$, $i=1,\ldots,4$. In particular,
\begin{align*}
 \Pi_{311}^{11}(x_{11},\pi_{11},\nu,\epsilon_{11}) = (x_{11+}(x_{11},\pi_{11},\epsilon_{11}),\pi_{11+}(x_{11},\pi_{11},\epsilon_{11}),\delta,\epsilon_{11+}(x_{11},\pi_{11},\epsilon_{11})),
\end{align*}
with each coordinate function being $C^1$. In particular, these functions satisfy the following equalities
\begin{align*}
 x_{11+}(x_{11},\pi_{11},\epsilon_{11}) &= \xi \epsilon_{11+}(x_{11},\pi_{11},\epsilon_{11}) \delta e^{-\delta^{-1}} H_{311}(\pi_{11+}(x_{11},\pi_{11},\epsilon_{11}),\delta, \delta^{-1}e^{-\delta^{-1}},\epsilon_{11+}) \\
 &+\mathcal O(e^{-c\epsilon_{11}^{-1}e^{c/\log \pi_{11}^{-1}}}),\\
 \pi_{11+}(x_{11},\pi_{11},\epsilon_{11}) &=G_{311}(\epsilon_{11+}(x_{11},\pi_{11},\epsilon_{11}),\delta,\delta^{-1} e^{-\delta^{-1}})+\mathcal O(e^{-ce^{c/\log \pi_{10}^{-1}}}),\\
 \epsilon_{11+}(x_{11},\pi_{11},\epsilon_{11})&=e^{\delta^{-1}-\nu^{-1}}\pi_{11} \epsilon_{11}(1+\mathcal O(\epsilon_{11}+\delta e^{-\delta_1})). 
\end{align*}
for $c>0$ sufficiently small. 
\end{lemma}
\begin{proof}
We consider the following system
\begin{align*}
 \dot x_{11} &=-x_{11}/\xi+\epsilon_{11}\mu_1^2 q\left(-\xi \pi_1 (1-\mu_1q)-q(-\xi +\pi_1 \mu_1 x_{11}-\alpha \pi_1)\right),\\
 \dot \pi_{11} &= \pi_{11} \left(\pi_{11}^2 \epsilon_{11}\mu_1^3 q^2(1-\mu_1 q)+x_{11}/\xi\right),\\
 \dot \mu_1 &=-\mu_1^2 x_{11}/\xi,\\
 \dot \epsilon_{11}&=-\pi_{11}^2 \mu_1^3 q^3 \epsilon_{11}^2 (1-\mu_1 q),
\end{align*}
with $q(\mu_1) = \mu_1^{-1} e^{-\mu_1^{-1}}$, 
obtained by substituting \eqref{finalBlowupK1} into \eqref{eqnsBlowupK1}. First, we straighten out the stable fibers of the center manifold \eqref{M1ManifoldK1} through a transformation of the form $(x_{11},\pi_{11},\mu_1,\epsilon_{11})\mapsto (\tilde \pi_{11},\tilde \mu_1) = (\pi_{11}(1+\mathcal O(x_{11})),\mu_1(1+\mathcal O(\mu_1 x_{11})))$. Dropping the tildes we then obtain \eqref{ReducedM1ManifoldK1Pi11} after division by $\epsilon_{11}\mu_1^2q(\mu_1)^2=\epsilon_{11}e^{-2\mu_1^{-1}}$ on the right hand side. We further divide the right hand side by 
$$\tilde H_{311}(\pi_{11},\mu_1,q,\epsilon_{11})+\pi_1^2 q\mu_1(1-\mu_1 q)\approx \xi,$$ such that 
\begin{align*}
  \dot \mu_1 &=-\mu_1^2 \frac{\tilde H_{311}(\pi_{11},\mu_1,q,\epsilon_{111})}{\tilde H_{311}(\pi_{11},\mu_1,q,\epsilon_{11})+\pi_1^2 q\mu_1(1-\mu_1 q)} ,\nonumber\\
 \dot \pi_{11}&=\pi_{11},\nonumber\\
 \dot{\epsilon}_{11}&=-\frac{\pi_{11}^2 \mu_1 q\epsilon_{11}(1-\mu_1 q)}{\tilde H_{311}(\pi_{11},\mu_1,q,\epsilon_{11})+\pi_1^2 q\mu_1(1-\mu_1 q)},\nonumber
\end{align*}
We then straighten out the unstable fibers of the local invariant manifold $\pi_{11}=0$ by a transformation of the form $(\pi_{11},\mu_1,\epsilon_{11})\mapsto (\tilde \mu_1,\tilde \epsilon_{11})=(\mu_1(1+\mathcal O(\mu_1 \pi_{11})),\epsilon_{11}(1+\mathcal O(e^{-\mu_1^{-1}} \pi_{11}^2))$ such that
 \begin{align*}
  \dot \pi_{11} &=\pi_{11},\\
  \dot \mu_1 &=-\mu_1^2,\\
  \dot \epsilon_{11}&=0,
 \end{align*}
upon dropping the tildes. Now, we integrate these equations from $\mu_1(0)=\nu$ to $\pi_{11}(T)=\delta$ using $\pi_{11}(0)\le \beta_2 \ll \delta$. This gives
\begin{align*}
 \mu_1(T) = \frac{\nu}{1+\nu T},
\end{align*}
for $T=\log (\pi_{11}(0)^{-1} \delta)$. Working our way backwards, we realise that the contraction along the stable fibers of the center manifold \eqref{M1ManifoldK1}, during this transition, is at least $\mathcal O(e^{-c/(\epsilon_{11}(0)\pi_{11}(0))})$ for some $c>0$ sufficiently small. 

Subsequently, from \eqref{ReducedM1ManifoldK1Pi11}, we then apply a finite time flow map up close to $\pi_{11}=1-\nu$. From here, we then straighten out the center manifold by a transformation of the form
\begin{align*}
 \pi_{11} = 1+\epsilon_{11} \mu_1^2 h_{311}(0,\mu_1,0,\epsilon_{11})+ q\left(\frac{\alpha}{\xi}+\frac{\xi+1}{\xi}\mu_1 +\mathcal O(\mu_1^2\epsilon_{11}+\mu_1^2+q)\right)+\tilde \pi_{11}.
\end{align*}
This gives
\begin{align*}
 \dot{\pi}_{11}&=-\xi \pi_{11},\\
 \dot \mu_1 &=\mu_1^2 \left(\mu_1 q +\mathcal O(\pi_{11}+\mu_1^2 q)\right),\\
 \dot \epsilon_{11} &= -(1+\mathcal O(\epsilon_{11} \mu_1^2 +q+\pi_{11}))\mu_1 q\epsilon_{11}.
\end{align*}
after a transformation of time and dropping the tilde. Now, we straighten out the stable fibers by a transformation of the form $(\pi_{11},\mu_1,\epsilon_{11})\mapsto (\tilde \mu_1,\tilde \epsilon_{11}) = (\mu_1(1+\mathcal O(\mu_1\pi_{11})),\epsilon_{11}(1+\mathcal O(e^{-\mu_1^{-1}} \pi_{11})))$. This gives
\begin{align*}
 \dot{\pi}_{11}&=-\xi \pi_{11},\\
 \dot \mu_1 &=\mu_1^3 q \left(1 +\mathcal O(\mu_1)\right),\\
 \dot \epsilon_{11} &= -(1+\mathcal O(\mu_1))\mu_1 q\epsilon_{11},
\end{align*}
upon dropping the tildes. $\pi_{11}$ decouples from this system. We therefore consider the $(\mu_1,\epsilon_{11})$ system. Dividing the right hand side by $\mu_1 q (1+\mathcal O(\mu_1^2))$, and applying a transformation of the form $(\epsilon_{11},\mu_1)\mapsto \tilde \mu_1 = \mu_1 (1+\mathcal O(\mu_1))$ gives 
\begin{align*}
\dot \mu_1 &=\mu_1^2 ,\\
 \dot \epsilon_{11} &= -\epsilon_{11}
 \end{align*}
upon dropping the tildes. We then integrate this system from $\mu_1(0)=\nu_{10}$ to $\mu_1(T)=\delta$ taking $\nu_{10}\ll \delta$. This gives
\begin{align*}
 \epsilon_{11}(T) = e^{-T}\epsilon_{11}(0),
\end{align*}
with $T=\frac{1}{\nu_{10}}\left(1-\nu_{10}/\delta\right)$. Therefore 
\begin{align*}
 \epsilon_{11}(T) = e^{-\nu_{10}^{-1}\left(1-\nu_{10}/\delta\right)} \epsilon_{11}(0). 
\end{align*}
Working our way backwards, we realise that the contracting along the stable fibers of the center manifold \eqref{M11ManifoldK1} under this transition is at least $\mathcal O(e^{-c e^{1/(2\nu_{10})}})$. Similarly, the contraction along the stable fibers of the center manifold $K_{311}$, see \eqref{M1ManifoldK1}, is at least $\mathcal O(e^{-c\epsilon_{10}^{-1} e^{2\nu_{10}^{-1}}})$. Both constants $c$ here are sufficiently small. Now, we combine these estimates to obtain the desired result. In particular, the expression for $\epsilon_{11+}$ follows from the conservation of $\epsilon=e^{-\mu_1} \pi_{11} \epsilon_{11}$. Therefore
\begin{align*}
 \epsilon_{11+} = e^{\delta^{-1}-\nu^{-1}} \pi_{11+}^{-1} \pi_{11}\epsilon_{11}.
\end{align*}
Here $\pi_{11+} = 1+\mathcal O(\epsilon_{11+}\mu_1^2+\mu_1e^{-\mu_1})$. We therefore solve this equation for $\epsilon_{11+}$ using the implicit function theorem.

\end{proof}
\section{Discussion}\seclab{Discussion}
To prove \thmref{mainThm}, we applied the method in \cite{kristiansen2017a}, developed by the present author, to gain hyperbolicity where this is lost due to exponential decay.  
In \cite{kristiansen2017a}, this method is mainly used on toy examples and the present analysis therefore provides the most important application of this method to obtain rigorous result in singular perturbed systems with this special loss of hyperbolicity. In ongoing work, I use similar methods to show a similar result to \thmref{mainThm} for the spring-block model with the Dietrich friction law:
\begin{align*}
 \dot x&=(1+\alpha)(e^{-x/(1+\alpha)}-e^z),\\
 \dot y &=e^z-1,\\
 \epsilon \dot z &=-e^{-z} \left(y+\frac{x+z}{\xi}\right).
\end{align*}

From the results in the present paper we deduce the following interesting consequences: Firstly, in \cite{erickson2008a} chaos is observed in numerical computations of \eqref{system} through a period doubling cascade of the relaxation oscillation studied in the present manuscript. A corollary of our results is that this chaos is an $\epsilon=\mathcal O(1)$ phenomenon. It is not persistent as $\epsilon\rightarrow 0$. 

Secondly: Let $(x(t),y(t),z(t))\in \Gamma_\epsilon$. Then for $\alpha\le 1$, $z(t)$ attains its minimum close to $W^{cu}(Q^6)$, see \eqref{WcuA} with $y\gg 1$. On the other hand, for $\alpha>1$, the minimum of $z(t)$ occurs at a smaller value, near the line $z=\frac{\xi(\alpha-1)}{2\alpha} y$ on $C$ with $y\gg 1$. This follows from \eqref{gamma91} and the statements proceeding it, see also \appref{alphaGE1}. There is therefore a transition in how the minimum of $z(t)$ depends upon $\epsilon$ (and $\alpha$) when $\alpha$ crosses $\alpha=1$. We illustrate this further in the bifurcation diagram in \figref{bif} using $\text{min}\,z$ as a measure of the amplitude for $\xi=0.5$ and three different values of $\epsilon$: $\epsilon=0.01$ (full line), $\epsilon=0.001$ (dotted line), $\epsilon=0.0001$ (dash-dotted line). These diagrams were computed using AUTO. In this diagram, we see that the limit cycles are born in Hopf bifurcations near $\alpha=0.5$. The amplitudes increase rapidly due to the underlying Hamiltonian structure, recall \eqref{yzHam}, see also \cite{bossolini2017a}. Subsequently they flatten out. This is where the connection to the relaxation oscillations, described in \thmref{mainThm}, occurs. Between $\alpha\approx 0.5 - 1.1$ the increase in amplitude is more moderate, like $\text{min}\,z\sim \log \alpha$. Examples of limit cycles are shown in \figref{xyz}. Beyond this interval of $\alpha$-values, the amplitudes increase linearly in $\alpha$: $\text{min}\,z\sim \alpha$. 

\begin{figure}[h!]
\begin{center}
{\includegraphics[width=.75\textwidth]{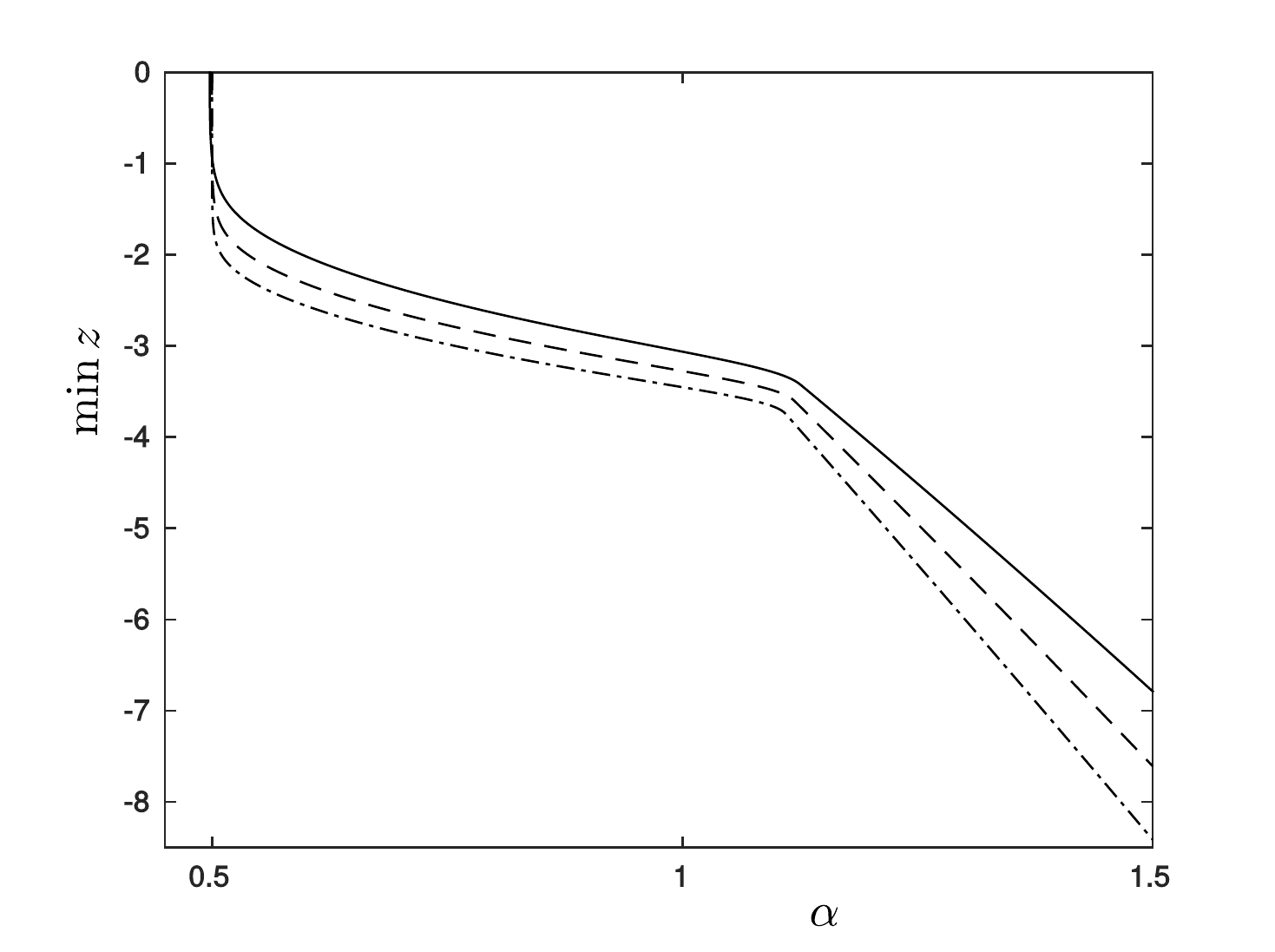}}
\end{center}
\caption{Bifurcation diagram of limit cycles for $\xi=0.5$ and three different values of $\epsilon$: $\epsilon=0.01$ (full line), $\epsilon=0.001$ (dotted line), $\epsilon=0.0001$ (dash-dotted line), using $\text{min}\,z$ as a measure of the amplitude. Around $\alpha\approx 1$, we see a dramatic change  in $\text{min}\,z$. This transition is captured by the method of the paper. }
\figlab{bif}
\end{figure}

In \cite{bossolini2017a} it was conjectured that the relaxation oscillations and the local limit cycles near the Hopf bifurcation belong to the same family of stable limit cycles for all $0<\epsilon\ll 1$, as exemplified in \figref{bif} for particular values of $\epsilon$. I believe that this result can be proven using the methods in the present paper, but it requires a detailed description near $Q^3$ where the transition from small to big oscillations occur. I have not yet pursued such an analysis. On a related matter, we highlight that, as a consequence of our approach, $\Gamma_\epsilon$ attracts a large set of initial conditions. In fact, $\Gamma_\epsilon$ attracts all initial conditions in $K\backslash U$ where $U$ is a small neighborhood of $K\cap W^{cs}(Q^3)$, recall \thmref{mainThm}. A detailed analysis near $Q^3$ may in fact reveal that $\Gamma_\epsilon$ attracts all points in $K\backslash \{0\}$. 
\bibliography{refs}
\bibliographystyle{plain}
\clearpage 
\appendix

\section{Case $\alpha\ge 1$} \applab{alphaGE1}
First, we describe $\alpha>1$. Since the details in chart $\phi_3$, in particular the proof of \lemmaref{Pi17} is unchanged, we will work in the chart $\phi_1$ only. We then consider the chart $(\bar w=1,\bar q=1)_{11}$, recall \eqref{Psi111app}. This gives the following equations:
\begin{align*}
\dot x &=- \epsilon_1 \theta_1w_1\left(\theta_1w_1 x F(w_1^{-1})+(x+(1+\alpha)\theta_1)\right),\nonumber\\
\dot \theta_1 &=-\theta_1 w_1 \left(\epsilon_1 \theta_1^2w_1 F(w_1^{-1})+\left(1+\frac{x+\theta_1}{\xi}\right)\right),\nonumber\\
\dot w_1 &=w_1^2 \left(1+\frac{x+\theta_1}{\xi}\right),\nonumber\\
\dot \epsilon_1 &=-2\epsilon_1 \left(1+\frac{x+\theta_1}{\xi}\right).
\end{align*}
$r_1$ decouples as usual and shall therefore be ignored. In this chart $\gamma^7$ becomes
\begin{align*}
 \gamma_{11}^7 = \left\{(x,\theta_1,w_1,\epsilon_1)\vert x=-\frac{\xi}{2\alpha}(1+\alpha),\,\theta_1\in \left(0,\frac{\xi}{2\alpha}\right],\,w_1=\epsilon_1=0\right\},
\end{align*}
for $\alpha>1$, see e.g. \eqref{gamma17loc}. It is contained within the invariant manifold $\epsilon_1=0$. By desingularization through division by $w_1$ within this set, we obtain 
\begin{align*}
 \dot x &=0,\\
 \dot \theta_1 &= -\theta_1 \left(\epsilon_1 \theta_1^2w_1 F(w_1^{-1})+\left(1+\frac{x+\theta_1}{\xi}\right)\right),\\
 \dot w_1 &=w_1 \left(1+\frac{x+\theta_1}{\xi}\right).
\end{align*}
The $x$-axis is therefore a line of equilibria. $\gamma_{11}^7$ is asymptotic to $(x,\theta_1,w_1) = (-\xi(1+\alpha)/(2\alpha),0,0)$ within this set, following the associated stable manifold. Notice here that 
\begin{align}
  \left(1+\frac{x+\theta_1}{\xi}\right) = \frac{\alpha-1}{2\alpha}>0,\eqlab{hallohere}
\end{align}
for $x=-\xi(1+\alpha)/(2\alpha)$ and $\theta_1=0$, by assumption. As usual, we can track a small neighborhood of $\gamma_{11}^7$ near $\theta_1=\text{const}.>0$ up to $w_1=\text{const}.>0$ in a $C^1$-fashion by following the unstable manifold 
\begin{align*}
 \gamma_{11}^8 = \left\{(x,\theta_1,w_1,\epsilon_1)\vert x=-\frac{\xi}{2\alpha}(1+\alpha),\,\theta_1=0,\,w_1\ge 0,\,\epsilon_1=0\right\}.
\end{align*}
In fact, the result is almost identical to \lemmaref{Pi1118}. We therefore skip the details. 

Next, recall that $\epsilon=e^{-2w_1^{-1}}\epsilon_1$, see \eqref{q1app}, so at $w_1=\text{const}.$ we have $\epsilon_1\sim \epsilon$. We can therefore transform the result in $(\bar z=1,\bar q=1)_{11}$ into the chart $(\bar w=1)_2$, see \eqref{K1Barw1}, using that $z_2=-w_1^{-1}$. The system is a regular perturbation problem in this $(\bar w=1)_2$-chart. In particular, along 
\begin{align*}
 \gamma_2^8 = \left\{(x,\theta_2,z_2,\epsilon)\vert x=-\frac{\xi}{2\alpha}(1+\alpha),\,\theta_2=0,\,z_2\in \mathbb R,\,\epsilon=0\right\},
\end{align*}
$z_2$ is decreasing. This brings us into the chart $(\bar z=-1)_3$ using the coordinate transformation $w_3=z_2^{-1}$. The equations in this chart are given in \eqref{K3BarZN1}. The $x$-axis is again a line of equilibria for this system and $\gamma_3^8$ is asymptotic to the point $(x,\theta_3,w_3,\epsilon) = (-\xi(1+\alpha)/(2\alpha),0,0,0)$ within this line by following the associated stable manifold. We can (again) track a small neighborhood of $\gamma_{3}^8$ near $w_3=\text{const}.>0$ up to $\theta_3=\text{const}.>0$ in a $C^1$-fashion by following the unstable manifold
\begin{align*}
 \gamma_3^9 =\left\{(x,\theta_3,w_3,\epsilon)\vert x=-\frac{\xi}{2\alpha}(1+\alpha),\,\theta_3=\left[0,\frac{\xi(\alpha-1)}{2\alpha}\right),\,w_3=0,\,\epsilon=0\right\}.
\end{align*}
The result is almost identical to \lemmaref{Pi1118}. We skip the details again.
Here $\gamma_3^9$ is asymptotic to a point $(x,\theta_3,w_3) = (-\xi(1+\alpha)/(2\alpha),{\xi}(\alpha-1)/({2\alpha}),0)$ on $C_\infty:\,x=-\xi-\theta_3,\,w_3=0$, which is normally hyperbolic in this chart. Following \lemmaref{Wcsu}, see also \cite{bossolini2017a} and \figref{poincareReduced}(c), we obtain an orbit $\gamma_{3}^{10}$ of the slow flow on $C_\infty$ 
\begin{align*}
 \gamma_3^{10} = \left\{(x,\theta_3,w_3,\epsilon)\vert x=-\xi+\theta_3, \theta_3 \in \left(0,\frac{\xi(\alpha-1)}{2\alpha}\right],\,w_3=0,\,\epsilon=0\right\},
\end{align*}
along which $\theta_3$ is decreasing. Notice in particular, that for $\alpha>1$ the point $(x,\theta_3,w_3) = (-\xi(1+\alpha)/(2\alpha),{\xi}/{2\alpha}(\alpha-1),0)$ is always contained  between $Q^6$ and the unstable node $Q^7$, the latter having coordinates $x=0,\,\theta_3=\xi,\,w_3=0$ in this chart. $\gamma_3^{10}$ therefore brings us into the chart $(\overline z=-1,\overline \theta=1,\overline w_3=1)_{311}$ where the analysis in \secref{sechere311} is valid. See \figref{gamma10} where $\gamma_3^{10}$ is shown in purple. This completes the (sketch of) proof for $\alpha>1$. We illustrate the singular segments in \figref{gammaK1New}.

\begin{figure}[h!]
\begin{center}
{\includegraphics[width=.495\textwidth]{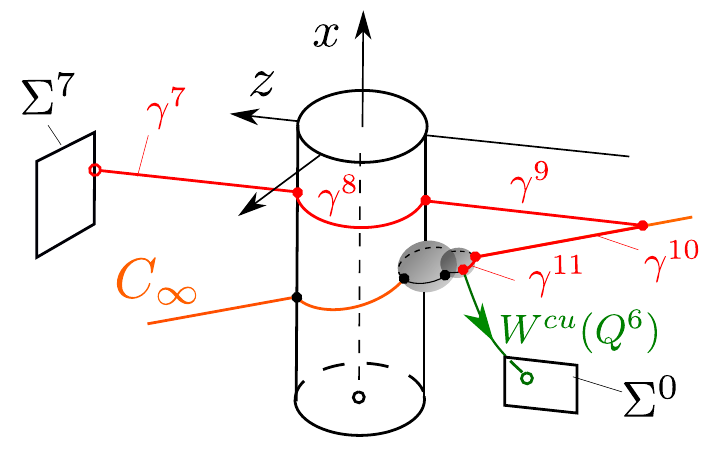}}
\end{center}
\caption{Singular orbit segments of the blowup in $\phi_1$ for $\alpha>1$.} 
\figlab{gammaK1New}
\end{figure}

Up until now our approach is not uniform in $\alpha$. For $\alpha>1$ for example, the approach breaks down at $\alpha=1$ since the condition in \eqref{hallohere} is violated (the bracket vanishes). To capture this, we may follow the approach in \secref{seceqns12111}, see \figref{gamma8}. Here both $\alpha<1$ and $\alpha>1$ are visible (red and purple in \figref{gamma8}). However, the $w_{11}$-axis in \figref{gamma8} is degenerate. To obtain results uniform in $\alpha$ we therefore blowup this axis by introducing polar coordinates in the $(\rho_1,\theta_1)$-plane. The details are pretty standard so we also leave these out of the manuscript for simplicity. 

 \end{document}